\documentclass[preprint,compress,12pt]{elsarticle}

\usepackage{graphicx}
\usepackage{pifont,latexsym,ifthen,amsthm,rotating,calc,textcase,booktabs}
\usepackage{amsfonts,amssymb,amsbsy,amsmath}
\usepackage{mathrsfs}

\newtheorem{theorem}{Theorem}[section]
\newtheorem{lemma}[theorem]{Lemma}
\newtheorem{corollary}[theorem]{Corollary}
\newtheorem{remark}[theorem]{Remark}
\newtheorem{proposition}[theorem]{Proposition}
\newtheorem{definition}[theorem]{Definition}

\usepackage{natbib}

\numberwithin{equation}{section}
\usepackage{diagbox}
\usepackage{todonotes}
\usepackage{float}
\usepackage{subfigure}
\usepackage{hyperref}
\usepackage{enumitem}
\usepackage{geometry}
\usepackage{setspace}

\hypersetup{
	colorlinks=true, 
	linktoc=all,     
	linkcolor=blue,  
	citecolor=blue
}

\begin{document}
	\journal{ {J. Math. Pure Appl.}}

	\begin{frontmatter}
		\title{Local and global well-posedness for the  nonlinear Schr\"{o}dinger equation  with nonhomogeneous boundary conditions}

		\author[inst1]{Engui Fan$^{*,}$  }
		\author[inst1]{Yuan Li}
		\author[inst1]{Xinhan Liu}
		\address[inst1]{ School of Mathematical Sciences and Key Laboratory for Nonlinear Science, Fudan University, Shanghai, 200433, China\\
			* Corresponding author and e-mail address: faneg@fudan.edu.cn  }

		\begin{abstract}
			In this paper, we study the initial-boundary value problem  for the nonlinear Schr\"{o}dinger equation   in $\mathbb{R}^{n}_{+}$
			\begin{equation*}
				i\partial_{t}u+\Delta u+\lambda |u|^pu=0, \qquad (x, t) \in \mathbb{R}_{+}^{n} \times \mathbb{R}_{+},\ \ p\in\mathbb{R}_{+}
			\end{equation*}
		 with nonhomogeneous Dirichlet boundary conditions. For the corresponding linear problem, endpoint Strichartz estimates  are derived. For the nonlinear problem, we prove local well-posedness in $H^{s}(\mathbb{R}^{n}_{+})$ with $s\in[0,\frac{5}{2})$ and $p<\frac{4}{n-2s}$.  Moreover, global well-posedness is established  in the same regularity range. For $s\in[1,\frac{5}{2})$, the  one-dimensional global theory of \cite{figment} in $H^{s}(\mathbb{R}_{+})$   is extended to  $H^{s}(\mathbb{R}^{n}_{+})$.
		  Additionally, we obtain global solutions in the lower regularity setting $s\in[0,1)$ for the first time.   It is noteworthy that for $s=0$, we overcome the lack of mass conservation resulting from the nonzero boundary data and derive the pivotal $L^{2}(\mathbb{R}^{n}_{+})$ a priori estimates.
		\end{abstract}
		
	\end{frontmatter}
	\tableofcontents
	\section{Introduction}
	We consider the nonhomogeneous initial-boundary value problem (IBVP) for the nonlinear Schr\"{o}dinger (NLS) equation in $\mathbb{R}^{n}_{+}$, formulated as:
	\begin{equation}\label{ibvp}
		\left\{\begin{aligned}
			&	i\partial_{t}u+\Delta u+\lambda |u|^pu=0, \quad (x, t) \in \mathbb{R}_{+}^{n} \times \mathbb{R}_{+}, \\
			&u(x,0)=u_0(x), \quad x\in\mathbb{R}^{n}_{+},\\& u(x',0,t)=h_{0}(x',t), \quad  (x',t)\in \mathbb{R}^{n-1}\times \mathbb{R}_{+},
		\end{aligned}
		\right.		
	\end{equation}
	where $n\geq 2$ and $x=(x_{1},\cdots,x_{n-1},x_{n})=(x',x_{n})\in \mathbb{R}^{n-1}\times\mathbb{R}_{+}$, with parameters  $\lambda\in\mathbb{R}\setminus\{0\}$ and $p\in\mathbb{R}_{+}$. The cases $\lambda<0$ and $\lambda>0$
	correspond to defocusing and focusing situations, respectively.
	 Our main purpose is to establish local and global well-posedness for the IBVP \eqref{ibvp} in the $L^{2}$-based Sobolev space $H^{s}(\mathbb{R}^{n}_{+})$.
	
	The NLS equation arises as a model for various physical phenomena, it has been widely used in fields such as condensed matter physics, water wave dynamics, nonlinear optics and plasma physics  \cite{Gross, Pitaevskii, wuli, optic, Zakharov, Zakharov2}. Specifically, this equation describes the Bose-Einstein condensate, a state of matter occurring near absolute zero where particles macroscopically manifest quantum mechanical behavior \cite{Gross, Pitaevskii}. It also models the nonlinear waves on the surface of an infinitely deep fluid \cite{Zakharov} and the collapse of Langmuir waves in a hot plasma \cite{Zakharov2}.

	The IBVP for the NLS equation under different boundary conditions has been widely studied. The first class is homogeneous IBVP in the exterior of a compact set, an unbounded domain $\Omega\subset\mathbb{R}^{n}$:
	\begin{equation}\label{ibvpo}
		\left\{\begin{aligned}
			&	i\partial_{t}u+\Delta u+\lambda |u|^pu=0, \quad (x, t) \in \Omega \times \mathbb{R}_{+}, \\
			&u(x,0)=u_0(x), \quad x\in\Omega, \\& u(x,t)|_{x\in\partial \Omega}=0, \quad t\in\mathbb{R}_{+}.
		\end{aligned}
		\right.		
	\end{equation}
	Brezis and Gallouet \cite{Brezis} first considered local and global existence of \eqref{ibvpo} in the Sobolev space $H^{2}(\Omega)\cap H_{0}^{1}(\Omega)$ for $n=2$. In the context of $n>2$, higher regularity assumptions on the initial data were required \cite{TM2,TY}. Lower regularity results in $H_{0}^{1}(\Omega)$ and $L^{2}(\Omega)$ were  provided by Burq, G\'{e}rard, and Tzvetkov \cite{Burq}. Additional related studies of \eqref{ibvpo} can be found in \cite{inveterate, novice, kil}.
	
	The second class is nonhomogeneous IBVP in  $\mathbb{R}^{n}_{+}$ with $n\geq1$. This includes the high-dimensional setting \eqref{ibvp} addressed in this paper and the one-dimensional NLS equation posed on the half line:
	\begin{equation}\label{ibvpde}
		\left\{\begin{aligned}
			&	i\partial_{t}u+\partial_{x}^{2}u+\lambda |u|^pu=0, \quad (x, t) \in \mathbb{R}_{+} \times \mathbb{R}_{+}, \\
			&u(x,0)=u_0(x),\quad x\in\mathbb{R}_{+},\\& u(0,t)=h_{0}(t),\quad t\in\mathbb{R}_{+}.
		\end{aligned}
		\right.		
	\end{equation}
	  The foundational work was established by Holmer \cite{chatter}, who proved the local well-posedness of \eqref{ibvpde} in $H^{s}(\mathbb{R}_{+})$ for $s\in[0,\frac{3}{2})\setminus\{\frac{1}{2}\}$, with $p$ satisfying \(\left\{\begin{aligned}
	  	&1 \leq p < \tfrac{4}{1-2s},\ s \in [0, \tfrac{1}{2}), \\
	  	&\quad p \geq 1,\quad\quad\; s \in (\tfrac{1}{2}, \tfrac{3}{2}).
	  \end{aligned}\right.\) Here, the boundary data $h_{0}(t)$ was selected in $H^{\frac{2s+1}{4}}(\mathbb{R}_{+})$, this stems  from the Kato smoothing estimate 	\cite{ken}
	\begin{equation}\label{o}
		\|e^{it\partial_{x}^{2}} u_{0}\|_{L^{\infty}(\mathbb{R};\dot{H}_{t}^{\frac{2s+1}{4}}(\mathbb{R}))}\lesssim \|u_{0}\|_{\dot{H}^{s}(\mathbb{R})}.
	\end{equation}	
Bona, Sun, and Zhang \cite{figment} extended Holmer's local well-posedness result \cite{chatter} to  $s\in[0,\frac{5}{2})\setminus\{\frac{1}{2}\}$. More significantly, they established global well-posedness of \eqref{ibvpde} in Sobolev space  $H^{s}(\mathbb{R}_{+})$ for $s\in[1,\frac{5}{2})$, including both defocusing case with $p>0$ and focusing case with $0<p\leq2$. In addition, Fokas, Himonas, and Mantzavinos \cite{dew} were the first to employ the unified transform method (Fokas method) to establish local well-posedness for the cubic NLS equation (\eqref{ibvpde} with $p=2$). Further analysis of this case was conducted by Erdogan and Tzirakis, who investigated several regularity properties and proved that in the defocusing setting, the $H^{s}$ norm (for $s\in [1,\frac{5}{2})$) of global solutions grows at most polynomially in time  \cite{Erdogan}.

Compared to the one-dimensional case \eqref{ibvpde}, the $n$-dimensional IBVP \eqref{ibvp} is more challenging, in which the boundary data $h_{0}(x',t)$ depends not only on time variable $t$ but also on the partial space variables $x'=(x_{1},\cdots, x_{n-1})$. For $n=2$, Ran, Sun and Zhang \cite{fallacious} introduced an optimal boundary function space $\mathcal{H}^s$  (given by \eqref{huahs} with $n=2$) and derived the global well-posedness of the IBVP \eqref{ibvp} in $H^{1}(\mathbb{R}_{+}^{2})$ for the defocusing case with $p\geq1$ and the focusing case with $1\leq p\leq \frac{4}{3}$. Subsequently, Audiard \cite{facilitate} generalized the space  $\mathcal{H}^{s}$ to $\mathbb{R}^{n}$ for $n\geq2$, defined by
\begin{equation}\label{huahs}
	\mathcal{H}^s(\mathbb{R}^n):=\left\{f \in \mathscr{S}'(\mathbb{R}^n): \widehat{f}\in L^1_{\mathrm{loc}}(\mathbb{R}^n), \left(1+|\xi'|^2+|\xi_{n}|\right)^{\frac{s}{2}}\left||\xi'|^{2}+\xi_{n}\right|^{\frac{1}{4}}\widehat{f}(\xi)\in L_{\xi}^{2}(\mathbb{R}^n) \right\},
\end{equation}
where $s\in\mathbb{R}$, $\xi=(\xi',\xi_{n})\in \mathbb{R}^{n-1}\times\mathbb{R}$, and equipped with the norm
\begin{equation}\label{blackmail}			\|f\|_{\mathcal{H}^{s}(\mathbb{R}^{n})}:=\left(\int_{\mathbb{R}^{n}}\left(1+|\xi'|^2+|\xi_{n}|\right)^{s}\left||\xi'|^{2}+\xi_{n}\right|^{\frac{1}{2}}|\widehat{f}(\xi)|^{2}d\xi\right)^{\frac{1}{2}}.
\end{equation}
Correspondingly, the trace estimate for the $n$-dimensional Schr\"{o}dinger group reads
\begin{equation}
	\|e^{it\Delta} u_{0}\|_{L^{\infty}(\mathbb{R};\mathcal{H}  ^{s}(\mathbb{R}^{n}))}\lesssim \|u_{0}\|_{H^{s}(\mathbb{R}^{n})}.
\end{equation}
The properties of the space $\mathcal{H}^s$  are thoroughly studied in \cite{facilitate}, and based on this analysis, the author established  Strichartz estimates without endpoint for the linear problem corresponding to \eqref{ibvp}. Moreover, the work in \cite{facilitate} obtained the existence of local solutions (for $0<p<\frac{4}{n-2}$)  and global solutions for small data  (for $\frac{4}{n}\leq p<\frac{4}{n-2}$) to the  IBVP \eqref{ibvp}  in $H^{1}(\mathbb{R}_{+}^{n})$.

The aim of this paper is to establish the local and global well-posedness of the IBVP \eqref{ibvp} for the $n$-dimensional NLS equation in $H^{s}(\mathbb{R}^{n}_{+})$, with $s\in[0,\frac{5}{2})$. To our knowledge, the nonlinear well-posedness results for the one-dimensional IBVP \eqref{ibvpde} established in \cite{chatter,figment} have not been fully extended to the $n$-dimensional setting \eqref{ibvp}. Here, local well-posedness is established under the condition $p<\frac{4}{n-2s}$, which parallels the one-dimensional condition $p<\frac{4}{1-2s}$. Furthermore, we remove the small-norm assumption in \cite{facilitate} and prove global well-posedness in $H^{s}(\mathbb{R}^{n}_{+})$ for the range $s\in[1,\frac{5}{2})$, thereby generalizing the one-dimensional global  result of \cite{figment} in $H^{s}(\mathbb{R}_{+})$. Moreover, global solutions are constructed for the first time in the lower regularity range $s\in[0,1)$, based on the key derivation of the $L^2(\mathbb{R}^{n}_{+})$ a priori estimates. Our results are formally stated in the next section.

\subsection{Statement of results}
In this paper, we consider initial data $u_{0}\in H^{s}(\mathbb{R}^{n}_{+})$ and boundary data $h_{0}\in \mathcal{H}^{s}(\mathbb{R}^{n}_{+})$, and focus on the distributional solutions  with strong traces of the IBVP \eqref{ibvp}  in the sense defined below.  The counterpart of this definition for the one-dimensional nonhomogeneous IBVP \eqref{ibvpde} was provided in \cite{chatter}.
	\begin{definition}
		We say $u:\overline{\mathbb{R}^{n}_{+}}\times [0,T)\rightarrow\mathbb{C}$, $0<T\leq\infty$ is a distributional solution of the IBVP \eqref{ibvp} with strong traces if
		\begin{enumerate}[label=(\roman*)]
			\item $u\in L^{1}_{\mathrm{loc}}(\mathbb{R}_{+}^{n}\times(0,T))$ and $|u|^{p}u\in L^{1}_{\mathrm{loc}}(\mathbb{R}_{+}^{n}\times(0,T))$;
			\item $u$ satisfies  equation  \eqref{ibvp} in the sense of distributions on $\mathbb{R}_{+}^{n}\times(0,T)$;
			\item \label{covenant} space trace:  $t \mapsto u(\cdot,t) \in C([0,T);H^{s}(\mathbb{R}_{+}^{n}))$ and $u(\cdot,0)=u_{0}$;
			\item \label{hollow} time trace:  $x_{n} \mapsto u(\cdot,x_{n},\cdot) \in C\left(\overline{\mathbb{R}_{+}};\mathcal{H}^{s}_{\mathrm{loc}}(\mathbb{R}^{n-1}\times(0,T))\right)$ and $u(\cdot,0,\cdot)=h_{0}$.
		\end{enumerate}
	\end{definition}
Moreover, we always assume that the initial data $u_{0}$ and boundary data  $h_{0}$ satisfy the following compatibility conditions
\begin{equation}\label{comll}
	u_{0}(x',0)=h_{0}(x',0),\quad \text{for}\ \ \tfrac{1}{2}<s<\tfrac{5}{2},
\end{equation}
\begin{equation} \label{agdj} \int_{\mathbb{R}^{n-1}}\int_{\mathbb{R}_{+}}\frac{|u_{0}(x',t)-e^{-it^{2}\Delta'}h_{0}(x',t^{2})|^{2}}{t}dt\,dx'<\infty,\quad \text{for}\ \ s=\tfrac{1}{2},	
\end{equation}
where \eqref{agdj} is given by \cite{facilitate}.
	\begin{definition}	[See \cite{tao}]\label{fraternity}
		We say that the exponent pair $(q,r)$ is an admissible pair if  $q, r\geq 2$, and
		\begin{equation}
			\frac{2}{q}=n\left(\frac{1}{2}-\frac{1}{r}\right),
		\end{equation}
		with the exception of the case $(q, r) = (2, \infty)$ when $n = 2$.
		In particular, for $n>2$, $(q,r)=(2,\frac{2n}{n-2})$ is called the endpoint.
	\end{definition}

We now state the main results of this paper. Theorem \ref{chaste} provides the local well-posedness theory, and Theorems \ref{global} and \ref{new L2} establish the global well-posedness results.
\begin{theorem}[Local well-posedness]\label{chaste}
	Let $s\in [0,\frac{5}{2})$, $p\geq \lceil s\rceil-1$ and $(n-2s)\,p<4$. For initial data $u_{0}\in H^{s}(\mathbb{R}^{n}_{+})$ and boundary data $h_{0}\in\mathcal{H}^{s}(\mathbb{R}_{+}^{n})$  satisfying the compatibility condition, there exists a maximal solution
	\begin{equation}\label{tarnish}
		u\in C\left(\left[0,T_{\max}\right);H^{s}(\mathbb{R}^{n}_{+})\right)\cap C\left(\overline{\mathbb{R}_{+}\mkern-4mu}\,;\mathcal{H}^{s}_{\mathrm{loc}}(\mathbb{R}^{n-1}\times(0,T_{\max}))\right)
	\end{equation}
	to the IBVP \eqref{ibvp}. Moreover, the solution $u$ satisfies the following properties:
	\begin{enumerate}[label=(\roman*)]
		\item When $s < 1$ and $p \leq \frac{2+2s}{n-2s}$ (with $p < \frac{2+2s}{n-2s}$ when $n = 2$) or $s \geq 1$, the solution is unique. In other cases, uniqueness requires an additional condition $u\in L_{\mathrm{loc}}^{q}\left(\left[0,T_{\max}\right);L^{r}(\mathbb{R}^{n}_{+})\right)$, for some admissible pair $(q,r)$.
		\item If $T_{\max}< \infty$, then $\lim\limits_{t\rightarrow T_{\max}} \left\|u(t)\right\|_{H^{s}(\mathbb{R}^{n}_{+})}=\infty$.
		\item $u\in L_{\mathrm{loc}}^{q}\left(\left[0,T_{\max}\right);W^{s,r}(\mathbb{R}^{n}_{+})\right)$ for each admissible pair $(q,r)$.
		\item If $p\geq \lceil s\rceil$ or $p\in2\mathbb{Z}$, the solution map $\left(u_{0},h_{0}\right)\mapsto u $ is continuous.
	\end{enumerate}
\end{theorem}
Theorem \ref{chaste} concerns two distinct cases:  low regularity solutions for $s\leq\frac{n}{2}$ and high regularity solutions for $s>\frac{n}{2}$. Accordingly, the condition on the exponent $p$ splits into two forms:
\begin{equation}
	(n-2s)p<4 \quad  \text{implies} \quad \left\{
	\begin{aligned}
		&p<\tfrac{4}{n-2s},\quad  s\leq\tfrac{n}{2},\\
		&p<\infty,\quad  s>\tfrac{n}{2},
	\end{aligned}
	\right.
\end{equation}
where the inequality  $p<\frac{4}{n-2s}$ is
interpreted as $p<\infty$ when $s=\frac{n}{2}$.  The main challenge in proving Theorem \ref{chaste} resides in obtaining the low-regularity solutions for $s\leq\frac{n}{2}$ under the condition $p<\frac{4}{n-2s}$.

The upper bound $\frac{4}{n-2s}$ arises from the scaling
\begin{equation*}
	u(x,t) \mapsto u_{\nu}(x,t):=\nu^{\frac{2}{p}}u(\nu x,\nu^{2}t),\quad \nu>0,
\end{equation*}
which leaves the equation invariant. This scaling yields
\begin{equation*}
	\|u_{\nu}(\cdot,0)\|_{\dot{H}^{s}(\mathbb{R}^{n})}=\nu^{\frac{2}{p}-\frac{n}{2}+s} \|u(\cdot,0)\|_{\dot{H}^{s}(\mathbb{R}^{n})},
\end{equation*}
where the case $\frac{2}{p}-\frac{n}{2}+s=0$ is called critical, corresponding to $p = \frac{4}{n-2s}$, and the case $\frac{2}{p}-\frac{n}{2}+s>0$ is called subcritical, that is, $p < \frac{4}{n-2s}$ \cite{LinPon}. The case $p<\frac{4}{n-2s}$   in the initial value problem (IVP) for the $n$-dimensional NLS equation has been widely studied. For instance, Ginibre and Velo \cite{Ginibre} established the existence and uniqueness in  $H^{1}(\mathbb{R}^{n})$ for $p<\frac{4}{n-2}$, a result later extended by Tsutsumi \cite{Tsutsumi} to $L^{2}(\mathbb{R}^{n})$ for $p<\frac{4}{n}$. Cazenave and Weissler \cite{Cazenave} generalized the local existence theory  to $H^{s}(\mathbb{R}^{n})$ for $s\in\left[0,\frac{n}{2}\right)$ with $p\leq \frac{4}{n-2s}$. In addition, the IVP has also been investigated extensively in the critical cases: the energy-critical problem ($\dot{H}^{1}(\mathbb{R}^{n})$ with  $p=\frac{4}{n-2}$) \cite{Bourgain, tyro, Ryckman Visan,Tao 2005,Visan 2007} and the mass-critical problem ($L^{2}(\mathbb{R}^{n})$ with  $p=\frac{4}{n}$)  \cite{cohesion,rookie,neophyte}.

\begin{remark}
In Theorem \ref{chaste}, establishing the local well-posedness for the low-regularity case $s\in[0,\frac{5}{2})$ with $s\leq\frac{n}{2}$ under the condition  $p<\frac{4}{n-2s}$ requires appropriate bounds over the time derivative of the nonlinear term.  The one-dimensional NLS equation does not involve this issue, and the result for $p<\frac{4}{1-2s}$ is obtained in \cite{chatter, figment}. For the $n$-dimensional NLS, the argument in \cite{facilitate} that yields local  $H^{1}(\mathbb{R}^{n}_{+})$ solutions for $p<\frac{4}{n-2}$ cannot be applied to Theorem~\ref{chaste} directly. To overcome this, we employ distinct analytical strategies depending on the range of $s$:
\begin{enumerate}[label=(\roman*)]
\item  For $s \in (0, 2)$, two precise parameter claims  (\eqref{qwu}--\eqref{qwui} and \eqref{die}--\eqref{dier}) are established to bound the norm of the nonlinear term.
\item For $s \in [2, \frac{5}{2})$, we modify the original estimate for the operator $S(t)$ (defined by \eqref{st}) in a norm weaker than $\mathcal{H}^s$. This  refinement yields the $T$-uniform estimates (Proposition~\ref{abss}) to enable us to handle the nonlinearity.
\end{enumerate}
\end{remark}

\begin{theorem}[Global well-posedness for $s \in [1, \frac{5}{2})$]\label{global}
	Let $s\in [1, \frac{5}{2})$ and $p\geq \lceil s\rceil-1$. Given initial and boundary data:
		\begin{itemize}
		\item For $s\in [1,2)$, $u_{0}\in H^{s}(\mathbb{R}^{n}_{+})$  and $h_{0}\in\mathcal{H}^{s}(\mathbb{R}^{n}_{+})\cap H^{1}(\mathbb{R}^{n}_{+})$;
		\item For $s\in [2,\frac{5}{2})$, $u_{0}\in H^{s}(\mathbb{R}^{n}_{+})$ and $h_{0}\in\mathcal{H}^{s}(\mathbb{R}^{n}_{+})$,
		\end{itemize}
	satisfying the compatibility condition, and assume the parameters satisfy
	\begin{equation}
		\lambda<0:\, \begin{cases}
			p<\frac{4}{n-2}, & s=1\ \text{or} \ s\in[2,\frac{5}{2}),\\ p<\frac{8}{3(n-2)}, &s\in(1,2),
		\end{cases}\;
		\quad\ \lambda>0: \ p\leq \tfrac{4}{n+1}.
	\end{equation}
	Then the IBVP \eqref{ibvp} admits a unique global solution $u \in C\left([0,\infty); H^s(\mathbb{R}^n_{+})\right)$.
\end{theorem}

\begin{remark}\label{remaa}
 For $s=1$,	the proof of Theorem \ref{global}  under the condition $0<p< \frac{4}{n-2}$  ($\lambda<0$) or $0<p\leq \frac{4}{n+1}$ ($\lambda>0$)	relies on  the  $H^1$  a priori estimate \eqref{drone}, which requires solutions in $H^{2}$.
	For $p\geq1$, Theorem \ref{chaste} establishes the existence of such solutions. However, in the low-power case where $p<1$, Theorem~\ref{chaste} fails to provide  $H^2$  solutions, as it requires the solution  to satisfy the boundary data in the $\mathcal{H}^2$  sense. To address this, we establish  weaker solutions (in Proposition \ref{weak2}) satisfying the boundary data  in the $\mathcal{W}^2$  sense \eqref{huaw}. This is sufficient to obtain global $H^1$ solutions through approximation.

In addition, for $s\in(1,\frac{5}{2})$, the analytical approach \cite{figment} used to establish global well-posedness for the one-dimensional IBVP \eqref{ibvpde} cannot be directly applied, as the key embedding $H^{1}(\mathbb{R}^n_+) \hookrightarrow L^{\infty}(\mathbb{R}^n_+)$ is not valid for $n\geq2$. Therefore, new analytical strategies are required. Specifically, for  $s\in(1,\frac{3}{2}]$, we first derive a priori bound for  $\|u(t)\|_{H^s}$ on a small interval  $s\in(1,1+\varepsilon)$ with $\varepsilon>0$, and then develop a decomposition-recursion technique \eqref{421+}–\eqref{423+} to extend the bound up to $s=\frac{3}{2}$. To handle the range  $s>\frac{3}{2}$,  the core idea of our argument is to first obtain a difference estimate \eqref{buddy} via Lemma \ref{imp}, and subsequently,  through the $L^{p}$-mean continuity to convert it into a pointwise estimate of the time derivative  \eqref{ideology}.
\end{remark}

\begin{theorem}[Global well-posedness for $s \in [0, 1)$]\label{new L2}	
Let $s \in [0, 1)$. The following global well-posedness results hold:
\begin{enumerate}[label=(\roman*)]
	\item For $n\geq 2$, given initial data $u_{0}\in H^{s}(\mathbb{R}^{n}_{+})$ and boundary data   $h_{0}\in\mathcal{H}^{s}(\mathbb{R}^{n}_{+})$ satisfying the compatibility condition, and with the parameters satisfying
\begin{equation*}
\begin{aligned}
	& s = 0: \ p \leq \tfrac{2}{n} \ \text{with} \ (p,n)
	\neq (1, 2), \\
	& s \in (0, 1): \  p \leq \min\{1, \tfrac{2(s+1)}{n}\} \ \text{with} \
	\begin{cases}
		p < \frac{2}{n-2}, & n < 6, \\
		p \leq \frac{6-4s}{3(n-2)}, & n \geq 6,
	\end{cases}
\end{aligned}
\end{equation*}
then the IBVP \eqref{ibvp} admits a unique global solution
$u\in C\left([0,\infty);H^{s}(\mathbb{R}^{n}_{+})\right).$
\item  For $n=1$, assume $p\leq1$.  Then for any $u_{0}\in H^{s}(\mathbb{R}_{+})$ and $h_{0}\in H^{\frac{2s+1}{4}}(\mathbb{R}_{+})$ satisfying the compatibility condition,
the IBVP \eqref{ibvpde} admits a unique global solution $	u\in C\left([0,\infty);H^{s}(\mathbb{R}_{+})\right).$
\end{enumerate}
\end{theorem}
\begin{remark}
	Unlike the initial value problem and the homogeneous IBVP,  the nonhomogeneous IBVP \eqref{ibvp} lacks mass conservation. 	The standard $L^{2}$ a priori estimate \eqref{endow} produces  the trace of the normal derivative $\partial_{n}u|_{x_{n}=0}$. To eliminate this negative effect,  we consider the solution $v$ of a unforced linear IBVP \eqref{coagulate} with the same initial and boundary data as the nonlinear IBVP \eqref{ibvp}. By  establishing an $L^{2}$ a priori estimate for the  difference  $u-v$, we prove,   for the first time, the global well-posedness of the IBVP \eqref{ibvp} in the space  $L^{2}(\mathbb{R}_+^n)$.  As a corollary, we also obtain the first global well-posedness result in   $L^{2}(\mathbb{R}_+)$ for the one-dimensional IBVP \eqref{ibvpde}.	
\end{remark}



Our results for the nonlinear IBVP \eqref{ibvp} rely on the Strichartz estimates for the corresponding linear problem \eqref{temporal},  which serve as the foundation for the further nonlinear estimates of \eqref{ibvp}.   Notably, in contrast to the linear results provided in \cite{facilitate}, we establish the endpoint Strichartz estimates for the first time; this is guaranteed by Propositions~\ref{peony} and \ref{kiln}. The core idea behind these propositions is to  construct multiplier-type Fourier integral operators  that meet the conditions of the  Keel-Tao theorem \cite{tao}. These operators are designed to reformulate the components of the solution formulae that cannot be represented by the Schr\"{o}dinger group, thereby enabling the application of the  Keel-Tao theorem \cite{tao} to complete the proofs.

An outline of this paper is as follows.    Section \ref{sec2} is devoted to deriving Strichartz estimates including the endpoint for the linear IBVP \eqref{temporal}.  This linear problem is decomposed into three components, each of which is estimated  separately. Section \ref{sec3} develops the local  theory for the nonlinear IBVP \eqref{ibvp} and proves Theorem \ref{chaste} through a fixed-point argument in several settings. In Section \ref{sec4}, we establish $H^1$
and $L^{2}$ a priori estimates for the nonlinear IBVP \eqref{ibvp} and perform associated regularity analyses, which subsequently yields the global  results stated in Theorems \ref{global} and \ref{new L2}.

\subsection{Notation}
\begin{itemize}[label={\textbf{--}}]
	\item For $s\in\mathbb{R}$ and  $1<p<\infty$,  \( W^{s,p}(\mathbb{R}^{n})\) denotes the standard Sobolev space, and $H^{s}(\mathbb{R}^{n})=W^{s,2}(\mathbb{R}^{n})$. For an open set  $\Omega$, $W^{s,p}(\Omega)$ is defined as the restriction of \( W^{s,p}(\mathbb{R}^{n})\) to $\Omega$.
\item For $s\in (0,1)$, $1\leq p,q \leq \infty$, a Banach space $X$, and an open interval $I\subset\mathbb{R}$, the vector-valued Besov space $B^{s}_{p,q}(I;X)$ is defined by the norm
\begin{equation}\label{denounce}
	\begin{aligned} \|f\|_{B^{s}_{p,q}(I;X)}&:=\,\|f\|_{L^{p}(I;X)}+\|f\|_{\dot{B}^{s}_{p,q}(I;X)}\\			&=\left(\int_{I}\|f(t)\|_{X}^{p}dt\right)^{\frac{1}{p}}+\left(\int_{\mathbb{R}} \|f(t+h)-f(t)\|_{L^{p}(I_{h};X)}^{q} \frac{dh}{|h|^{1+sq}}\right)^{\frac{1}{q}},
	\end{aligned}
\end{equation}
where $I_{h}=\{t: t+h\in I \}$, with the standard modifications when $p=\infty$ or $q=\infty$.
	\item $\mathcal{H}^{s}(\mathbb{R}^{n-1}\times (a,b))$ is  defined as the restriction space of $\mathcal{H}^{s}(\mathbb{R}^{n})$   to the domain $\mathbb{R}^{n-1}\times(a,b)$, with norm
\begin{equation}
	\|f\|_{\mathcal{H}^{s}(\mathbb{R}^{n-1}\times (a,b))}:=\inf\left\{\|F\|_{\mathcal{H}^{s}(\mathbb{R}^{n})}: F\in\mathcal{H}^{s}(\mathbb{R}^{n}), F|_{\mathbb{R}^{n-1}\times (a,b)}=f\right\}.
\end{equation}
	\item For function spaces on $\mathbb{R}^{n-1}\times(0,T)$, the subscript ``$\mathrm{loc}$'' refers only to the time variable:  $f\in\mathcal{H}_{\mathrm{loc}}^{s}(\mathbb{R}^{n-1}\times(0,T))$ if and only if $f\in \mathcal{H}^{s}(\mathbb{R}^{n-1}\times(0,T'))$ for any $T'\in (0,T)$. The same convention also applies to $H_{\mathrm{loc}}^{s}$.
	\item For $\xi\in\mathbb{R}^{n}$, we set $\langle\xi\rangle:=\sqrt{1+|\xi|^{2}}$.
	\item We write $A\lesssim B$ if $A\leq CB$ for some constant $C>0$, and $A\approx B$ if $A\lesssim B$ and $B\lesssim A$.
	\item
     Throughout, $(q,r)$ denotes an admissible pair as in Definition \ref{fraternity}, with $q'$ and $r'$ being their respective conjugate exponents. The pair $(\gamma, \rho)$ denotes another admissible pair, with conjugates $\gamma'$ and $\rho'$.
\end{itemize}

\section{Linear estimates}\label{sec2}
In this section, we focus on the following linear IBVP
\begin{align}\label{temporal}
	\left\{\begin{aligned}
		&	i\partial_{t}u+\Delta u=f, \quad (x, t) \in \mathbb{R}_{+}^{n} \times \mathbb{R}_{+}, \\
		&u(x,0)=u_0(x), \quad x\in\mathbb{R}^{n}_{+},\\& u(x',0,t)=h_{0}(x',t), \quad  (x',t)\in \mathbb{R}^{n-1}\times \mathbb{R}_{+}.
	\end{aligned}
	\right.		
\end{align}
Our primary objective is to establish Strichartz-type estimates for solution $u$, as these estimates serve as fundamental tools in establishing both local and global well-posedness for the nonlinear problem. By the linear superposition principle, we decompose the linear IBVP \eqref{temporal} into three components: the unforced initial value problem \eqref{21}, the forced initial value problem \eqref{inh}, and the reduced initial-boundary value problem \eqref{locust}.
\subsection{Unforced initial value problem}\label{sec21}
First, we consider the IVP
\begin{equation}\label{21}
	\left\{\begin{aligned}&i\partial_{t}v+\Delta v=0, \quad (x,t)\in \mathbb{R}^n \times \mathbb{R},\\
		&v(x,0)=v_0(x),\quad x\in \mathbb{R}^{n}.
	\end{aligned}\right.
\end{equation}
The solution to \eqref{21} can be represented as $v=e^{it\Delta}v_{0}$, where the corresponding estimates are established in the following proposition. In particular, the $(x',t)$-estimate \eqref{circuitous} of $v$ provides a feasible function space $\mathcal{H}^{s}$ for boundary data.
\begin{proposition}\label{p21}
	For $s\in\mathbb{R}$, $j\in\mathbb{N}$, and $0<l<1$,  the solution $v$ to initial value problem \eqref{21} satisfies the following estimates:
	\begin{equation}\label{autocrat}	\|v\|_{C_{b}^{j}(\mathbb{R};H^{s-2j}(\mathbb{R}^{n}))}+\|v\|_{W^{j,q}(\mathbb{R};W^{s-2j,r}(\mathbb{R}^{n}))}+\|v\|_{B^{l}_{q,2}(\mathbb{R};W^{s-2l,r}(\mathbb{R}^{n}))}\lesssim \|v_{0}\|_{H^{s}(\mathbb{R}^{n})}.
	\end{equation}
	Moreover, for $s\geq0$, we have
	\begin{equation}\label{circuitous}
		\|v\|_{C_{b}(\mathbb{R};\mathcal{H}^{s}(\mathbb{R}^{n}))}\lesssim \|v_0\|_{H^{s}(\mathbb{R}^{n})},
	\end{equation}
	and
	\begin{equation}\label{tentative}
		\|v\|_{C_{b}^{1}(\mathbb{R};\mathcal{H}^{s-1}(\mathbb{R}^{n}))}\lesssim \|v_0\|_{H^{s}(\mathbb{R}^{n})}.
	\end{equation}
\end{proposition}
\begin{proof}
	The estimates for $\|v\|_{C_{b}(\mathbb{R};H^{s})}$ and $\|v\|_{L^{q}(\mathbb{R};W^{s,r})}$  can be found in Chapter 2 of \cite{facetious}. From $v_{t}=i\Delta v$, we derive the bounds for $\|v\|_{C_{b}^{j}(\mathbb{R};H^{s-2j})}$ and $\|v\|_{W^{j,q}(\mathbb{R};W^{s-2j,r})}$. The third term in \eqref{autocrat} is then obtained by interpolation. For the estimate \eqref{circuitous}, note that the solution $v$ to \eqref{21} can be expressed as
	\begin{equation}\label{assault}
		v(x,t)=\frac{1}{(2\pi)^n}\int_{\mathbb{R}^n}e^{ix\cdot\xi} e^{-i|\xi|^{2}t}\widehat{v}_{0}({\xi})d\xi.
	\end{equation}
	Applying the variable substitution to \eqref{assault}: For $\xi_{n}\in\mathbb{R}_{+}$, set $\xi_{n}=\sqrt{|\eta+|\xi'|^{2}|}$; For $\xi_{n}\in\mathbb{R}_{-}$, set $\xi_{n}=-\sqrt{|\eta+|\xi'|^{2}|}$, where $\eta\in(-\infty,-|\xi'|^{2})$. This yields $v=\frac {1}{(2\pi)^n}\left(v_{+}+v_{-}\right)$ with
	\begin{equation}\label{vain}	v_{\pm}(x,t)=\pm\int_{\mathbb{R}^{n-1}}\int_{-\infty}^{-|\xi'|^{2}}e^{ix'\cdot\xi'\pm ix_{n}\sqrt{\left|\eta+|\xi'|^{2}\right|} } \,e^{i\eta t}  \, \tfrac{\widehat{v}_0\left(\xi',\pm\sqrt{\left|\eta+|\xi'|^{2}\right|}\right)}{2\sqrt{\left|\eta+|\xi'|^{2}\right|}}d\eta \, d\xi'.
	\end{equation}
	We only estimate $v_{+}$ since $v_{-}$  can be treated analogously.  By the definition of $\mathcal{H}^s$ and the substitution $\eta=-\xi_{n}^{2}-|\xi'|^{2}$ with $\xi_{n}\in\mathbb{R}_{+}$, we obtain
	\begin{equation}
		\begin{aligned}
			\|v_{+}(\cdot,x_{n},\cdot)\|_{\mathcal{H}^s(\mathbb{R}^{n})}^2	&\lesssim\int_{\mathbb{R}^{n-1}}\int_{-\infty}^{-|\xi'|^{2}}\left(1+|\xi'|^2+|\eta|\right)^s\tfrac{\left|\widehat{v}_0\left(\xi',\sqrt{\left|\eta+|\xi'|^{2}\right|}\right)\right|^2}{\sqrt{\left|\eta+|\xi'|^{2}\right|}}d\eta \, d\xi' \\		
			&\lesssim \int_{\mathbb{R}^{n-1}}\int_{\mathbb{R}_{+}} \left(1+|\xi'|^2+\xi_{n}^2\right)^s|\widehat{v}_0(\xi',\xi_{n})|^2d\xi_{n}\, d\xi'\\
			&\lesssim\|v_0\|_{H^s(\mathbb{R}^n)}^2.
		\end{aligned}
	\end{equation}
	The continuity follows from the dominated convergence theorem, and hence \eqref{circuitous} holds. For \eqref{tentative}, differentiating \eqref{vain} with respect to $x_{n}$ yields
	\begin{equation*}	\partial_{n}v_{+}(x,t)=\frac{i}{2}\int_{\mathbb{R}^{n-1}}\int_{-\infty}^{-|\xi'|^{2}}e^{ix'\cdot\xi'+i x_{n}\sqrt{|\eta+|\xi'|^{2}|}}\, e^{i\eta t} \, \widehat{v}_0(\xi',\sqrt{\left|\eta+|\xi'|^{2}\right|})d\eta\, d\xi'.
	\end{equation*}
	Similarly, by the definition of $\mathcal{H}^s$ and substitution, we have
	\begin{equation*} \left\|\partial_{n} v_{+}(\cdot,x_{n},\cdot)\right\|_{\mathcal{H}^{s-1}(\mathbb{R}^{n})}^{2}
		\lesssim \int_{\mathbb{R}^{n-1}}\int_{\mathbb{R}_{+}}\left(1+|\xi'|^2+\xi_{n}^{2}\right)^{s-1}\xi_{n}^{2}\left|\widehat{v}_0(\xi',\xi_{n})\right|^2d\xi_{n}\, d\xi'
		\lesssim \|v_{0}\|_{H^{s}(\mathbb{R}^{n})}^{2}.
	\end{equation*}
\end{proof}
\subsection{Forced initial value problem}\label{sec22}
Next, we study the forced IVP
\begin{equation}\label{inh}
	\left\{\begin{aligned}
		&i\partial_{t}w+\Delta w=f, \quad (x,t)\in \mathbb{R}^{n} \times \mathbb{R} ,\\
		&w(x,0)=0,\quad x\in \mathbb{R}^{n}.
	\end{aligned}\right.
\end{equation}
By Duhamel's principle, the solution admits the representation
\begin{equation}\label{thorny}
	w(t)=-i\int_{0}^{t}e^{i(t-\tau)\Delta}f(\tau)d\tau.
\end{equation}
We now give the Strichartz estimates for $w$.
\begin{proposition}\label{fuss}
	For $s\in\mathbb{R}$, the $w$ given by \eqref{thorny} satisfies
	the following estimates
	\begin{equation}\label{wrath}
		\|w\|_{C_{b}(\mathbb{R}; H^{s}(\mathbb{R}^{n}))}+\|w\|_{L^{q}(\mathbb{R}; W^{s,r}(\mathbb{R}^{n}))}\lesssim \|f\|_{L^{\gamma'}(\mathbb{R}; W^{s,\rho'}(\mathbb{R}^{n}))},
	\end{equation}
	\begin{equation}\label{famine} \|w\|_{C^{1}([0,T]; H^{s-2}(\mathbb{R}^{n}))}+\|w\|_{W^{1,q}(0,T;W^{s-2,r}(\mathbb{R}^{n}))}\lesssim \|f(0)\|_{H^{s-2}(\mathbb{R}^{n})}+\|f\|_{W^{1,\gamma'}(0,T;W^{s-2,\rho'}(\mathbb{R}^{n}))}.
	\end{equation}
	For $s\in (0,2)$,
	\begin{equation}\label{kidney}
		\|w\|_{L^{q}(\mathbb{R};W^{s,r}(\mathbb{R}^{n}))} +\|w\|_{B^{\frac{s}{2}}_{q,2} (\mathbb{R}; L^{r} (\mathbb{R}^{n}))}\lesssim \|f\|_{L^{\gamma'} (\mathbb{R};  W^{s,\rho'}(\mathbb{R}^{n}))}+\|f\|_{B^{\frac{s}{2}}_{\gamma',2} (\mathbb{R}; L^{\rho'} (\mathbb{R}^{n}))}.
	\end{equation}
\end{proposition}
\begin{proof}
	The estimate \eqref{wrath} can be found in Chapter 2 of \cite{facetious}. For \eqref{famine}, starting from $$\partial_{t}w(t)=-ie^{it\Delta}f(0)-i\int_{0}^{t}e^{i(t-\tau)\Delta}\partial_{t}f(\tau)d\tau,$$ we apply \eqref{autocrat} and \eqref{wrath} to get
	\begin{equation} \label{one}
		\|\partial_{t}w\|_{C([0,T]; H^{s-2}(\mathbb{R}^{n}))\cap L^{q}(0,T;W^{s-2,r}(\mathbb{R}^{n}))}\lesssim \|f(0)\|_{H^{s-2}(\mathbb{R}^{n})}+  \|\partial_{t}f\|_{L^{\gamma'}(0,T;W^{s-2,\rho'}(\mathbb{R}^{n}))}.
	\end{equation}
	Thus \eqref{famine} follows by combining \eqref{wrath} with \eqref{one}. The estimate \eqref{kidney} is presented in Section 3 of \cite{facilitate}.
\end{proof}

The remainder of this section is devoted to the $\mathcal{H}^{s}$ estimates of $w$.

\begin{proposition}\label{peony}
	The $w$ given by \eqref{thorny} satisfies
	\begin{equation}\label{cope}	
		\|w\|_{C_{b}(\mathbb{R};\mathcal{H}^{0}(\mathbb{R}^{n}))}\lesssim \|f\|_{L^{\gamma'}(\mathbb{R};L^{\rho'}(\mathbb{R}^{n}))}.
	\end{equation}
\end{proposition}
We emphasize that the right-hand side of the inequality holds for arbitrary admissible pair $(\gamma,\rho)$. Significantly, the framework successfully extends to the endpoint case $(\gamma,\rho)=\left(2,\frac{2n}{n-2}\right)$, $n>2$, which is established for the first time. The proof of this proposition makes use of the following lemma.
\begin{lemma}[See \cite{tao}]\label{tas}
	Let $H$ be a Hilbert space. Suppose that for all $g\in H$, $h\in L^{1}(\mathbb{R}^{n})\cap L^{2}(\mathbb{R}^{n})$ and $s\in\mathbb{R}$, the operator $U(s): H\rightarrow L^{2}(\mathbb{R}^{n})$ satisfies the following estimates:
	\begin{equation}\label{thrift}
		\|U(s)g\|_{L^{2}(\mathbb{R}^{n})}\lesssim \left\|g\right\|_{H} \quad\text{and}\quad
		\|U(s)U(\tau)^{*}h\|_{L^{\infty}(\mathbb{R}^{n})}\lesssim |s-\tau|^{-\frac{n}{2}} \|h\|_{L^{1}(\mathbb{R}^{n})},
	\end{equation}
	where $s\neq\tau$, and $U(s)^{*}$ denotes the adjoint operator of  $U(s)$. Then the estimates
	\begin{equation}\label{alloy}
		\|U(s) g\|_{L^{\gamma}(\mathbb{R};L^{\rho}(\mathbb{R}^{n}))}\lesssim \left\|g\right\|_{H},
	\end{equation}
	\begin{equation}\label{utter}
		\left\|\int_{\mathbb{R}} U(\tau)^{*}f(\tau)d\tau\right\|_{H}\lesssim \|f\|_{L^{\gamma'}(\mathbb{R};L^{\rho'}(\mathbb{R}^{n}))}
	\end{equation}
	hold for all admissible pairs $(\gamma,\rho)$, and all $f\in L^{1}(\mathbb{R};L^{2}(\mathbb{R}^{n}))\cap L^{\gamma'}(\mathbb{R};L^{\rho'}(\mathbb{R}^{n}))$.
\end{lemma}
\begin{proof}[Proof of Proposition \ref{peony}]
	Following the methodology in \cite{kenig,chatter}, the solution formula can be reformulated as
	\begin{equation}
		\begin{aligned}
			w(x,t)=&-\int_{\mathbb{R}} \frac{\text{sgn}(\tau)}{2} e^{i(t-\tau)\Delta}f(\tau)d\tau\\
			&+\frac{i}{2\pi}\int_{\mathbb{R}^{n-1}} \int_{\mathbb{R}}\left[\lim_{\varepsilon\rightarrow0}\int_{\mu+|k_{n}|^{2}\geq\varepsilon}e^{ik\cdot x}\frac{e^{it\mu}}{\mu+|k|^{2}}\widehat{f}(k,\mu)dk_{n}\right]\,d\mu\, dk'
			\equiv I_{1}+I_{2}.
		\end{aligned}
	\end{equation}
	From estimate \eqref{circuitous} and the Strichartz estimates listed in Chapter 2 of \cite{facetious}, we obtain
	\begin{equation}
		\|I_{1}\|_{C_{b}(\mathbb{R};\mathcal{H}^{0}(\mathbb{R}^{n}))}\lesssim \left\|\int_{\mathbb{R}} \text{sgn}(\tau) e^{-i\tau\Delta}f(\tau)d\tau\right\|_{L^{2}(\mathbb{R}^{n})}\lesssim  \|f\|_{L^{\gamma'}(\mathbb{R};L^{\rho'}(\mathbb{R}^{n}))}.
	\end{equation}
	By applying properties of the Fourier transform, we rewrite $I_{2}$ in convolution form:
	\begin{equation}\label{kee}
		I_{2}(x,t)=	-\frac{1}{2\pi}  \int_{\mathbb{R}} \int_{\mathbb{R}} \int_{\mathbb{R}^{n-1}}e^{it\mu}e^{ik'\cdot x'}\left[\widehat{f}^{\,y'}(k',\cdot,\tau)\ast M(k',\cdot,\tau,\mu) \right](x_{n})\,dk'\,d\mu\, d\tau,
	\end{equation}
	where the integral kernel
	{\small	\begin{equation*}
			M(k',\delta,\tau,\mu)=\frac{e^{-i\mu \tau}}{2\sqrt{\left|\mu+|k'|^{2}\right|}}	\left(-\raisebox{0.5ex}{$\chi$}_{\mathbb{R}_{+}}(\mu+|k'|^{2})e^{-|\delta|\sqrt{\left|\mu+|k'|^{2}\right|}}+\raisebox{0.5ex}{$\chi$}_{\mathbb{R}_{-}}(\mu+|k'|^{2})\sin(|\delta|\sqrt{|\mu+|k'|^{2}|} )\right).
	\end{equation*}}
	Here, we have used $\left(\tfrac{1}{(\cdot)^{2}+\zeta}\right)^{\vee}=\tfrac{e^{-|\cdot|\sqrt{|\zeta|}}}{2\sqrt{|\zeta|}}$ when $\zeta>0$, and  $\left(\text{P.V.}\tfrac{1}{(\cdot)^{2}+\zeta}\right)^{\vee}=\tfrac{\sin\left(|\cdot|\sqrt{|\zeta|}\right)}{2\sqrt{|\zeta|}}$ when $\zeta<0$. The former is a classical conclusion regarding the Poisson kernel, and the latter  follows from the decomposition:
	\begin{equation*}
		\begin{aligned}
			\left(\text{P.V.}\tfrac{1}{(\cdot)^{2}+\zeta}\right)^{\vee} &= \tfrac{1}{2\sqrt{|\zeta|}} \left(\text{P.V.} \tfrac{1}{(\cdot)-\sqrt{|\zeta|}}-\text{P.V.}\tfrac{1}{(\cdot)+\sqrt{|\zeta|}}\right)^{\vee}\\
			&=\tfrac{i\text{ sgn}\left(\cdot\right)}{4\sqrt{|\zeta|}}  \left(e^{i\sqrt{|\zeta|}(\cdot)} -e^{-i\sqrt{|\zeta|}(\cdot)} \right)=\tfrac{\sin \left(|\cdot|\sqrt{|\zeta|}\right)}{2\sqrt{|\zeta|}}.
		\end{aligned}
	\end{equation*}
	
	For all $x_{n}\in\mathbb{R}$, $s\in\mathbb{R}$ and $g\in \mathscr{S}(\mathbb{R}^{n})$, we define
	\begin{equation}\label{oper}
		\left[U_{x_{n}}(s)g\right] (z) :=  \left(\int_{\mathbb{R}} \widehat{g}(k',\mu) \overline{M(k',x_{n}-z_{n},s,\mu)}\sqrt{\left|\mu+|k'|^{2}\right|}\, d\mu \right)^{\vee_{k'}} (z'),
	\end{equation}
	where $z=(z',z_{n})\in\mathbb{R}^{n-1}\times\mathbb{R}$. By Plancherel's identity combined with the variable substitution  $\mu=\eta|\eta|-|k'|^{2}$, we derive
	{\footnotesize	\begin{equation*}
			\begin{aligned}
				&\left\|U_{x_{n}}(s)g\right\|_{L^{2}(\mathbb{R}^{n})}\approx \left\|   \left\|\int_{\mathbb{R}} \widehat{g}(k',\mu) \overline{M(k',x_{n}-z_{n},s,\mu)}\sqrt{\left|\mu+|k'|^{2}\right|} \,d\mu \right\|_{L^{2}_{z_{n}}} \right\|_{L_{k'}^{2}}\\
				\approx &
				\left\|\left\|\int_{\mathbb{R}} \widehat{g}(k',\eta|\eta|-|k'|^{2}) e^{i\left(\eta|\eta|-|k'|^{2}\right)s}\left(-\raisebox{0.5ex}{$\chi$}_{\mathbb{R}_{+}} (\eta) e^{-|x_{n}-z_{n}||\eta|}+\raisebox{0.5ex}{$\chi$}_{\mathbb{R}_{-}}(\eta)\sin \left(\left|x_{n}-z_{n}\right|\left|\eta\right|\right)\right)|\eta|\, d\eta\right\|_{L^{2}_{z_{n}}} \right\|_{L_{k'}^{2}}.
			\end{aligned}
	\end{equation*}}
	Applying the $L^{2}$-boundedness of the Laplace transform (see \cite{kdv}) yields
	\begin{equation}\label{tass}
		\|U_{x_{n}}(s)g\|_{L^{2}(\mathbb{R}^{n})}\lesssim \left\|\widehat{g}\left(k',\eta|\eta|-|k'|^{2}\right)	\left|\eta\right|\right\|_{L^{2}_{k',\eta}(\mathbb{R}^{n})}=\|g\|_{\mathcal{H}^{0}(\mathbb{R}^{n})}.
	\end{equation}
	Then, we can define a bounded operator $U_{x_{n}}(s):\mathcal{H}^{0}(\mathbb{R}^{n})\rightarrow L^{2}(\mathbb{R}^{n})$,  as $\mathscr{S}(\mathbb{R}^{n})$ is dense in $\mathcal{H}^{0}(\mathbb{R}^{n})$. Furthermore, the adjoint operator $U_{x_{n}}(s)^{*}$  can be expressed explicitly as	
	\begin{equation}\label{adj}
		U_{x_{n}}(s)^{*}h= \left(\int_{\mathbb{R}} \widehat{h}^{\,y'}(k',y_{n}) M(k',x_{n}-y_{n},s,\mu)\,dy_{n}\right)^{\vee_{(k',\mu)}},
	\end{equation}
	where $h\in\mathscr{S}(\mathbb{R}^{n})$. We can then reformulate $I_{2}$ in \eqref{kee} as
	\begin{equation}
		I_{2}(x,t)=-\frac{1}{2\pi}\int_{\mathbb{R}}\left[U_{x_{n}}(\tau)^{*}f(\tau)\right](x',t)d\tau. \end{equation}
	From \eqref{oper}, \eqref{adj} and  the variable substitution $\mu=\eta|\eta|-|k'|^{2}$, we have
	\begin{equation*}
		\begin{aligned}
			&\left[U_{x_{n}}(s)U_{x_{n}}(\tau)^{*}h\right](z)\\
			=  & \left(\int_{\mathbb{R}}\int_{\mathbb{R}}   \widehat{h}^{y'}(k',y_{n}) M(k',x_{n}-y_{n},\tau,\mu)\,  \overline{M(k',x_{n}-z_{n},s,\mu)}\sqrt{\left|\mu+|k'|^{2}\right|} \,dy_{n}\,d\mu\right)^{\vee_{k'}}(z')\\
			=&\int_{\mathbb{R}^{n-1}}\int_{\mathbb{R}}\int_{\mathbb{R}}\int_{\mathbb{R}^{n-1}}h(y) e^{i(\eta|\eta|-|k'|^{2})(s-\tau)} e^{ik'\cdot(z'-y')}  \Big(-\raisebox{0.5ex}{$\chi$}_{\mathbb{R}_{+}}(\eta)e^{-|\eta|\left(|x_{n}-y_{n}|+|x_{n}-z_{n}|\right)}\\
			&+\raisebox{0.5ex}{$\chi$}_{\mathbb{R}_{-}}(\eta)\sin(|x_{n}-y_{n}||\eta|) \sin\left(|x_{n}-z_{n}||\eta|\right)\Big) dy'\,dy_{n}\,d\eta\,dk'.
		\end{aligned}
	\end{equation*}
	Applying Fubini's theorem and van der Corput's lemma \cite{mat} gives the estimate
	{\small		\begin{equation}\label{iconic}
			\begin{aligned} &\left\|[U_{x_{n}}(s)U_{x_{n}}(\tau)^{*}h]\right\|_{L^{\infty}(\mathbb{R}^{n})}
				\leq  \varlimsup_{m\rightarrow\infty}\int_{\mathbb{R}^{n}}|h(y)|\Big|\int_{[-m,m]^{n-1}}e^{-i|k'|^{2}(s-\tau)} e^{ik'\cdot(z'-y')}dk'\Big|\\
				\times&\left(\,\Big|\int_{0}^{m}e^{i\eta^{2}(s-\tau)} e^{-\eta\left(|x_{n}-y_{n}|+|x_{n}-z_{n}|\right)} d\eta\Big|
				+\Big|\int_{-m}^{0}e^{-i\eta^{2}(s-\tau)}\sin(|x_{n}-y_{n}|\eta)\sin(|x_{n}-z_{n}|\eta)d\eta\Big|\,\right)dy\\
				\lesssim &\,
				\frac{1}{|s-\tau|^{\frac{n}{2}}}\|h\|_{L^{1}(\mathbb{R}^{n})}.
			\end{aligned}
	\end{equation}}
	Combining the estimates \eqref{tass}, \eqref{iconic} with a density argument, we verify that the hypotheses of Lemma \ref{tas} are fulfilled. Consequently, for any $x_{n}\in\mathbb{R}$, we derive
	\begin{equation}
		\|I_{2}\|_{\mathcal{H}^{0}(\mathbb{R}^{n})}\lesssim \|f\|_{L^{\gamma'}(\mathbb{R};L^{\rho'}(\mathbb{R}^{n}))}.
	\end{equation}
\end{proof}

Building on Proposition \ref{peony}, we establish the following result for $s\in(0,2)$, which is also valid for arbitrary admissible pairs $(\gamma,\rho)$. The norm equivalence
\begin{equation}\label{shenjing}			\|f\|_{\mathcal{H}^{s}(\mathbb{R}^{n})}\approx\|f\|_{\mathcal{H}^{s-2}(\mathbb{R}^{n})}+\|\partial_{t}f\|_{\mathcal{H}^{s-2}(\mathbb{R}^{n})}+\|\Delta'f\|_{\mathcal{H}^{s-2}(\mathbb{R}^{n})},\quad s\in \mathbb{R}
\end{equation}
 is needed, which follows immediately from the standard properties of the Fourier transform.

\begin{proposition}
	For $s\in(0,2)$, the $w$ given by \eqref{thorny} satisfies
	\begin{equation}\label{juvenile}
		\|w\|_{C_{b}(\mathbb{R};\mathcal{H}^{s}(\mathbb{R}^{n}))}\lesssim \|f\|_{L^{\gamma'}(\mathbb{R};W^{s,\rho'}(\mathbb{R}^{n}))}+\|f\|_{B^{\frac{s}{2}}_{\gamma',2}(\mathbb{R};L^{\rho'}(\mathbb{R}^{n}))}.
	\end{equation}
\end{proposition}
\begin{proof}
	We first prove the following inequality
	\begin{equation}\label{bribery}	\|w\|_{C_{b}(\mathbb{R};\mathcal{H}^{2}(\mathbb{R}^{n}))}\lesssim\|f\|_{L^{\gamma'}(\mathbb{R};W^{2,\rho'}(\mathbb{R}^{n}))}+\|f\|_{W^{1,\gamma'}(\mathbb{R};L^{\rho'}(\mathbb{R}^{n}))}.
	\end{equation}
	By a density argument, it suffices to consider $f\in C_{c}^{\infty}(\mathbb{R}^{n+1})$. Clearly,  there exists $a>0$ such that $f(\cdot,t)=0$ for $t\in\mathbb{R}\setminus(-a,a)$. We decompose \eqref{thorny} as
	\begin{equation*}
		\begin{aligned}
			w(t)&=i\int_{-a}^{0}e^{i(t-\tau)\Delta} f(\tau) d\tau-i\int_{-a}^{t}e^{i(t-\tau)\Delta}f(\tau)d\tau \\
			&\equiv w_{1}(t)-w_{2}(t).
		\end{aligned}
	\end{equation*}
	From \eqref{circuitous} and \eqref{utter}, we derive
	\begin{equation*}  \|w_{1}\|_{C_{b}(\mathbb{R};\mathcal{H}^{2}(\mathbb{R}^{n}))}\lesssim \left\|\int_{-a}^{0} e^{-i\tau\Delta} f(\tau)d\tau\right\|_{H^{2}(\mathbb{R}^{n})}\lesssim \|f\|_{L^{\gamma'}(\mathbb{R};W^{2,\rho'}(\mathbb{R}^{n}) )},
	\end{equation*}
	where we have used the fact that the adjoint operator of $e^{i\tau\Delta}$ is $e^{-i\tau\Delta}$. Since $f(-a)=0$, we have $\partial_{t}w_{2}(t)=i\int_{-a}^{t}e^{i(t-\tau)\Delta}\partial_{t}f(\tau)d\tau$. By \eqref{cope}, it follows that
	\begin{equation*}
		\|w_{2}\|_{C_{b}(\mathbb{R};\mathcal{H}^{0}(\mathbb{R}^{n}))}+\|\partial_{t}w_{2}\|_{C_{b}(\mathbb{R};\mathcal{H}^{0}(\mathbb{R}^{n}))}\lesssim \|f\|_{W^{1,\gamma'}(\mathbb{R};L^{\rho'}(\mathbb{R}^{n}))},
	\end{equation*}
	and similarly, $\|\Delta'w_{2}\|_{C_{b}(\mathbb{R};\mathcal{H}^{0}(\mathbb{R}^{n}))}\lesssim \|f\|_{L^{\gamma'}(\mathbb{R};W^{2,\rho'}(\mathbb{R}^{n}))}$.
	Therefore, \eqref{bribery} holds for $w_{2}$  via \eqref{shenjing}. Finally, \eqref{juvenile} follows by interpolation between the estimates \eqref{cope} and \eqref{bribery}.
\end{proof}	
\begin{remark}
	Especially, for the case $s=1$, we have
	\begin{equation}\label{crisp}
		\|w\|_{C_{b}^{1}(\mathbb{R};\mathcal{H}^{s-1} (\mathbb{R}^{n}) )}\lesssim \|f\|_{L^{\gamma'}(\mathbb{R};W^{s,\rho'}(\mathbb{R}^{n}))}.
	\end{equation}
	This result is obtained by differentiating \eqref{thorny} with respect to $x_{n}$  and then applying \eqref{cope}.
\end{remark}

Using the zero extension property, we establish the following  fundamental inequalities.
\begin{proposition}\label{tardy}
	For $s\in\left[0,\frac{1}{2}\right)$, the $w$ given by \eqref{thorny} satisfies
	\begin{equation}\label{covet}		\|w\|_{C_{b}(\mathbb{R};\mathcal{H}^{s}(\mathbb{R}^{n}))}\lesssim \|f\|_{L^{1}(\mathbb{R};H^{s}(\mathbb{R}^{n}))}.
	\end{equation}
\end{proposition}
\begin{proof}
	We define
	\begin{equation*}
		G(t,\tau):=\left\{\begin{aligned}&e^{i(t-\tau)\Delta}f(\tau), \quad 0\leq\tau<t, \\
			&-e^{i(t-\tau)\Delta}f(\tau),\quad t\leq\tau<0,  \\
			&0,\quad \text{otherwise},
		\end{aligned}\right.
	\end{equation*}	
	then $w(t)=-i\int_{\mathbb{R}}G(t,\tau)d\tau$ for $t\in \mathbb{R}$. Since $\forall t, \tau\in\mathbb{R}$, $e^{i(t-\tau)\Delta}f(\tau)\in H^{s}(\mathbb{R}^{n})$, we can naturally define $\left[e^{i(t-\tau)\Delta}f(\tau)\right](x',x_{n})$, and $G(t,\tau)$ can be naturally written as $G(x',x_{n},t,\tau)$. Then for any $x_{n}\in\mathbb{R}$, we have
	\begin{equation}\label{jie}			\|w(x',x_{n},t)\|_{\mathcal{H}_{x',t}^{s}(\mathbb{R}^{n})}\leq \int_{\mathbb{R}}\left\|G(x',x_{n},t,\tau)\right\|_{\mathcal{H}_{x',t}^{s}\left(\mathbb{R}^{n}\right)}d\tau.
	\end{equation}
	Note that for any $x_{n}\in \mathbb{R}$ and $\tau\geq0$, $G(x',x_{n},t,\tau)$ is a zero extension of $[e^{i(t-\tau)\Delta}f(\tau)](x',x_{n})$ to the domain $\{(x',t)\in \mathbb{R}^{n-1}\times(\tau,\infty)\}$. When $0\leq s<\frac{1}{2}$, by the boundedness of zero extension shown in the Corollary 2.4 (2) of \cite{facilitate}, we have
	\begin{equation*}
		\begin{aligned}		\|G(\cdot,x_{n},\cdot,\tau)\|_{\mathcal{H}^{s}(\mathbb{R}^{n})}\approx		\left\|\left[e^{i(\cdot-\tau)\Delta}f(\tau)\right](\cdot,x_{n})\right\|_{\mathcal{H}^{s}(\mathbb{R}^{n-1}\times(\tau,\infty))}=\left\|\left[e^{i(\cdot)\Delta}f(\tau)\right](\cdot,x_{n})\right\|_{\mathcal{H}^{s}(\mathbb{R}^{n}_{+})},
		\end{aligned}
	\end{equation*}
	a similar result holds for $\tau<0$, hence
	\begin{equation}\label{zhi}
		\|G(\cdot,x_{n},\cdot,\tau)\|_{\mathcal{H}^{s}(\mathbb{R}^{n})}\lesssim \left\|\left[e^{i(\cdot)\Delta} f(\tau)\right](\cdot,x_{n})\right\|_{\mathcal{H}^{s}(\mathbb{R}^{n})}.
	\end{equation}
	By using \eqref{jie}--\eqref{zhi} and inequality \eqref{circuitous}, we have
	\begin{equation}\label{xx}
		\begin{aligned}			\|w(\cdot,x_{n},\cdot)\|_{\mathcal{H}^{s}(\mathbb{R}^{n})}\lesssim\int_{\mathbb{R}}\left\|\left[e^{i(\cdot)\Delta}f(\tau)\right](\cdot,x_{n})\right\|_{\mathcal{H}^{s}(\mathbb{R}^{n})}d\tau			\lesssim\|f\|_{L^{1}(\mathbb{R};H^{s}(\mathbb{R}^{n}))}.
		\end{aligned}
	\end{equation}
\end{proof}
\begin{corollary}
	For $s\in\left(2,\frac{5}{2}\right)$,  the $w$ given by \eqref{thorny} satisfies
	\begin{equation}\label{perilous} \|w\|_{C_{b}(\mathbb{R};\mathcal{H}^{s}(\mathbb{R}^{n}))}\lesssim\|f\|_{L^{1}(\mathbb{R};H^{s}(\mathbb{R}^{n}))}+\|f\|_{W^{1,1}(\mathbb{R};H^{s-2}(\mathbb{R}^{n}))}.
	\end{equation}
\end{corollary}

In the study of global well-posedness for the $n$-dimensional NLS equation, we consider constructing Gronwall-type inequalities. For this purpose, the derivation of the estimates presented below is necessary, with \eqref{cope} and \eqref{covet} providing the basis for this derivation.
\begin{proposition}\label{p33}
	For $s\in\left[0,\frac{3}{2}\right]$, $1\leq \alpha,\beta\leq 2$ satisfing $\frac{1}{\alpha}+\frac{n}{2\beta}=1+\frac{3n-4s}{12}$,  the $w$ given by \eqref{thorny} satisfies
	\begin{equation}\label{taowuji}
		\|w\|_{C_{b}(\mathbb{R};\mathcal{H}^{s}(\mathbb{R}^{n-1}\times(0,T)))}\lesssim\langle T\rangle^{\frac{s}{3}}\|f\|_{L^{\alpha}(0,T;W^{s,\beta}(\mathbb{R}^{n}))}.
	\end{equation}
	In addition, when  $n=2$ and $s\neq\frac{3}{2}$, we require $\left(\alpha,\beta\right)\neq \left(2,\frac{6n}{3n+6-4s}\right)$.
\end{proposition}

Establishing Proposition \ref{p33} for $s=\frac{3}{2}$ requires norm equivalence \eqref{sj}. However, in high dimensions, we lack a theory for the space $\mathcal{H}^{s}$ with negative index. To this end, we connect it to the one-dimensional estimates, which is achieved through a partial Fourier transform, reducing the expression containing the $n$-dimensional Schr\"{o}dinger group to that involving the one-dimensional Schr\"{o}dinger group \eqref{star}. Let us recall the Kato smoothing inequality \eqref{o}. Next, we use \eqref{o} to derive the estimate \eqref{faction}, thereby proving Proposition \ref{p33}. When $t\geq0$, we define
\begin{equation*}
	g(t,\tau):=\left\{\begin{aligned}&e^{i(t-\tau)\partial_{x}^{2}}f(\tau),  \quad 0\leq\tau <t, \\&0, \quad  0\leq t\leq \tau, \end{aligned}\right.
\end{equation*}
clearly, $\int_{0}^{t} e^{i(t-\tau)\partial_{x}^{2}}f(\tau)d\tau=\int_{0}^{\infty}g(t,\tau)d\tau$.
Note that the solution of \eqref{inh} can be given by $-i\int_{0}^{t} e^{i(t-\tau)\partial_{x}^{2}}f(\tau)d\tau$. Then from \eqref{o}, we have
\begin{equation*}
	\begin{aligned}
		\left\|\int_{0}^{t} e^{i(t-\tau)\partial_{x}^{2}}f(\tau)d\tau\right\|_{L_{t}^{2}(\mathbb{R}_{+})}
		&\leq \int_{0}^{\infty}\|g(t,\tau)\|_{L_{t}^{2}(\mathbb{R}_{+})}d\tau	=\int_{0}^{\infty}\left\|e^{it\partial_{x}^{2}}f(\tau)\right\|_{L_{t}^{2}(\mathbb{R}_{+})}d\tau\\
		&\leq \int_{0}^{\infty}\|f(\tau)\|_{\dot{H}^{-\frac{1}{2}}(\mathbb{R})}d\tau,
	\end{aligned}
\end{equation*}
and similarly for the case of $L_{t}^{2}(\mathbb{R}_{-})$, thus we obtain
\begin{equation}\label{faction}
	\left\|-i\int_{0}^{t} e^{i(t-\tau)\partial_{x}^{2}}f(\tau)d\tau\right\|_{L_{t}^{2}(\mathbb{R})}\lesssim\|f\|_{L^{1}(\mathbb{R};\dot{H}^{-\frac{1}{2}}(\mathbb{R}))}.
\end{equation}	

\begin{proof}[Proof of Proposition \ref{p33}]
	Let $\boldsymbol{f}(x,t):=\left\{\begin{aligned}&f(x,t), \quad t\in (0,T)\\
		&0, \quad t\in\mathbb{R}\setminus (0,T)\end{aligned}\right.$, i.e., $\boldsymbol{f}$ is the zero extension of $f$. Denote $\boldsymbol{w}(t):=-i \int_{0}^{t}e^{i(t-\tau)\Delta}\boldsymbol{f}(\tau)d\tau$, clearly, $\boldsymbol{w}(t)=w(t)$ for $t\in (0,T)$. In order to prove \eqref{taowuji}, we first consider $s=\frac{3}{2}$, where the only case is $(\alpha,\beta)=(2,2)$. According to \eqref{shenjing}, we have
	{\small\begin{equation}\label{sj} \|\boldsymbol{w}(x',x_{n},t)\|_{\mathcal{H}^{\frac{3}{2}}(\mathbb{R}^{n})}\approx\|\boldsymbol{w}(x',x_{n},t)\|_{\mathcal{H}^{-\frac{1}{2}}(\mathbb{R}^{n})}+\|\partial_{t}\boldsymbol{w}(x',x_{n},t)\|_{\mathcal{H}^{-\frac{1}{2}}(\mathbb{R}^{n})}+\|\Delta'\boldsymbol{w}(x',x_{n},t)\|_{\mathcal{H}^{-\frac{1}{2}}(\mathbb{R}^{n})}.
	\end{equation}}
	It follows from \eqref{cope} that
	\begin{equation}\label{fir} \|\boldsymbol{w}\|_{L^{\infty}(\mathbb{R};\mathcal{H}^{-\frac{1}{2}}(\mathbb{R}^{n}))}\leq \|\boldsymbol{w}\|_{L^{\infty}\left(\mathbb{R};\mathcal{H}^{0}(\mathbb{R}^{n})\right)}\lesssim\|\boldsymbol{f}\|_{L^{1}(\mathbb{R};L^{2}(\mathbb{R}^{n}))}.
	\end{equation}
	For $\partial_{t}\boldsymbol{w}$, we have
	\begin{equation}\label{wf}
		i\partial_{t}\boldsymbol{w}=-\Delta \boldsymbol{w}+\boldsymbol{f}.
	\end{equation}
	Next, we will introduce \eqref{qian} to estimate $\boldsymbol{f}$. By trace theorem,
	\begin{equation} \label{qian}
		\begin{aligned}
			\left \| \|\boldsymbol{f}(x',x_{n},t)\|_{L_{x'}^{2}(\mathbb{R}^{n-1})}  \right\|_{L_{t}^{2}(\mathbb{R})} &= \left \| \|\boldsymbol{f}(x',0+x_{n},t)\|_{L_{x'}^{2}(\mathbb{R}^{n-1})}  \right\|_{L_{t}^{2}(\mathbb{R})}\\&\lesssim   \left\|\|\boldsymbol{f}(\cdot+x_{n}e_{n},t)\|_{H^{\frac{1}{2}+\varepsilon} (\mathbb{R}^{n})}\right\|_{L_{t}^{2}(\mathbb{R})}\leq \|\boldsymbol{f}\|_{L^{2}(\mathbb{R};H^{\frac{3}{2}}(\mathbb{R}^{n}))} .
		\end{aligned}
	\end{equation}
	From \eqref{blackmail}, the embedding   $L^{2}(\mathbb{R}^{n})\hookrightarrow\mathcal{H}^{-\frac{1}{2}}(\mathbb{R}^{n})$ follows, together with \eqref{qian}, we get
	\begin{equation}\label{batter}
		\begin{aligned}	\left\|\|\boldsymbol{f}(\cdot,x_{n},\cdot)\|_{\mathcal{H}^{-\frac{1}{2}}(\mathbb{R}^{n})}\right\|_{L_{x_{n}}^{\infty}(\mathbb{R})}\lesssim\|\boldsymbol{f}\|_{L^{2}(\mathbb{R};H^{\frac{3}{2}}(\mathbb{R}^{n}))}.
		\end{aligned}
	\end{equation}
	
	In order to estimate $\Delta \boldsymbol{w}$ and $\Delta' \boldsymbol{w}$, we consider
	\begin{equation}\label{pw}
		\partial_{j}^{2} \boldsymbol{w} =-i\int_{0}^{t}e^{i(t-\tau)\Delta} \partial_{j}^{2}\boldsymbol{f}(\tau)d\tau,\quad j=1,\cdots,n.
	\end{equation}
	It suffices to give the estimate for one case, for example, for $\partial_{n}^{2}\boldsymbol{w}$. The estimate for other cases is analogous.
	Taking the Fourier transform to the both sides of the equality \eqref{pw} with respect to $x'$ gives
	\begin{equation}\label{star}	\widehat{\partial_{n}^{2}\boldsymbol{w}}^{x'}(\xi',x_{n},t)=-i\int_{0}^{t}e^{-i|\xi'|^{2}(t-\tau)}\left[e^{i(t-\tau)\partial_{x}^{2}}\widehat{\partial_{n}^{2}\boldsymbol{f}}^{x'}(\xi',\cdot,\tau)\right](x_{n}) d\tau.
	\end{equation}
	According to the one-dimensional estimate \eqref{faction}, we get
	\begin{equation*} \left\|\widehat{\partial_{n}^{2}\boldsymbol{w}}^{x'}(\xi',x_{n},\cdot)\right\|_{L^{2}(\mathbb{R})}\lesssim \int_{\mathbb{R}} \left\|\widehat{\partial_{n}^{2}\boldsymbol{f}}^{x'}(\xi',\cdot,\tau)\right\|_{\dot{H}^{-\frac{1}{2}}(\mathbb{R})}d\tau\leq\int_{\mathbb{R}} \left\|\widehat{\boldsymbol{f}}^{x'}(\xi',\cdot,\tau)\right\|_{\dot{H}^{\frac{3}{2}}(\mathbb{R})}d\tau.
	\end{equation*}
	Using  Plancherel's identity and Minkowski's integral inequality, we obtain
	\begin{equation*}		\left\|\partial_{n}^{2}\boldsymbol{w}(\cdot,x_{n},\cdot)\right\|_{L^{2}(\mathbb{R}^{n})}\lesssim\int_{\mathbb{R}}\left\|\left\|\widehat{\boldsymbol{f}}^{x'}(\xi',\cdot,\tau)\right\|_{\dot{H}^{\frac{3}{2}}(\mathbb{R})}\right\|_{L^{2}_{\xi'}(\mathbb{R}^{n-1})}d\tau\leq\int_{\mathbb{R}}\|\boldsymbol{f}(\tau)\|_{H^{\frac{3}{2}}(\mathbb{R}^{n})}d\tau,
	\end{equation*}
	thus, we have
	\begin{equation}\label{x} \left\|\partial_{j}^{2}\boldsymbol{w}(\cdot,x_{n},\cdot)\right\|_{L^{2}(\mathbb{R}^{n})}\leq\|\boldsymbol{f}\|_{L^{1}(\mathbb{R};H^{\frac{3}{2}}(\mathbb{R}^{n}))},\quad j=1,\cdots,n.
	\end{equation}
	From $L^{2}(\mathbb{R}^{n})\hookrightarrow\mathcal{H}^{-\frac{1}{2}}(\mathbb{R}^{n})$ and \eqref{x}, we get
	\begin{equation}\label{del}		\|\Delta'\boldsymbol{w}(\cdot,x_{n},\cdot)\|_{\mathcal{H}^{-\frac{1}{2}}(\mathbb{R}^{n})}+\|\Delta\boldsymbol{w}(\cdot,x_{n},\cdot)\|_{\mathcal{H}^{-\frac{1}{2}}(\mathbb{R}^{n})}\lesssim\|\boldsymbol{f}\|_{L^{1}(\mathbb{R};H^{\frac{3}{2}}(\mathbb{R}^{n}))}.
	\end{equation}
	Combining \eqref{wf}, \eqref{batter} and \eqref{del}, we obtain
	\begin{equation}\label{proficient} \left\|\partial_{t}\boldsymbol{w}\right\|_{L^{\infty}(\mathbb{R};\mathcal{H}^{-\frac{1}{2}}(\mathbb{R}^{n}))}+\left\|\Delta'\boldsymbol{w}\right\|_{L^{\infty}(\mathbb{R};\mathcal{H}^{-\frac{1}{2}}(\mathbb{R}^{n}))}\lesssim  \|\boldsymbol{f}\|_{L^{1}(\mathbb{R};H^{\frac{3}{2}}(\mathbb{R}^{n}))}+\|\boldsymbol{f}\|_{L^{2}(\mathbb{R};H^{\frac{3}{2}}(\mathbb{R}^{n}))}.
	\end{equation}
	By substituting \eqref{fir} and \eqref{proficient} into \eqref{sj}, we have
	\begin{equation}\label{32}
		\begin{aligned}
			\|w\|_{L^{\infty}(\mathbb{R};\mathcal{H}^{\frac{3}{2}}(\mathbb{R}^{n-1}\times(0,T)))}\leq \|\boldsymbol{w}\|_{L^{\infty}(\mathbb{R};\mathcal{H}^{\frac{3}{2}}(\mathbb{R}^{n}))}\lesssim&\,\|\boldsymbol{f}\|_{L^{1}(\mathbb{R};H^{\frac{3}{2}}(\mathbb{R}^{n}))}+\|\boldsymbol{f}\|_{L^{2}(\mathbb{R};H^{\frac{3}{2}}(\mathbb{R}^{n}))}\\
			\lesssim&\,\langle T\rangle^{\frac{1}{2}} \|f\|_{L^{2}(0,T;H^{\frac{3}{2}}(\mathbb{R}^{n}))}.
		\end{aligned}
	\end{equation}
	Therefore, we complete the proof of \eqref{taowuji} for $s=\frac{3}{2}$. The estimate \eqref{taowuji} is then  obtained by  interpolating between \eqref{32} and \eqref{cope}.
\end{proof}

When $s>\frac{3}{2}$, $\|w\|_{C_{b}(\mathbb{R};\mathcal{H}^{s}(\mathbb{R}^{n-1}\times(0,T)))}$ is controlled by a larger norm, which involves time derivative of $f$. The proposition stated below is guaranteed by the key embedding~\eqref{embe}, obtained via a frequency-space analysis.
\begin{proposition}\label{tart}
	For $s\in\left(\frac{3}{2},\frac{5}{2}\right)$, the $w$ given by \eqref{thorny} satisfies
	\begin{equation}\label{gender} \|w\|_{C_{b}(\mathbb{R};\mathcal{H}^{s}(\mathbb{R}^{n-1}\times(0,T)))}\lesssim   \left(\langle T\rangle^{\frac{1}{2}}+T^{-\frac{2s-3}{4}}\right)\left(\|f\|_{L^{2}(0,T;H^{s}(\mathbb{R}^{n}))}+\|f\|_{H^{1}(0,T;H^{s-2}(\mathbb{R}^{n}))}\right).
	\end{equation}
\end{proposition}
\begin{proof}
	By the Sobolev extension theory \cite{leo}, we can choose $\boldsymbol{f}$ such that $\boldsymbol{f}=f$ in $(0,T)$, $\text{supp} (\boldsymbol{f})\subset(-T,2T)$ and the following estimates hold:
	\begin{equation}\label{y1}
		\|\boldsymbol{f}\|_{L^{2}(\mathbb{R};H^{s}(\mathbb{R}^{n}))}\lesssim\|f\|_{L^{2}(0,T;H^{s}(\mathbb{R}^{n}))},
	\end{equation}
	\begin{equation}\label{y11}
		\|\boldsymbol{f}\|_{H^{1}(\mathbb{R};H^{s-2}(\mathbb{R}^{n}))}\lesssim \left(1 + T^{-1}\right)\|f\|_{H^{1}(0,T;H^{s-2}(\mathbb{R}^{n}))}.
	\end{equation}
	We denote
	\begin{equation*}
		\boldsymbol{w}(t):=-i \int_{0}^{t}e^{i(t-\tau)\Delta}\boldsymbol{f}(\tau)d\tau,
	\end{equation*}
	then $\boldsymbol{w}(t)=w(t)$ for $t\in (0,T)$. From \eqref{shenjing}, we need to consider some $\mathcal{H}^{s-2}(\mathbb{R}^{n})$ estimates. First, from \eqref{cope} and \eqref{covet}, we can obtain the estimate
	\begin{equation}\label{w1} \|\boldsymbol{w}\|_{L^{\infty}(\mathbb{R};\mathcal{H}^{s-2}(\mathbb{R}^{n}))}\leq\|\boldsymbol{w}\|_{L^{\infty}(\mathbb{R};\mathcal{H}^{0}(\mathbb{R}^{n}))}\lesssim\|\boldsymbol{f}\|_{L^{1}(\mathbb{R};H^{s}(\mathbb{R}^{n}))}\lesssim T^{\frac{1}{2}} \|\boldsymbol{f}\|_{L^{2}(\mathbb{R};H^{s}(\mathbb{R}^{n}))}
	\end{equation}
	holds for $s\in\left(\frac{3}{2},2\right)$ and
	\begin{equation}\label{w2} \|\boldsymbol{w}\|_{L^{\infty}(\mathbb{R};\mathcal{H}^{s-2}(\mathbb{R}^{n}))}\lesssim\|\boldsymbol{f}\|_{L^{1}(\mathbb{R};H^{s}(\mathbb{R}^{n}))}
	\end{equation}
	holds for $s\in \left[2,\frac{5}{2}\right)$.
	
	Next, for $\partial_{t}\boldsymbol{w}$ we have
	\begin{equation}\label{par}
		i\partial_{t}\boldsymbol{w}=-\Delta \boldsymbol{w}+\boldsymbol{f}.
	\end{equation}
	We introduce the space $H^{\alpha,\beta}(\mathbb{R}^{n-1}\times \mathbb{R})$ defined by the norm
	\begin{equation}
		\|f\|_{H^{\alpha,\beta}(\mathbb{R}^{n-1}\times \mathbb{R})}:=\left\|\left(\langle \xi \rangle^{\alpha}+\langle\eta\rangle^{\beta}\right)\widehat{f}(\xi,\eta)\right\|_{L^{2}(\mathbb{R}^{n-1}_{\xi}\times\mathbb{R}_{\eta})}.
	\end{equation}
	We first deduce two embeddings \eqref{emb} and \eqref{embe} to estimate $\boldsymbol{f}$. Let $(\xi', \xi_{n})\in \mathbb{R}^{n-1}\times\mathbb{R}$. By using the fact that
	\begin{equation*} \left||\xi'|^{2}+\xi_{n}\right|^{\frac{1}{2}}\leq\left(\langle\xi'\rangle^{2}+\langle\xi_{n}\rangle\right)^{\frac{1}{2}} \ \text{and}\ \left(1+|\xi'|^2+|\xi_n|\right)^s\approx \left(\langle\xi'\rangle^{2}+\langle\xi_{n}\rangle\right)^{s},
	\end{equation*}
	we obtain
	\begin{equation} \label{emb} H^{s+\frac{1}{2},\frac{2s+1}{4}}(\mathbb{R}^{n-1}\times\mathbb{R})\hookrightarrow\mathcal{H}^{s}(\mathbb{R}^{n}).
	\end{equation}
	When $s\in\left(\frac{3}{2},\frac{5}{2}\right)$, both $\frac{3}{6-2s}$ and $\frac{3}{2s-3}$ are belong to $(1,\infty)$, and their reciprocal sum equal to $1$. Then an application of the Young's equality implies
	\begin{equation*}		\langle\xi_{n}\rangle^{3-s}\langle\xi'\rangle^{s-\frac{3}{2}}\leq\frac{6-2s}{3}\langle\xi_{n}\rangle^{\frac{3}{2}}+
		\frac{2s-3}{3}\langle\xi'\rangle^{\frac{3}{2}}.
	\end{equation*}
	Clearly, $\langle\xi_{n}\rangle^{3-s}\leq\langle\xi_{n}\rangle^{\frac{3}{2}}$, so we have
	\begin{equation*}	\langle\xi_{n}\rangle^{3-s}\left(\langle\xi'\rangle^{s-\frac{3}{2}}+\langle\eta\rangle^{\frac{2s-3}{4}}\right)\lesssim\langle\eta\rangle^{\frac{2s-3}{4}}\left(\langle\xi'\rangle^{\frac{3}{2}}+\langle\xi_{n}\rangle^{\frac{3}{2}}\right)\approx\langle\eta\rangle^{\frac{2s-3}{4}}\langle\xi\rangle^{\frac{3}{2}},
	\end{equation*}
	where $\eta\in\mathbb{R}$. Based on these inequalities above we can deduce the second embedding
	\begin{equation}\label{embe}	H_{t}^{\frac{2s-3}{4}}(\mathbb{R};H_{x}^{\frac{3}{2}}(\mathbb{R}^{n}))\hookrightarrow		H^{3-s}_{x_{n}}(\mathbb{R};H^{s-\frac{3}{2},\frac{2s-3}{4}}(\mathbb{R}_{x'}^{n-1}\times\mathbb{R}_{t})).
	\end{equation}
	From \eqref{emb} and \eqref{embe}, we obtain
	\begin{equation}\label{fallacious}	
		\begin{aligned}	\|\boldsymbol{f}(\cdot,x_{n},\cdot)\|_{\mathcal{H}^{s-2}(\mathbb{R}^{n})}\lesssim&\,\|\boldsymbol{f}(\cdot,x_{n},\cdot)\|_{H^{s-\frac{3}{2},\frac{2s-3}{4}}(\mathbb{R}^{n-1}\times\mathbb{R})}\lesssim\|\boldsymbol{f}\|_{L_{x_{n}}^{\infty}(\mathbb{R};H^{s-\frac{3}{2},\frac{2s-3}{4}}(\mathbb{R}^{n-1}\times\mathbb{R}))}\\ \lesssim&\,\|\boldsymbol{f}\|_{H^{3-s}(\mathbb{R};H^{s-\frac{3}{2},\frac{2s-3}{4}}(\mathbb{R}^{n-1}\times\mathbb{R}))}\lesssim\|\boldsymbol{f}\|_{H^{\frac{2s-3}{4}}(\mathbb{R};H^{\frac{3}{2}}(\mathbb{R}^{n}))}\\			\lesssim&\,\|\boldsymbol{f}\|_{L^{2}(\mathbb{R};H^{s}(\mathbb{R}^{n}))}^{\frac{7-2s}{4}}\|\boldsymbol{f}\|_{H^{1}(\mathbb{R};H^{s-2}(\mathbb{R}^{n}))}^{\frac{2s-3}{4}}.
		\end{aligned}
	\end{equation}
	Here, the last inequality follows from the fact
	\begin{equation*}
		\begin{aligned}
			&\left\|\langle\xi\rangle^{\frac{3}{2}} \langle\eta\rangle^{\frac{2s-3}{4}} |\widehat{\boldsymbol{f}}(\xi,\eta)|\right\|_{L^{2}(\mathbb{R}^{n+1})}\\
			\leq&\left\|\langle\xi\rangle^{\frac{s(7-2s)}{4}} |\widehat{\boldsymbol{f}}(\xi,\eta)|^{\frac{7-2s}{4}}\right\|_{L^{\frac{8}{7-2s}}(\mathbb{R}^{n+1})}  \left\|\langle\xi\rangle^{\frac{(s-2)(2s-3)}{4}}\langle\eta\rangle^{\frac{2s-3}{4}} |\widehat{\boldsymbol{f}}(\xi,\eta)|^{\frac{2s-3}{4}}\right\|_{L^{\frac{8}{2s-3}}(\mathbb{R}^{n+1})}.
		\end{aligned}
	\end{equation*}
	
	It follows from \eqref{cope} and \eqref{covet} that
	\begin{equation}\label{cha}			\|\Delta'\boldsymbol{w}(\cdot,x_{n},\cdot)\|_{\mathcal{H}^{s-2}(\mathbb{R}^{n})}+ \|\Delta\boldsymbol{w}(\cdot,x_{n},\cdot)\|_{\mathcal{H}^{s-2}(\mathbb{R}^{n})}\lesssim\|\boldsymbol{f}\|_{L^{1}(\mathbb{R};H^{s}(\mathbb{R}^{n}))}
	\end{equation}
	holds for $s\in \left[2,\frac{5}{2}\right)$. The inequality \eqref{del} implies that \eqref{cha} remains valid for $s=\frac{3}{2}$, and by interpolation, it extends to $s\in \left(\frac{3}{2},2\right)$. Combining with \eqref{par}, \eqref{fallacious}--\eqref{cha}, and  the support condition $\text{supp} (\boldsymbol{f})\subset(-T,2T)$, we obtain
	\begin{equation}\label{chacha}
		\begin{aligned}			&\left\|\Delta'\boldsymbol{w}\right\|_{L^{\infty}(\mathbb{R};\mathcal{H}^{s-2}(\mathbb{R}^{n}))}+\left\|\partial_{t}\boldsymbol{w}\right\|_{L^{\infty}(\mathbb{R};\mathcal{H}^{s-2}(\mathbb{R}^{n}))}\\
			\lesssim &\ T^{\frac{1}{2}}\left\|\boldsymbol{f}\right\|_{L^{2}(\mathbb{R};H^{s}(\mathbb{R}^{n}))}+\left\|\boldsymbol{f}\right\|_{L^{2}(\mathbb{R};H^{s}(\mathbb{R}^{n}))}^{\frac{7-2s}{4}} \left\|\boldsymbol{f}\right\|_{H^{1}(\mathbb{R};H^{s-2}(\mathbb{R}^{n}))}^{\frac{2s-3}{4}}.
		\end{aligned}
	\end{equation}
	Finally, combining \eqref{shenjing},  \eqref{w1}--\eqref{w2}, \eqref{chacha} and \eqref{y1}--\eqref{y11}, we establish the desired estimate \eqref{gender}.
\end{proof}

\subsection{Reduced initial-boundary value problem}\label{sec23}
Finally, we study the reduced IBVP
\begin{equation}\label{locust}
	\left\{\begin{aligned}
		&	i\partial_{t}\phi+\Delta \phi=0, \quad (x, t) \in \mathbb{R}_{+}^{n} \times \mathbb{R}_{+}, \\
		&\phi(x,0)=0, \quad x\in\mathbb{R}^{n}_{+},\\& \phi(x',0,t)=g_{0}(x',t), \quad  (x',t)\in \mathbb{R}^{n-1}\times \mathbb{R}_{+}.
	\end{aligned}
	\right.
\end{equation}
In the following, we  employ the Fokas method to derive the solution formula for \eqref{locust} and establish the Strichartz estimates. The Fokas method is a novel approach for solving IBVP  for linear and integrable nonlinear PDEs, it was discovered by Fokas in 1997 \cite{Fokas} in his request to generalize inverse scattering method and was further developed by \cite{Fokas1, Fok}. This method has been successfully applied to establish well-posedness for both the one-dimensional  and  two-dimensional NLS equations \cite{dew,Himonas}.  As noted in \cite{facilitate}, Strichartz estimates for the $n$-dimensional problem have not been obtained via this method; in this work, we  provide these estimates.

For $\phi$ sufficiently smooth and sufficiently decaying as $x\rightarrow\infty$, the half space Fourier transform is defined by
\begin{equation}\label{torture}
	\mathscr{F}[\phi](\zeta,k,t):= \int_{\mathbb{R}^{n-1}}\int_{\mathbb{R}_{+}} \phi(x',x_{n},t) e^{-i\zeta\cdot x'-ik x_{n}} dx_{n}\,dx',\quad \zeta\in\mathbb{R}^{n-1},\ \text{Im} k\leq 0,
\end{equation}
where	$\mathscr{F}[\phi]\,(\zeta,\cdot,t)$ is analytic in $\mathbb{C}^{-}$ and continuous up to  $\mathbb{R}$.  Apply the half space Fourier transform to the equality in \eqref{locust}. After integrating by parts twice with respect to $x\in\mathbb{R}_{+}^{n}$ and integrating in time over $(0,t)$, we obtain
\begin{equation*} ie^{i(|\zeta|^{2}+k^{2})t}\mathscr{F}[\phi]\,(\zeta,k,t)=\int_{0}^{t}e^{i(|\zeta|^{2}+k^{2})\tau} \left(ik\widehat{\phi}^{\,x'}\,(\zeta,0,\tau) +\widehat{\phi}^{\,x'}_{x_{n}}(\zeta,0,\tau) \right) d\tau.
\end{equation*}
We then obtain the global relation
\begin{equation}\label{hawker}
	ie^{i(|\zeta|^{2}+k^{2})t}\mathscr{F}[\phi]\,(\zeta,k,t)=ik \widetilde{g}_{0}\left(\zeta,i(|\zeta|^{2}+k^{2}),t\right)+\widetilde{g}_{1}\left(\zeta,i(|\zeta|^{2}+k^{2}),t\right),\quad k\in\overline{\mathbb{C}^{-}},
\end{equation}
where\begin{equation}\label{251}
	\widetilde{g}_{0}(\zeta,k,t)=\int_{0}^{t} e^{k\tau} \widehat{\phi}^{\,x'}(\zeta,0,\tau)d\tau,\quad\widetilde{g}_{1}(\zeta,k,t)=\int_{0}^{t} e^{k\tau} \widehat{\phi}^{\,x'}_{x_{n}}(\zeta,0,\tau)d\tau.
\end{equation}

Taking the inverse Fourier transform of the global relation \eqref{hawker}, we obtain
\begin{equation} \label{253} \phi(x,t)=\frac{-i}{(2\pi)^{n}}\int_{\mathbb{R}^{n-1}}\int_{\mathbb{R}}e^{i\zeta\cdot x'+ikx_{n}} G(\zeta,k,t) dk\, d\zeta,
\end{equation}
where\begin{equation*}
	G(\zeta,k,t)=e^{-i(|\zeta|^{2}+k^{2})t} \left(ik \widetilde{g}_{0}(\zeta,i(|\zeta|^{2}+k^{2}),t)+\widetilde{g}_{1}(\zeta,i(|\zeta|^{2}+k^{2}),t)\right).
\end{equation*}
By applying integration by parts twice with respect to $\tau$ to \eqref{251}, we derive
\begin{equation*} e^{-i(|\zeta|^{2}+k^{2})t}\widetilde{g}_{0}\left(\zeta,i(|\zeta|^{2}+k^{2}),t\right)=\frac{\widehat{\phi}^{\,x'}(\zeta,0,t)}{i(|\zeta|^{2}+k^{2})}+\frac{\widehat{\phi}^{\,x'}_{t}(\zeta,0,t)}{(|\zeta|^{2}+k^{2})^{2}}-\int_{0}^{t} \widehat{\phi}^{\,x'}_{\tau\tau}(\zeta,0,\tau)\frac{e^{i(|\zeta|^{2}+k^{2})(\tau-t)}}{(|\zeta|^{2}+k^{2})^{2}} d\tau,
\end{equation*}
and similar equality holds for $\widetilde{g}_{1}$. Consequently,
\begin{equation}\label{smash}
	G(\zeta,k,t)=O\left(k^{-1}\right) \ \text{as}\ k\rightarrow\infty, \quad	G_{k}(\zeta,k,t)= O\left(k^{-2}\right) \ \text{as} \ k\rightarrow \infty.
\end{equation}
Considering the contour $\Gamma_{R}=\left\{Re^{i\theta},\theta\in\left[\frac{\pi}{2},\pi\right]\right\}$ and applying integration by parts, we obtain
\begin{equation*}
	\begin{aligned}
		\int_{\Gamma_{R}}e^{ikx_{n}} G(\zeta,k,t)dk&=
		\frac{G(\zeta,iR,t)e^{-Rx_{n}}-G(\zeta,-R,t)e^{-iR x_{n}}}{ix_{n}}-\int_{\Gamma_{R}} G_{k}(\zeta,k,t) \frac{e^{ikx_{n}}}{ix_{n}} dk\\
		&\rightarrow0,\quad R\rightarrow\infty,
	\end{aligned}
\end{equation*}
which is guaranteed by \eqref{smash}. By Cauchy's theorem, the real integral over $k$ in \eqref{253} deforms  into a complex integral, and then
{\small\begin{equation}\label{oral}
		\phi(x,t)=\frac{-i}{(2\pi)^{n}}\int_{\mathbb{R}^{n-1}}\int_{\partial D^{+}}e^{i\zeta\cdot x'+ik x_{n}-i(|\zeta|^{2}+k^{2})t} \left(ik \widetilde{g}_{0}(\zeta,i(|\zeta|^{2}+k^{2}),t)+\widetilde{g}_{1}(\zeta,i(|\zeta|^{2}+k^{2}),t)\right) dkd\zeta,
\end{equation}}
where $D^{+}$ and the direction of $\partial D^{+}$ are shown in Figure \ref{figure1}.
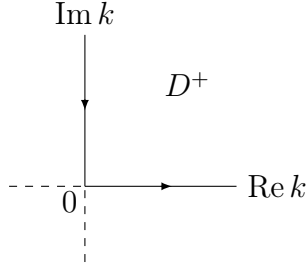
\begin{figure}
	\begin{center}
		\begin{tikzpicture}
			\node at (1.35, 1.35)  {$D^{+}$};
			\draw [ ](0,0)--(2,0)  node[right, scale=1] {\text{Re}\,$k$};
			\draw [ ](0,0)--(0,2)  node[above, scale=1] {\text{Im}\,$k$};
			\draw [dashed](0,-1)--(0,0)  node[above, scale=1]{};
			\draw [dashed](-1,0)--(0,0)  node[right, scale=1]{};
			\draw [-latex](1,0)--(1.15,0);
			\draw [-latex](0,1.15)--(0,1);
			\node at (-0.2, -0.2 )  {$0$};
		\end{tikzpicture}
	\end{center}
	\caption{The region $D^{+}$ and its positively oriented boundary $\partial D^{+}$.}
	\label{figure1}
\end{figure}

On the other hand, for any $k\in \overline{D^{+}}$, note that $-k\in \overline{\mathbb{C}^{-}}$ . Applying the global relation \eqref{hawker} to $\mathscr{F}[\phi]\,(\zeta, -k,t)$, we eliminate the $\widetilde{g}_{1}$ term in \eqref{oral}, yielding
{\small\begin{equation}\label{burglar}
		\phi(x,t)=\frac{-i}{(2\pi)^{n}}\int_{\mathbb{R}^{n-1}}\int_{\partial D^{+}} e^{i\zeta\cdot x'+ik x_{n}}\left(e^{-i(|\zeta|^{2}+k^{2})t} 2ik \widetilde{g}_{0}(\zeta,i(|\zeta|^{2}+k^{2}),t)+i\mathscr{F}[\phi](\zeta,-k,t)\right) dkd\zeta.
\end{equation}}
It is easy to see that $\lim\limits_{k\rightarrow\infty}\mathscr{F}[\phi]\,\left(\zeta,-k,t\right)=0$. Then, by Jordan's lemma and Cauchy's theorem, we have
\begin{equation}\label{aka}
	\int_{\partial D^{+}}e^{ik x_{n}}\mathscr{F}[\phi]\,\left(\zeta,-k,t\right) dk=0, \quad \text{for}\quad x_{n}>0.
\end{equation}
Using the boundary condition $\phi(x,0,t)=g_{0}(x,t)$ and definition \eqref{251}, we obtain
\begin{equation}\label{brittle} \mathscr{F}[g_{0}]\left(\zeta,-|\zeta|^{2}-k^{2}\right)=\widetilde{g}_{0}\left(\zeta,i(|\zeta|^{2}+k^{2}),t\right)+\int_{\mathbb{R}^{n-1}}\int_{t}^{\infty}    g_{0}(x',\tau) e^{i\zeta\cdot x'} e^{i(|\zeta|^{2}+k^{2})\tau} d\tau\,dx'.
\end{equation}
Substituting  \eqref{brittle} into \eqref{burglar} yields an exponential term $e^{ik^{2}(\tau-t)}$. Here, $\text{Re} \left(ik^{2}(\tau-t)\right)<0$ when $k\in D^{+}$. By combining \eqref{burglar}--\eqref{brittle}, we derive the final solution formula to \eqref{locust}
\begin{equation}\label{mutter}
	\phi(x,t)=\frac{1}{(2\pi)^{n}}\int_{\mathbb{R}^{n-1}}\int_{\partial D^{+}} e^{i\zeta\cdot x'+ik x_{n}-i(|\zeta|^{2}+k^{2})t} 2k \mathscr{F}[g_{0}]\left(\zeta,-|\zeta|^{2}-k^{2}\right) dk\,d\zeta.
\end{equation}

Next, we bound the solution to \eqref{locust} by the boundary data. The boundary data $g_{0}$ is chosen from the space $\mathcal{H}_{0}^{s}(\mathbb{R}^{n}_{+})$, a subspace of $\mathcal{H}^{s}(\mathbb{R}^{n}_{+})$, with the additional assumption of a vanishing trace on the boundary if $s\in(\frac{1}{2},\frac{5}{2})$, and this space can also be described as the closure of $C_{c}^{\infty}(\mathbb{R}^{n}_{+})$ in  $\mathcal{H}^{s}(\mathbb{R}^{n}_{+})$. When  $s=\frac{1}{2}$, the space for boundary data $g_{0}$ is $\mathcal{H}_{00}^{\frac{1}{2}}(\mathbb{R}^{n}_{+})$, which consists of functions whose zero extension belongs to $\mathcal{H}^{\frac{1}{2}}(\mathbb{R}^{n})$. The relevant properties, such as boundedness of the zero extension, density and interpolation, are shown in \cite{facilitate}.

The Strichartz estimate for the solution $\phi$, involving the endpoint $(q,r)=\left(2,\frac{2n}{n-2}\right)$ for $n>2$, is as follows.
\begin{proposition}\label{kiln}	
	Let $s\in\left[0, \frac{5}{2}\right)$, $j\in\mathbb{N}$, $0<l<1$, and $s-2l\geq0$. For $g_{0}\in\mathcal{H}_{0}^{s}(\mathbb{R}_{+}^{n})$ ( $\mathcal{H}_{00}^{\frac{1}{2}}(\mathbb{R}_{+}^{n})$ when $s=\frac{1}{2}$), the following estimate holds:
	\begin{equation}\label{detergent}	\|\phi\|_{C_{b}^{j}(\overline{\mathbb{R}_{+}\mkern-4mu}\,;H^{s-2j}(\mathbb{R}^{n}_{+}))}+\|\phi\|_{W^{j,q}(\mathbb{R}_{+};W ^{s-2j,r}(\mathbb{R}_{+}^{n}))}+\|\phi\|_{B^{l}_{q,2}(\mathbb{R}_{+};W^{s-2l,r}(\mathbb{R}^{n}_{+}))}\lesssim\|g_{0}\|_{\mathcal{H}_{0}^{s}{(\mathbb{R}_{+}^{n})}}.
	\end{equation}
\end{proposition}
\begin{proof}
	Let $g:\mathbb{R}^{n-1} \times\mathbb{R}\rightarrow \mathbb{C}$ be the zero extension of $g_0$, satisfying  $g(x',t)=g_{0}(x',t)$ for $t\in\mathbb{R}_{+}$. The representation \eqref{mutter} can be reformulated in Fourier transform terms as
	\begin{equation}\label{paucity}
		\phi(x,t)=\frac{1}{(2\pi)^{n}}\int_{\mathbb{R}^{n-1}}\int_{\partial D^{+}}e^{i\zeta\cdot x'+ikx_{n}-i(|\zeta|^{2}+k^{2})t}2k\, \widehat{g}\left(\zeta,-|\zeta|^{2}-k^{2}\right)dk\,d\zeta.
	\end{equation}
	By decomposing the contour $\partial D^{+}$ into $[0,+\infty)$ and $(i\infty,0]$, we can express \eqref{paucity} as $\phi=\phi_{1}+\phi_{2}$, where
	\begin{align}	&\phi_{1}(x,t)=\frac{1}{(2\pi)^{n}}\int_{\mathbb{R}^{n-1}}\int_{\mathbb{R}_{+}}e^{i\zeta\cdot x'+imx_{n}-i|\zeta|^{2}t-im^{2}t}\, 2m \, \widehat{g}(\zeta,-m^{2}-|\zeta|^{2})dm\, d\zeta,\label{q1}\\ &\phi_{2}(x,t)=\frac{1}{(2\pi)^{n}}\int_{\mathbb{R}^{n-1}}\int_{\mathbb{R}_{+}}e^{i\zeta\cdot x'-mx_{n}-i|\zeta|^{2}t+im^{2}t} \,2m\,\widehat{g}(\zeta,m^{2}-|\zeta|^{2}) dm\,d\zeta.\label{overwhelm}
	\end{align}
	We observe that the solution component $\phi_{1}$ admits the representation
	\begin{equation}\label{trestle}
		\phi_{1}(x,t)= \left[e^{it\Delta}\left(2m \, \raisebox{0.5ex}{$\chi$}_{\mathbb{R}_{+}}(m) \,\widehat{g}(\zeta ,-m^{2}-|\zeta|^{2})\right)^{\vee\,(\zeta,m)}\right](x).
	\end{equation}
	Then, by \eqref{autocrat} we have
	\begin{equation} \label{ahdi} \|\phi_{1}\|_{L^{q}(\mathbb{R}_{+};L^{r}(\mathbb{R}^{n}_{+}))}\lesssim\left\|2m\, \raisebox{0.5ex}{$\chi$}_{\mathbb{R}_{+}}(m)\, \widehat{g}(\zeta ,-m^{2}-|\zeta|^{2})\right\|_{L^{2}_{\zeta,m}(\mathbb{R}^{n})}\lesssim \|g\|_{\mathcal{H}^{0}(\mathbb{R}^{n})}.
	\end{equation}
	
	To analyze $\phi_{2}$, we define the following operator for any $t\in\mathbb{R}_{+}$:
	\begin{equation}\label{chagrin}
		[U(t)g](x):=\int_{\mathbb{R}_{+}}2m e^{-mx_{n}} e^{im^{2}t} \left(e^{-i|\zeta|^{2}t}\,\widehat{g}(\zeta,m^{2}-|\zeta|^{2})\right)^{\vee_{\zeta}}(x') \,dm.
	\end{equation}
	By using the $L^{2}$-boundedness of the Laplace transform (see \cite{kdv}) and Plancherel's identity, we have
	\begin{equation}\label{ajfd}
		\|U(t)g \|_{L^{2}(\mathbb{R}^{n}_{+})}\lesssim \left\|2m \,\widehat {g}(\zeta,m^{2}-|\zeta|^{2})\right\|_{L^{2}_{\zeta,m}(\mathbb{R}^{n}_{+})}\lesssim \|g\|_{\mathcal{H}^{0}(\mathbb{R}^{n})}.
	\end{equation}
	Then the operator $U(t)$ is bounded from $\mathcal{H}^{0}(\mathbb{R}^{n})$ to $L^{2}(\mathbb{R}^{n}_{+})$. Its adjoint operator has the explicit form:
	\begin{equation*}
		U(t)^{*}h=  \left(\int_{\mathbb{R}_{+}} \frac{\widehat{h}^{\,y'}(\zeta,y_{n})e^{-\sqrt{|\eta+|\zeta|^{2}|}\,y_{n}}e^{-i\eta t}}{\sqrt{|\eta+|\zeta|^{2}|}} dy_{n}\right)^{\vee\,(\zeta,\eta)}.
	\end{equation*}
	Moreover, the composition $U(t)U(\tau)^{*}$ satisfies
	{\small	\begin{equation*}
			[U(t)U(\tau)^{*}h](x)= \int_{\mathbb{R}^{n-1}}\int_{\mathbb{R}_{+}}\int_{\mathbb{R}_{+}}\int_{\mathbb{R}^{n-1}}h(y',y_{n}) e^{-m(x_{n}+y_{n})} e^{i(m^{2}-|\zeta|^{2})(t-\tau)} e^{i\zeta\cdot x'-i\zeta\cdot y'}dy'\,dy_{n}\,dm\,d\zeta.
	\end{equation*}}
	By applying Fubini's theorem and van der Corput's lemma \cite{mat}, we derive
	\begin{equation}\label{sagf}
		\begin{aligned}
			&\left\|U(t)U(\tau)^{*}h\right\|_{L^{\infty}(\mathbb{R}^{n}_{+})}\\		\leq&\varlimsup_{\kappa\rightarrow\infty}\int_{\mathbb{R}^{n}_{+}}|h(y)|\left|\int_{[-\kappa,\kappa]^{n-1}}e^{-i|\zeta|^{2}(t-\tau)} e^{i\zeta\cdot(x'-y')}d\zeta\right|\,\left|\int_{\mathbb{R}_{+}}e^{im^{2}(t-\tau)} e^{-m(x_{n}+y_{n})} dm\right| dy\\
			\lesssim&\frac{1}{|t-\tau|^{\frac{n}{2}}}\|h\|_{L^{1}(\mathbb{R}^{n}_{+})}.
		\end{aligned}
	\end{equation}
	From \eqref{ajfd} and \eqref{sagf}, the assumptions of Lemma \ref{tas} are satisfied, and thus
	\begin{equation}\label{sdf} \|q_{2}\|_{L^{q}(\mathbb{R}_{+};L^{r}(\mathbb{R}^{n}_{+}))}\lesssim \|g\|_{\mathcal{H}^{0}(\mathbb{R}^{n})}.
	\end{equation}
	
	Differentiating \eqref{paucity} yields
	\begin{equation}\label{meditation}
		\begin{aligned}		
			&\partial^{\alpha}_{x}\phi(x,t)\\
			=\,&\frac{1}{(2\pi)^{n}}\int_{\mathbb{R}^{n-1}}\int_{\partial D^{+}} (i\zeta)^{\beta}(ik)^{\nu}e^{i\zeta\cdot x'+ikx_{n}-i(|\zeta|^{2}+k^{2})t}2k\,\widehat{g}\left(\zeta,-(k^{2}+|\zeta|^{2})\right) \,dk\,d\zeta\\
			=\,&\frac{1}{(2\pi)^{n}}\int_{\mathbb{R}^{n-1}}\int_{\partial D^{+}}e^{i\zeta\cdot x'+ikx_{n}-i(|\zeta|^{2}+k^{2})t}2k\, \widehat{\psi}(\zeta,-(k^{2}+|\zeta|^{2}))  \,dk\,d\zeta,\\
		\end{aligned}
	\end{equation}
	where		 $\alpha=(\beta,\nu)\in \mathbb{N}^{n-1}\times\mathbb{N}$ and $\psi=\left[(i\zeta)^{\beta} (\eta+|\zeta|^{2})^{\frac{\nu}{2}} \widehat{g}(\zeta,\eta)\right]^{\vee_{(\zeta,\eta)}}$. Here, $z^{\frac{\nu}{2}}=|z|^{\frac{\nu}{2}}e^{\frac{i\nu}{2}\arg(z)}$, $\arg(z)\in (0,2\pi]$. From the estimates \eqref{ahdi} and \eqref{sdf}, we have
	\begin{equation*}
		\begin{aligned} \|\partial^{\alpha}_{x}\phi\|_{L^{q}(\mathbb{R}_{+};L^{r}(\mathbb{R}^{n}_{+}))}^{2}
			&\lesssim \|\psi\|_{\mathcal{H}^{0}(\mathbb{R}^{n})}^{2}
			=\int_{\mathbb{R}^{n}}|\zeta|^{2|\beta|} \, \left|\eta+|\zeta|^{2}\right|^{\nu+\frac{1}{2}}  \,   |\widehat{g}(\zeta,\eta)|^{2}d\eta\, d\zeta\\
			&\lesssim \int_{\mathbb{R}^{n}}\left|1+|\zeta|^{2}+|\eta|\right|^{|\beta|+\nu} \left|\eta+|\zeta|^{2}\right|^{\frac{1}{2}} \left|\widehat{g}(\zeta,\eta)\right|^{2}d\eta \, d\zeta=\|g\|_{\mathcal{H}^{|\alpha|}(\mathbb{R}^{n})}^{2}.
		\end{aligned}
	\end{equation*}
	Hence, estimates for the positive integer order derivatives are obtained. Using the evolution equation $\phi_{t}=i\Delta \phi$ and interpolation, the estimate \eqref{detergent} holds.
\end{proof}

\begin{proposition}\label{vilify}
	For  $g_{0}\in\mathcal{H}_{0}^{s}(\mathbb{R}^{n}_{+})$ with $s\in\left[0,\frac{5}{2}\right)$ (  $\mathcal{H}^{\frac{1}{2}}_{00}(\mathbb{R}^{n}_{+})$ when $s=\frac{1}{2}$ ), the solution $q$ satisfies
	\begin{equation}\label{confront}
		\|\phi\|_{C_{b}(\overline{\mathbb{R}_{+}\mkern-4mu}\,;\mathcal{H}^{s}(\mathbb{R}_{+}^{n}) )}\lesssim \|g_{0}\|_{\mathcal{H}_{0}^{s}(\mathbb{R}_{+}^{n})},
	\end{equation}
	and
	\begin{equation}\label{scrap} \|\phi\|_{C_{b}^{1}(\overline{\mathbb{R}_{+}\mkern-4mu}\,;\mathcal{H}^{s-1}(\mathbb{R}_{+}^{n}) )}\lesssim \|g_{0}\|_{\mathcal{H}_{0}^{s}(\mathbb{R}_{+}^{n})}.
	\end{equation}
\end{proposition}
\begin{proof}
	Let us recall the decomposition $\phi=\phi_{1}+\phi_{2}$ used in the proof of Proposition \ref{kiln}. For \eqref{confront}, it follows  from \eqref{trestle} and \eqref{circuitous} that
	\begin{equation*}
		\|\phi_{1}\|_{C_{b}(\overline{\mathbb{R}_{+}\mkern-4mu}\,;\mathcal{H}^{s}(\mathbb{R}_{+}^{n}) )}\lesssim \|g\|_{\mathcal{H}^{s}(\mathbb{R}_{+}^{n})}.
	\end{equation*}
	Applying the variable substitution $m=\sqrt{\eta+|\zeta|^{2}}$ with $\eta\in(-|\zeta|^{2},\infty)$, we reformulate \eqref{overwhelm} as
	\begin{equation}\label{illusory}
		\phi_{2}(x,t)=\frac{1}{(2\pi)^{n}}\int_{\mathbb{R}^{n-1}} \int_{\mathbb{R}}  \raisebox{0.5ex}{$\chi$}_{(-|\zeta|^{2},\infty)}(\eta)e^{i\zeta\cdot x'+i\eta t-\sqrt{\eta+|\zeta|^{2}}x_{n}} \widehat{g}(\zeta,\eta)d\eta\, d\zeta,
	\end{equation}
	thus establishing \eqref{confront} for $\phi_{2}$ through the definition of the $\mathcal{H}^{s}$ spaces.
	For \eqref{scrap}, the estimate for $\phi_{1}$ is obtained from \eqref{tentative}. For $\phi_{2}$, by \eqref{illusory}, we have
	\begin{equation*}
		\begin{aligned} \|\phi\|_{C_{b}^{1}(\overline{\mathbb{R}_{+}\mkern-4mu}\,;\mathcal{H}^{s-1}(\mathbb{R}_{+}^{n}) )}^{2}\lesssim \|\psi\|_{\mathcal{H}^{s-1}(\mathbb{R}^{n})}^{2}
			&=\int_{\mathbb{R}^{n-1}}\int_{\mathbb{R}}  \left(1+|\zeta|^{2}+|\eta|\right)^{s-1} \left|\eta+|\zeta|^{2}\right|^{\frac{3}{2}}|\widehat{g}(\zeta,\eta)|^{2}d\eta \, d\zeta\\
			& \lesssim \|g\|_{\mathcal{H}^{s}(\mathbb{R}^{n})}^{2}.
		\end{aligned}
	\end{equation*}
\end{proof}
We now extend our results from the temporal half-line to $(0,T)$. Define ${}_{0}\mathcal{H}^{s}(\mathbb{R}^{n-1}\times(0,T))$ as the restriction space
of $\mathcal{H}_{0}^{s}(\mathbb{R}^{n}_{+})$ to $\mathbb{R}^{n-1}\times(0,T)$. Recall that Proposition \ref{kiln} implies for any  $g_{0}\in\mathcal{H}_{0}^{s}(\mathbb{R}_{+}^{n})$, we can obtain the corresponding solution $q(\cdot,t)\in H^{s}(\mathbb{R}^{n}_{+})$ to \eqref{locust}. We therefore define the operator
\begin{equation}\label{st}
	S(t): {}_{0}\mathcal{H}^{s}(\mathbb{R}^{n-1}\times(0,T))\to H^{s}(\mathbb{R}^{n}_{+}),\quad  t\in\overline{\mathbb{R}_{+}},
\end{equation}
where $S(\cdot)g_{0}$  solves \eqref{locust} on $\mathbb{R}^{n}_{+}\times (0,T)$. Moreover, for $g_{0}\in{}_{0}\mathcal{H}^{s}(\mathbb{R}^{n-1}\times(0,T))$, the solution $S(\cdot)g_{0}$ satisfies
{\small
	\begin{equation}\label{grit}
		\|S(\cdot)g_{0}\|_{C_{b}^{j}([0,T]\,;H^{s-2j}(\mathbb{R}^{n}_{+}))\cap W^{j,q}(0,T;W ^{s-2j,r}(\mathbb{R}_{+}^{n}))}
		+\|S(\cdot)g_{0}\|_{ B^{l}_{q,2}(0,T;W^{s-2l,r}(\mathbb{R}^{n}_{+})) }
		\lesssim\|g_{0}\|_{_{0}\mathcal{H}^{s}({\mathbb{R}^{n-1}\times(0,T))}},
\end{equation}}
which follows from \eqref{detergent}. Similarly, $S(\cdot)g_{0}$ also satisfies the estimates in the form of \eqref{confront} and \eqref{scrap}, with the right-hand side being $\|g_{0}\|_{{}_{0}\mathcal{H}^{s}(\mathbb{R}^{n-1}\times(0,T))}$. Furthermore, for  $s\in\left[2,\frac{5}{2}\right)$, the solution $S(\cdot)g_{0}$ can be controlled by a norm weaker than $\mathcal{H}^{s}(\mathbb{R}^{n-1}\times(0,T))$.
\begin{proposition}
	For $s\in \left[2,\frac{5}{2}\right)$, $j\in\mathbb{N}$ and  $g_{0}\in  {}_{0}\mathcal{H}^{s}(\mathbb{R}^{n-1}\times(0,T))$, solution $S(\cdot)g_{0}$ satisfies
	\begin{equation}\label{squeak}
		\begin{aligned} &\|S(\cdot)g_{0}\|_{C^{j}([0,T];H^{s-2j}(\mathbb{R}^{n}_{+}))\cap W^{j,q}(0,T;W^{s-2j,r}(\mathbb{R}^{n}_{+}))}+\|S(\cdot)g_{0}\|_{C_{b}(\overline{\mathbb{R}_{+}\mkern-4mu}\,;\mathcal{H}^{s-2}(\mathbb{R}^{n-1}\times(0,T))}\\
			&\quad+\|\partial_{t}S(\cdot)g_{0}\|_{C_{b}(\overline{\mathbb{R}_{+}\mkern-4mu}\,;\mathcal{H}^{s-2}(\mathbb{R}^{n-1}\times(0,T))}
			\lesssim \|g_{0}\|_{\mathcal{H}^{s-2}(\mathbb{R}^{n-1}\times(0,T))} +\|\partial_{t}g_{0}\|_{\mathcal{H}^{s-2}(\mathbb{R}^{n-1}\times(0,T))}.
		\end{aligned}
	\end{equation}
\end{proposition}
\begin{proof}
	Since $g_{0}\in{}_{0}\mathcal{H}^{s}(\mathbb{R}^{n-1}\times(0,T))$, \eqref{shenjing} implies  $\partial_{t}g_{0}\in \mathcal{H}^{s-2}(\mathbb{R}^{n-1}\times(0,T))$ with $s-2\in\left[0,\frac{1}{2}\right)$. By the uniqueness of solutions, we have, $\partial_{t}S(t)g_{0}=S(t)\partial_{t}g_{0}$. Then use \eqref{grit}, we derive
	\begin{equation*}
		\|S(\cdot)g_{0}\|_{L^{q}(0,T;W^{s-2,r}(\mathbb{R}^{n}_{+}))}\lesssim  \|g_{0}\|_{\mathcal{H}^{s-2}(\mathbb{R}^{n-1}\times(0,T))},
	\end{equation*}
	and
	\begin{equation*}
		\begin{aligned}
		\|\Delta S(\cdot)g_{0}\|_{L^{q}(0,T;W^{s-2,r}(\mathbb{R}^{n}_{+}))}&\lesssim\|\partial_{t}S(\cdot)g_{0}\|_{L^{q}(0,T;W^{s-2,r}(\mathbb{R}^{n}_{+}))}\\
		&\lesssim  \|\partial_{t}g_{0}\|_{\mathcal{H}^{s-2}(\mathbb{R}^{n-1}\times(0,T))}.
		\end{aligned}
	\end{equation*}
	Combining these estimates with \eqref{confront} yields \eqref{squeak}.
\end{proof}

Based on the estimate for $s=2$ in \eqref{squeak}, we consider the new function space $\mathcal{W}^{2}(\mathbb{R}^{n-1}\times(0,T))$ defined by the norm
\begin{equation}\label{huaw}
	\|h\|_{\mathcal{W}^{2}(\mathbb{R}^{n-1}\times(0,T))}:=\|h\|_{\mathcal{H}^{0}(\mathbb{R}^{n-1}\times(0,T))}+\|h_{t}\|_{\mathcal{H}^{0}(\mathbb{R}^{n-1}\times(0,T))},
\end{equation}
which is weaker than $\mathcal{H}^{2}(\mathbb{R}^{n-1}\times(0,T))$ according to \eqref{shenjing}, and clearly ${}_{0}\mathcal{H}^{2}(\mathbb{R}^{n-1}\times(0,T))\subset \mathcal{W}^{2}(\mathbb{R}^{n-1}\times(0,T))$. We then define its closed subspace \[{}_{0}\mathcal{W}^{2}(\mathbb{R}^{n-1}\times(0,T)):=\overline{{}_{0}\mathcal{H}^{2}(\mathbb{R}^{n-1}\times(0,T))}.\] Through a density argument, the operator $S(t)$ can be extended to ${}_{0}\mathcal{W}^{2}(\mathbb{R}^{n-1}\times(0,T))$. For $g_{0}\in {}_{0}\mathcal{W}^{2}(\mathbb{R}^{n-1}\times(0,T))$, $S(\cdot)g_{0}$ satisfies \eqref{locust}  in a weaker way: \[\lim\limits_{a\rightarrow0}\big\|S(\cdot)g_{0}|_{x_{n}=a}-g_{0}\big\|_{\mathcal{W}^{2}(\mathbb{R}^{n-1}\times(0,T))}=0. \]
Moreover, the following estimates hold.
\begin{corollary}
	For $g_{0}\in{}_{0}\mathcal{W}^{2}(\mathbb{R}^{n-1}\times(0,T))$, we have
	{\small	\begin{equation}\label{invoice}
			\begin{aligned} \|&S(\cdot)g_{0}\|_{C^{j}([0,T];H^{2-2j}(\mathbb{R}^{n}_{+}))\cap W^{j,q}(0,T;W^{2-2j,r}(\mathbb{R}^{n}_{+}))}\\&+\|S(\cdot)g_{0}\|_{C_{b}(\overline{\mathbb{R}_{+}\mkern-4mu}\,;\mathcal{W}^{2}(\mathbb{R}^{n-1}\times(0,T)))}\lesssim\|g_{0}\|_{\mathcal{W}^{2}(\mathbb{R}^{n-1}\times(0,T))}.
			\end{aligned}
	\end{equation}}
\end{corollary}

\subsection{End of the linear estimates}\label{sec24}

In this section, we combine the results in Section \ref{sec21}--\ref{sec23} and obtain the final estimates for the linear IBVP \eqref{temporal}, which prepares the establishment of the nonlinear problem theory. By the superposition principle, the solution to \eqref{temporal} admits the representation
{\small \begin{equation}\label{turbine}	
		\begin{aligned}		u(t)&=e^{it\Delta}Eu_{0}-i\int_{0}^{t}e^{i(t-\tau)\Delta}Ef(\tau)d\tau+S(t)\left(h_{0}-e^{i(\cdot)\Delta}Eu_{0}|_{x_{n}=0}+i\int_{0}^{\cdot}e^{i(\cdot\,-\tau)\Delta}Ef(\tau)d\tau|_{x_{n}=0}\right)\\
			&\equiv I+II+III,
		\end{aligned}
\end{equation}}
where $E$ denotes a linear extension operator satisfying the bound $\|Ef\|_{W^{s,r}(\mathbb{R}^{n})}\lesssim \|f\|_{W^{s,r}(\mathbb{R}^{n}_{+})}$, and $S(t)$ is given by \eqref{st}.

\emph{Hereafter, to simplify notation, we often drop $\mathbb{R}_{+}^{n}$ in function spaces and norms when no confusion arises.}

\begin{proposition}\label{abss}
	Let $s\in \left[0,\frac{5}{2}\right)$, $u_{0}\in H^{s}(\mathbb{R}_{+}^{n})$ and $h_{0}\in\mathcal{H}^{s}(\mathbb{R}_{+}^{n})$ satisfy the compatibility condition. For any $T>0$, if $f$ belongs to the function spaces determined by the right-hand side norms in the following inequalities, then the solution $u$ given by \eqref{turbine} satisfies
	\begin{equation}\label{lo}
		u\in C\left(\left[0,T\right];H^{s}(\mathbb{R}^{n}_{+})\right)\cap C_{b}\left(\overline{\mathbb{R}_{+}\mkern-4mu}\,;\mathcal{H}^{s}(\mathbb{R}^{n-1}\times(0,T))\right)
	\end{equation}
	and the following estimates:
	\begin{enumerate}[label=(\arabic*)]
		\item For $s=0$:
		\begin{equation}\label{intimate}
			\|u\|_{L^{q}(\mathbb{R}_{+};L^{r})}\lesssim \|u_{0}\|_{L^{2}}+\|h_{0}\|_{\mathcal{H}^{0}}+ \|f\|_{L^{\gamma'}(\mathbb{R}_{+};L^{\rho'})}.
		\end{equation}
		\item For $s\in (0,2)$:
		\begin{equation}\label{confine}
			\begin{aligned}
				\|u\|_{L^{q}(\mathbb{R}_{+};W^{s,r})} +\|u\|_{B^{\frac{s}{2}}_{q,2} (\mathbb{R}_{+}; L^{r} )}
				\lesssim &\, \|u_{0}\|_{H^{s}}+\|h_{0}\|_{\mathcal{H}^{s}}+ \|f\|_{L^{\gamma'} (\mathbb{R}_{+};  W^{s,\rho'})}\\
				&+\|f\|_{B^{\frac{s}{2}}_{\gamma',2} (\mathbb{R}_{+}; L^{\rho'})}.
			\end{aligned}
		\end{equation}
		\item For $s=2$:
		\begin{equation}\label{dilute}
			\begin{aligned}
				\|u\|_{L^{q}(0,T;W^{2,r})}+\|u\|_{W^{1,q}(0,T;L^{r})}
				\lesssim &\, \|u_{0}\|_{H^{2}}+\|h_{0}\|_{\mathcal{H}^{2}}+\|f(0)\|_{L^{2}}\\
				&+\|f\|_{L^{\gamma'}(0,T;W^{2,\rho'})}+\|f\|_{W^{1,\gamma'}(0,T;L^{\rho'})}.
			\end{aligned}
		\end{equation}
		\item For $s\in \left(2,\frac{5}{2}\right)$:
		\begin{equation}\label{ratify}
			\begin{aligned}
				\|u\|_{L^{q}(0,T;W^{s,r})}+\|u\|_{W^{1,q}(0,T;W^{s-2,r})}
				\lesssim &\, \|u_{0}\|_{H^{s}}+\|h_{0}\|_{\mathcal{H}^{s}}+\|f(0)\|_{H^{s-2}}\\
				&+\|f\|_{L^{1}(0,T;H^{s})}+\|f\|_{W^{1,1}(0,T;H^{s-2})}.
			\end{aligned}
		\end{equation}
	\end{enumerate}
\end{proposition}
\begin{proof}
	We present a detailed proof for (1), the arguments for cases (2)--(4) are analoguous and we omit the details.
	\begin{enumerate}[label=(\arabic*)]
		\item For the first two terms $I$ and $II$ in \eqref{turbine}, it follows from \eqref{autocrat} and \eqref{wrath} that
		\begin{equation*}
			\left\|e^{it\Delta}Eu_{0}\right\|_{L^{q}_{t}(\mathbb{R}_{+};L^{r})}+\left\|i\int_{0}^{t}e^{i(t-\tau)\Delta}Ef(\tau)d\tau\right\|_{L_{t}^{q}(\mathbb{R}_{+};L^{r})}\lesssim\|u_{0}\|_{L^{2}}+\|f\|_{L^{\gamma'}(\mathbb{R}_{+}; L^{\rho'})}.
		\end{equation*}
		By successively applying \eqref{detergent}, \eqref{circuitous}, and \eqref{cope}, we derive the following estimate for the third term $III$:
		\begin{equation*}
			\begin{aligned}
				&\left\| S(t)\left(h_{0}-e^{i(\cdot)\Delta}Eu_{0}|_{x_{n}=0}+i\int_{0}^{\cdot}e^{i(\cdot\,-\tau)\Delta}Ef(\tau)d\tau|_{x_{n}=0}\right)\right\|_{L_{t}^{q}(\mathbb{R}_{+};L^{r})}\\
				\lesssim &\, \|h_{0}\|_{\mathcal{H}^{0}}+\left\|e^{i(\cdot)\Delta}Eu_{0}|_{x_{n}=0}\right\|_{\mathcal{H}^{0}}+\left\|i\int_{0}^{\cdot}e^{i(\cdot\,-\tau)\Delta}Ef(\tau)d\tau|_{x_{n}=0}\right\|_{\mathcal{H}^{0}}\\
				\lesssim &\, \|h_{0}\|_{\mathcal{H}^{0}}+\|u_{0}\|_{L^{2}}+\|f\|_{L^{\gamma'}(\mathbb{R}_{+};L^{\rho'})}.
			\end{aligned}
		\end{equation*}
		\item Following analogous methodology, the first two terms $I$ and $II$ admit estimation through \eqref{autocrat} and \eqref{kidney}. By sequentially applying \eqref{detergent}, \eqref{circuitous} and \eqref{juvenile}, we can estimate the third term $III$. Moreover, for $s=1$, applying \eqref{tentative}, \eqref{crisp}, \eqref{scrap}, and \eqref{juvenile} yields
		\begin{equation}\label{contagious}
			\|u\|_{ C_{b}^{1}(\overline{\mathbb{R}_{+}\mkern-4mu}\,;\mathcal{H}^{s-1}(\mathbb{R}_{+}^{n}))}\lesssim\|u_{0}\|_{H^{s}}+\|h_{0}\|_{\mathcal{H}^{s}}+ \|f\|_{L^{\gamma'} (\mathbb{R}_{+};  W^{s,\rho'})}+\|f\|_{B^{\frac{s}{2}}_{\gamma',2} (\mathbb{R}_{+}; L^{\rho'})}.
		\end{equation}
		\item The estimates for the first two terms follow directly from inequalities \eqref{autocrat}, \eqref{wrath}, and \eqref{famine}. The third term is estimated by employing \eqref{squeak}, \eqref{circuitous}, and \eqref{cope} in combination with the derivative formula
		\begin{equation}\label{militant} \partial_{t}\int_{0}^{t}e^{i(t-\tau)\Delta}Ef(\tau)d\tau=e^{it\Delta}Ef(0)+\int_{0}^{t}e^{i(t-\tau)\Delta}\partial_{t}Ef(\tau)d\tau.
		\end{equation}
		\item The proof of \eqref{ratify} follows the same approach as \eqref{dilute}, with the only modification being the substitution of \eqref{covet} for \eqref{cope}.
	\end{enumerate}
	When proving cases (2)--(4), it is necessary to verify the applicability conditions of \eqref{detergent} and \eqref{squeak}.
	In fact, from the compatibility condition between $u_{0}$ and $h_{0}$, as well as the estimate \eqref{circuitous}, we know that \[h_{0}-e^{i(\cdot)\Delta}Eu_{0}|_{x_{n}=0}\in {}_{0}\mathcal{H}^{s}(\mathbb{R}^{n}_{+}).\]
	By extending $f$ to the real axis and using \eqref{juvenile}, \eqref{bribery}, and \eqref{perilous}, it follows that  \[i\int_{0}^{\cdot}e^{i(\cdot\,-\tau)\Delta}Ef(\tau)d\tau|_{x_{n}=0}\in{}_{0}\mathcal{H}^{s}(\mathbb{R}^{n-1}\times(0,T)).\] Finally, the validity of \eqref{lo} can be established through
	an analogous argument developed above.
\end{proof}

\begin{remark}
	In the proof of \eqref{dilute} and \eqref{ratify}, it is necessary to use the modified estimate \eqref{squeak} instead of the previous one \eqref{grit}. Consider \eqref{dilute} as an illustration: Applying \eqref{grit} directly would introduce the $\mathcal{H}^{2}$-norm, which compels us to adopt \eqref{bribery} and thus obtain
	\begin{equation*}
		\|u\|_{L^{q}(\mathbb{R}_{+};W^{2,r})}+\|u\|_{W^{1,q}(\mathbb{R}_{+};L^{r})}\lesssim\|u_{0}\|_{H^{2}}+\|h_{0}\|_{\mathcal{H}^{2}}+\|f\|_{L^{\gamma'}(\mathbb{R}_{+};W^{2,\rho'})}+\|f\|_{W^{1,\gamma'}(\mathbb{R}_{+};L^{\rho'})},
	\end{equation*}
	which was derived in \cite{facilitate}.
	Appling this to local solution analysis would lead to an inevitable $T^{-1}$ singularity, as it requires the global integrability of $\partial_{t}f$. Our estimate \eqref{dilute} circumvents this issue by restricting the integrability to  $(0,T)$, weakening the requirements on the forcing term, as will be demonstrated in \eqref{afh}--\eqref{afhh}.
\end{remark}

Parallel to Proposition \ref{abss}, we apply Propositions \ref{p33} and \ref{tart} given in Section \ref{sec22} to obtain the estimates below which mainly contribute to global well-posedness. Specifically, they are employed in the regularity argument (Sections \ref{sec42} and \ref{painstakingly}). The involvement of the time derivative here, with $s=\frac{3}{2}$ as a  dividing line,  parallels that in the one-dimensional result of \cite{figment}.
\begin{proposition}\label{mago}
	Let $s\in \left[0,\frac{5}{2}\right)$, $u_{0}\in H^{s}(\mathbb{R}_{+}^{n})$ and
	$h_{0}\in\mathcal{H}^{s}(\mathbb{R}_{+}^{n})$ satisfy the  compatibility condition. Then the solution $u$ given by \eqref{turbine} satisfies $u\in C\left(\left[0,T\right];H^{s}(\mathbb{R}^{n}_{+})\right)\cap C_{b}\left(\overline{\mathbb{R}_{+}\mkern-4mu}\,;\mathcal{H}^{s}(\mathbb{R}^{n-1}\times(0,T))\right)$ and the following estimates:
	\begin{enumerate}[label=(\arabic*)]
		\item For $s\in\left[0,\frac{3}{2}\right]$, $1\leq \alpha,\beta\leq2$ satisfing $\frac{1}{\alpha}+\frac{n}{2\beta}=1+\frac{3n-4s}{12}$:
		\begin{equation}\label{buffoon}
			\|u\|_{L^{q}(0,T;W^{s,r})}\lesssim \|u_{0}\|_{H^{s}}+\|h_{0}\|_{\mathcal{H}^{s}}+\langle T\rangle^{\frac{s}{3}}\|f\|_{L^{\alpha}(0,T;W^{s,\beta})}.
		\end{equation}
		In addition,  when  $n=2$ and $s\neq\frac{3}{2}$, we require $\left(\alpha,\beta\right)\neq \left(2,\frac{6n}{3n+6-4s}\right)$.
		\item For $s\in \left(\frac{3}{2},\frac{5}{2}\right)$:
		\begin{equation}\label{aggravate}
			\begin{aligned}
				&\|u\|_{L^{q}(0,T;W^{s,r})}+\|u\|_{W^{1,q}(0,T;W^{s-2,r})}
				\lesssim \|u_{0}\|_{H^{s}}+\|h_{0}\|_{\mathcal{H}^{s}}+\|f(0)\|_{H^{s-2}}\\
				&\quad
				+\left(\langle T\rangle^{\frac{1}{2}}+T^{-\frac{2s-3}{4}}\right)\left(\|f\|_{L^{2}(0,T;H^{s}) }+\|f\|_{H^{1}(0,T;H^{s-2})}\right).
			\end{aligned}
		\end{equation}
	\end{enumerate}
\end{proposition}
By using the $\mathcal{W}^{2}$ theoretical framework presented in \eqref{huaw}--\eqref{invoice}, we derive the following additional $H^{2}$-estimate.
\begin{proposition}\label{weak}
	Let $u_{0}\in H^{2}(\mathbb{R}_{+}^{n})$ and $h_{0}\in\mathcal{H}^{2}(\mathbb{R}^{n}_{+})$ satisfy the  compatibility condition. Then the $u$ given by \eqref{turbine} satisfies \(u\in C\left(\left[0,T\right];H^{2}(\mathbb{R}^{n}_{+})\right)\cap C_{b}\left(\overline{\mathbb{R}_{+}\mkern-4mu}\,;\mathcal{W}^{2}(\mathbb{R}^{n-1}\times(0,T))\right)\) and the following estimate
	\begin{equation}\label{begrudge}
		\|u\|_{C([0,T];H^{2})\cap W^{1,q}(0,T;L^{r})}\lesssim\|u_{0}\|_{H^{2}}+\|h_{0}\|_{\mathcal{H}^{2}}+\|f\|_{C([0,T];L^{2})\cap W^{1,\gamma'}(0,T;L^{\rho'})}.
	\end{equation}

\end{proposition}	
\begin{proof}
	The estimate for the first term $I$ follow from \eqref{autocrat}. For the second term $II$, from \eqref{famine} we have
	\begin{equation}\label{trifle}		
		\|II\|_{W^{1,q}(0,T;L^{r})}\lesssim \|f(0)\|_{L^{2}}+\|f\|_{W^{1,\gamma'}(0,T;L^{\rho'})}.
	\end{equation}
	Since $II$ satisfies the equation $i\partial_{t}II+\Delta II=f$, the estimate \eqref{trifle} yields
	\begin{equation*}
		\|\Delta II\|_{C([0,T];L^{2})}\lesssim  \|f\|_{C([0,T];L^{2})}+\|f\|_{W^{1,\gamma'}(0,T;L^{\rho'})},
	\end{equation*}
	hence we obtain the estimate for $II$. Finally, we consider the third term $III$. For $f\in C([0,T];L^{2})\cap W^{1,\gamma'}(0,T;L^{\rho'})$, we denote
	\begin{equation*}
		\Phi_{f}(t):=\int_{0}^{t}e^{i(t-\tau)\Delta}Ef(\tau)d\tau.
	\end{equation*}
	Using the definition  \eqref{huaw}, identity \eqref{militant} and inequalities \eqref{circuitous}, \eqref{cope}, we obtain
	$\Phi_{f}\in C_{b}\left(\overline{\mathbb{R}_{+}\mkern-4mu}\,;\mathcal{W}^{2}(\mathbb{R}^{n-1}\times(0,T))\right)$ with the estimate
	\begin{equation}\label{rampant}
		\begin{aligned}
			\left\|\Phi_{f}|_{x_{n}=0}\right\|_{\mathcal{W}^{2}(\mathbb{R}^{n-1}\times(0,T))}\lesssim\|f(0)\|_{L^{2}}+\|f\|_{W^{1,\gamma'}(0,T;L^{\rho'})}.
		\end{aligned}
	\end{equation} By a density argument, we can choose $f_{n}\in C([0,T];L^{2})\cap W^{1,\gamma'}(0,T;L^{\rho'})\cap L^{\gamma'}(0,T;W^{2,\rho'})$ to approximate $f$. By
	\eqref{bribery}, $\Phi_{f_{n}}\in C_{b}\left(\overline{\mathbb{R}_{+}\mkern-4mu}\,;{}_{0}\mathcal{H}^{2}(\mathbb{R}^{n-1}\times(0,T))\right)$. Then, from \eqref{militant}, \eqref{circuitous} and \eqref{cope}, we have
	\[
	\lim\limits_{n\rightarrow\infty}\Phi_{f_{n}}= \Phi_{f}\quad \text{in} \quad C_{b}\left(\overline{\mathbb{R}_{+}\mkern-4mu}\,;\mathcal{W}^{2}(\mathbb{R}^{n-1}\times(0,T))\right).\]
	Recall the definition of ${}_{0}\mathcal{W}^{2}(\mathbb{R}^{n-1}\times(0,T))$, we also have
	\begin{equation*}
		\Phi_{f}\in C_{b}\left(\overline{\mathbb{R}_{+}\mkern-4mu}\,;{}_{0}\mathcal{W}^{2}(\mathbb{R}^{n-1}\times(0,T))\right).
	\end{equation*}
	The compatibility condition ensures that $h_{0}-e^{i(\cdot)\Delta}Eu_{0}|_{x_{n}=0}\in{}_{0}\mathcal{W}^{2}(\mathbb{R}^{n-1}\times(0,T))$. Applying \eqref{invoice}, \eqref{circuitous} and \eqref{rampant} to term $III$, we derive
	\begin{equation*}
		\begin{aligned}
			\|III\|_{C([0,T];H^{2})\cap W^{1,q}(0,T;L^{r}) }\lesssim & \left\|h_{0}-e^{i(\cdot)\Delta}Eu_{0}|_{x_{n}=0}+i\Phi_{f}|_{x_{n}=0}\right\|_{\mathcal{W}^{2}(\mathbb{R}^{n-1}\times(0,T))}\\
			\lesssim & \, \|h_{0}\|_{\mathcal{W}^{2}}+\|u_{0}\|_{H^{2}}+\|f\|_{C([0,T];L^{2})\cap W^{1,\gamma'}(0,T;L^{\rho'})}.
		\end{aligned}
	\end{equation*}
\end{proof}
\begin{remark}
	In contrast to \eqref{dilute} and \eqref{aggravate}, the estimate \eqref{begrudge} does not involve the spatial derivative of $f$. This allows us, even if the nonlinear term lacks smoothness, to obtain $H^{2}$ solutions of \eqref{ibvp} (satisfing boundary data in a weaker sense, see Section~\ref{stipple}), thereby deriving global $H^{1}$ solutions through approximation in this setting.
\end{remark}
The linear problem \eqref{temporal} exhibits uniqueness: Assume that
\begin{equation*}
	u_{1}, u_{2}\in C([0,T);L^{2}(\mathbb{R}^{n}_{+}))\cap C_{b}(\overline{\mathbb{R}_{+}\mkern-4mu}\,;\mathcal{H}^{0}(\mathbb{R}^{n-1}\times(0,T)))
\end{equation*}
solve \eqref{temporal} with identical initial and boundary data, then $u_{1}=u_{2}$. The one-dimensional case is proven in Section 8 of \cite{chatter}, employing mollification techniques under the framework where $u(\cdot,t)\in L^{2}(\mathbb{R}_{+})$ and $u(x,\cdot)\in L^{2}(0,T)$. The high-dimensional case follows analogously, since our boundary traces belong to $\mathcal{H}^{0}(\mathbb{R}^{n}_{+})$, which is embedded into $L^{2}(\mathbb{R}^{n-1}\times(0,T))$ by \eqref{facile}.

\section{Local well-posedness for the nonlinear Schr\"{o}dinger equation}\label{sec3}
In this section, we develop the local theory for the nonlinear IBVP \eqref{ibvp} via the estimates established in Section \ref{sec2}. For initial data $u_{0}\in H^{s}(\mathbb{R}^{n}_{+})$ and boundary data  $h_{0} \in \mathcal{H}^{s}(\mathbb{R}_{+}^{n})$, by replacing $f$ with $-\lambda\psi_{T}F(u)$ in \eqref{turbine}, we formally define the operator
\begin{equation}\label{ahha}
	\begin{aligned}
		\Gamma u(t):=&\,e^{it\Delta}Eu_{0}+i\lambda\int_{0}^{t}e^{i(t-\tau)\Delta}E\psi_{T}(\tau)F(u(\tau))d\tau\\	&+S(t)\left(h_{0}-e^{i(\cdot)\Delta}Eu_{0}|_{x_{n}=0}-i\lambda\int_{0}^{\cdot}e^{i(\cdot-\tau)\Delta}E\psi_{T}(\tau)F(u(\tau))d\tau|_{x_{n}=0}\right),
	\end{aligned}
\end{equation}
where $E$ is the spatial extension operator in \eqref{turbine}, $\psi_{T}(t):=\psi\left(\frac{t}{T}\right)$ with $\psi\in C^{\infty}(\mathbb{R})$ satisfying $\psi(t)=1$ for $t\in[0,1]$, and $S(t)$ is given by \eqref{st}.   \emph{Throughout this work, we consistently denote $F(u)=|u|^{p}u$, where $p>0$.} The objective is to find a $u$ such that $u=\Gamma u$ by using the Banach fixed-point theorem, thereby constructing the local solution to \eqref{ibvp}. Here we list the lemmas and remarks that will be used subsequently.
\begin{lemma}[See \cite{gul}]\label{elicit}
	Let $1<\kappa<\infty$, $s\geq0$, and  $\frac{1}{\kappa}=\frac{1}{p_{i}}+\frac{1}{q_{i}}$ with $i=1,2,$
	$1<p_{1},q_{2}<\infty$, $1<p_{2},q_{1}\leq\infty$. Then for all $f\in L^{p_{2}}(\mathbb{R}^{n}_{+})\cap W^{s,p_{1}}(\mathbb{R}^{n}_{+})$, $g\in L^{q_{1}}(\mathbb{R}^{n}_{+})\cap W^{s,q_{2}}(\mathbb{R}^{n}_{+})$, one has
	\begin{equation}\label{trot}
		\|fg\|_{W^{s,\kappa}}\lesssim \|f\|_{W^{s,p_{1}}} \|g\|_{L^{q_{1}}} +\|f\|_{L^{p_{2}}}\|g\|_{W^{s,q_{2}}}.
	\end{equation}
\end{lemma}
\begin{lemma}[See \cite{fanxian}]\label{accessory}
	Let $0\leq s\leq k\leq \sigma$, $k\in\mathbb{Z}_{+}$. Suppose $H\in C^{k}(\mathbb{C};\mathbb{C})$ and there exists a constant $C>0$ such that
	\begin{equation*}
		\left|H^{(j)}(z)\right|\leq C |z|^{\sigma-j}, \quad j=0,1, \cdots,k.
	\end{equation*}
	Then for all $1<q_{0},q_{1}<\infty, 1<q_{2}\leq\infty$, $\frac{1}{q_{0}}=\frac{1}{q_{1}}+\frac{\sigma-1}{q_{2}}$ and $f\in L^{q_{2}}(\mathbb{R}^{n}_{+})\cap W^{s,q_{1}}(\mathbb{R}^{n}_{+})$, one has
	\begin{equation}\label{riesz}
		\|H(f)\|_{W^{s,q_{0} }}\lesssim \|f\|_{L^{q_{2}}}^{\sigma-1}\|f\|_{W^{s,q_{1}}}.
	\end{equation}
\end{lemma}

\begin{lemma}[See \cite{outage}]\label{relay}
	Let $I$ be an open interval in $\mathbb{R}$ and $X$ be a Banach space.
	\begin{enumerate}[label=(\roman*)]
		\item  Suppose $1\leq q\leq\infty$ and  $f\in W^{1,q}(I;X)$.  Then for all $K\subset\subset I$ and all $h\in\mathbb{R}$ with $|h|< d \,(K,\partial I)$, one has
		\begin{equation}
			\|f(\cdot+h)-f(\cdot)\|_{L^{q}(K;X)}\leq  \|f'\|_{L^{q}(I;X)}|h|.
		\end{equation}
		\item Suppose $1< q\leq\infty$, $f\in L^{q}(I;X)$, $X$ is reflexive and there exists a constant $C>0$ such that
		\begin{equation}
			\|f(\cdot+h)-f(\cdot)\|_{L^{q}(K;X)}\leq C |h|,
		\end{equation}
		for all $K\subset\subset I$ and all $h\in\mathbb{R}$ with $|h|<d(K,\partial I)$. Then
		\begin{equation}
			f\in  W^{1,q}(I;X),\ \ \text{with} \ \ \|f'\|_{L^{q}(I;X)}\leq C.
		\end{equation}
	\end{enumerate}	
\end{lemma}	

\begin{lemma}[See \cite{chatter}]\label{holmer}
	If	$1\leq q_{1}<q_{2}\leq\infty$ and for all $t\geq0$,
	\begin{equation}
		\|f\|_{L^{q_{2}}(0,t)}\leq c \delta + c \|f\|_{L^{q_{1}}(0,t)},
	\end{equation}
	then defining $\gamma$ via $2c\gamma^{\frac{1}{q_{1}}-\frac{1}{q_{2}}}=1$, we have for all $t\geq0$:
	\begin{equation}
		\|f\|_{L^{q_{2}}(0,t)}\leq(\gamma t)^{\gamma t}\delta.
	\end{equation}
\end{lemma}
\begin{lemma} The space $\mathcal{H}^{s}$  admits the following properties:
	\begin{enumerate}[label=(\arabic*)]
		\item If $f\in \mathcal{H}^{0}(\mathbb{R}^{n-1}\times(a,b))$, then $f\in L^{4}(a,b;L^{2}(\mathbb{R}^{n-1}))$ with
		\begin{equation}\label{facile}		\|f\|_{L^{4}(a,b;L^{2}(\mathbb{R}^{n-1}))}\lesssim\|f\|_{\mathcal{H}^{0}(\mathbb{R}^{n-1}\times(a,b))}.
		\end{equation}
		\item For any $T>0$, if $f\in \mathcal{H}^{2}(\mathbb{R}^{n-1}\times(0,T))$, then $f\in H^{1}(\mathbb{R}^{n-1}\times(0,T))$ and satisfies
		\begin{equation} \label{pleat}
			\|f\|_{H^{1}(\mathbb{R}^{n-1}\times(0,T))}\lesssim T^{\frac{1}{4}}\|f\|_{\mathcal{H}^{2}(\mathbb{R}^{n-1}\times(0,T))}.
		\end{equation}
	\end{enumerate}
\end{lemma}
\begin{proof}
	 By the extension principle, it suffices to consider  $(a,b)=\mathbb{R}$. From the isometry property of the group $\{e^{it\Delta'}\}_{t\in\mathbb{R}}$ and Corollary 2.4 in \cite{facilitate}, we derive
		\begin{equation*} \|f\|_{L^{4}(\mathbb{R};L^{2}(\mathbb{R}^{n-1}))}=\|e^{-it\Delta'}f\|_{L^{4}(\mathbb{R};L^{2}(\mathbb{R}^{n-1}))}\lesssim\|e^{-it\Delta'}f\|_{\dot{H}^{\frac{1}{4}}(\mathbb{R};L^{2}(\mathbb{R}^{n-1}))}\approx\|f\|_{\mathcal{H}^{0}(\mathbb{R}^{n})}.
		\end{equation*}
		Therefore, \eqref{facile} holds. From \eqref{facile}, we obtain
		\begin{equation*}
			\|f\|_{L^{2}(0,T;L^{2}(\mathbb{R}^{n-1}))}\lesssim T^{\frac{1}{4}}\|f\|_{L^{4}(0,T;L^{2}(\mathbb{R}^{n-1}))} \lesssim  T^{\frac{1}{4}}\|f\|_{\mathcal{H}^{2}(\mathbb{R}^{n-1}\times(0,T))}.
		\end{equation*}
		Furthermore, \eqref{shenjing}  implies $\partial_{i}f\in\mathcal{H}^{0}(\mathbb{R}^{n-1}\times(0,T))$. Applying \eqref{facile} again yields
		\begin{equation*}
			\begin{aligned}
				\|\partial_{i}f\|_{L^{2}(0,T;L^{2}(\mathbb{R}^{n-1}))}\leq T^{\frac{1}{4}}	\|\partial_{i}f\|_{L^{4}(0,T;L^{2}(\mathbb{R}^{n-1}))}&\lesssim T^{\frac{1}{4}}\|\partial_{i}f\|_{\mathcal{H}^{0}(\mathbb{R}^{n-1}\times(0,T))}\\&\lesssim T^{\frac{1}{4}}\|f\|_{\mathcal{H}^{2}(\mathbb{R}^{n-1}\times(0,T))}.
			\end{aligned}
		\end{equation*}
		This completes the proof of \eqref{pleat}.
\end{proof}
\begin{remark}
	By using parametrization, we can calculate that
	\begin{equation}\label{slate}
		\begin{aligned} F(u_{1})-F(u_{2})=&\,\frac{p+2}{2}(u_{1}-u_{2})\int_{0}^{1}\left|(1-\theta)u_{1}+\theta u_{2}\right|^{p}d\theta \\
			&+\frac{p}{2}(\overline{u_{1}-u_{2}})\int_{0}^{1} |(1-\theta)u_{1}+\theta u_{2}|^{p-2}\left((1-\theta)u_{1}+\theta u_{2}\right)^{2}d\theta,
		\end{aligned}
	\end{equation}
	and the well-known inequality
	\begin{equation}\label{uproar}
		F(u_{1})-F(u_{2})\lesssim \left(|u_{1}|^{p}+|u_{2}|^{p}\right)|u_{1}-u_{2}|.
	\end{equation}
\end{remark}

In the following, we prove Theorem \ref{chaste}  by considering two cases: the low regularity case $s\leq\frac{n}{2}$ and the high regularity case $s>\frac{n}{2}$.

\subsection{Low regularity local solutions}
In this section, we establish the local existence of IBVP \eqref{ibvp} in $H^{s}(\mathbb{R}^{n}_{+})$ with   $s\in[0,\frac{5}{2})$ (Theorem \ref{chaste})  for the low regularity case  $s\leq \frac{n}{2}$.

\subsubsection{Local $L^{2}$ solutions}\label{65}
\begin{proof}[Proof of Theorem \ref{chaste} for $s=0$]
	We define the space  $$X:=L^{q}(\mathbb{R}_{+};L^{r}(\mathbb{R}_{+}^n)),$$ where $(q,r)=\left(\frac{4(p+2)}{np},p+2\right)$. The closed ball $B[0,M]\subset X$ is complete under the metric $d(u,\tilde{u})=\|u-\tilde{u}\|_{L^{q}(\mathbb{R}_{+};L^{r})}$. Let $\psi\in C_{c}^{\infty}(\mathbb{R})$, the cut-off function $\psi_{T}$ in \eqref{ahha} satisfies  $\|\psi_{T}\|_{L^{q}(\mathbb{R})}\approx T^{\frac{1}{q}}$. We will show that for any $R>0$, any initial data $u_{0}$ and boundary data $h_{0}$ satisfying $\|u_{0}\|_{L^{2}}+\|h_{0}\|_{\mathcal{H}^{0}}\leq R$, there exist constants $M>0$ and $T_{*}>0$ such that the operator $\Gamma$ is a contraction mapping on the closed ball $B[0, M]\subset X $. It follows from \eqref{intimate} that
	\begin{equation}\label{nocturnal}
		\|\Gamma u\|_{X}\lesssim  \|u_{0}\|_{L^{2}}+\|h_{0}\|_{\mathcal{H}^{0}}+\|\psi_{T}F(u)\|_{L^{q'}(\mathbb{R}_{+};L^{r'})}.
	\end{equation}
	Through computation $\frac{1}{r'}=\frac{p}{r}+\frac{1}{r}$ and
	$\frac{1}{q'}=\frac{p+1}{q}+\frac{q-(p+2)}{q}$, the application of  H\"{o}lder's inequality in both spatial and temporal variables yields
	\begin{equation}\label{lad}
		\|\psi_{T}F(u)\|_{L^{q'}(\mathbb{R}_{+};L^{r'})}\lesssim \|\psi_{T}\|_{L^{\frac{q}{q-(p+2)}}(\mathbb{R}_{+})}\|u\|_{L^{q}(\mathbb{R}_{+};L^{r})}^{p+1}\lesssim T^{\frac{q-(p+2)}{q}} \|u\|_{L^{q}(\mathbb{R}_{+};L^{r})}^{p+1}.
	\end{equation}
	Combining \eqref{nocturnal} and \eqref{lad}, we have
	\begin{equation}\label{laptop}
		\|\Gamma u\|_{X}\leq C \left( \|u_{0}\|_{L^{2}}+\|h_{0}\|_{\mathcal{H}^{0}}+ T^{\frac{q-(p+2)}{q}} \|u\|_{X}^{p+1}\right),
	\end{equation}
	where $C$ is a constant independent of $T$. For any $u, v\in B[0,M]$,  successive applications of \eqref{intimate} and \eqref{uproar} yield
	\begin{equation*}\label{lame}
		\begin{aligned} d\,(\Gamma u,\Gamma \tilde{u})&\lesssim \left\|\psi_{T}(F(u)-F(\tilde{u}))\right\|_{L^{q'}(\mathbb{R}_{+};L^{r'})}\\
			&\lesssim T^{\frac{q-(p+2)}{q}}		\left(\|u\|_{L^{q}(\mathbb{R}_{+};L^{r})}^{p}+\|\tilde{u}\|_{L^{q}(\mathbb{R}_{+};L^{r})}^{p}\right) d(u,\tilde{u}),
		\end{aligned}
	\end{equation*}
	leading to
	\begin{equation}\label{lapse}
		d\,(\Gamma u,\Gamma \tilde{u})\leq \widetilde{C}\, T^{\frac{q-(p+2)}{q}}
		\left(\|u\|_{X}^{p}+\|\tilde{u}\|_{X}^{p}\right) d(u,\tilde{u}),
	\end{equation}
	where $\widetilde{C}$ is $T$-independent.
	
	Let $M=2CR$.
	The condition $0<p<\frac{4}{n}$ ensures $q-(p+2)>0$, allowing us to select  $T_{*}>0$ satisfying
	\begin{equation}\label{lane}
		C\,T_{*}^{\frac{q-(p+2)}{q}}M^{p}\leq \frac{1}{2}\ \ \text{and} \ \ 2\,\widetilde{C}\,T_{*}^{\frac{q-(p+2)}{q}}M^{p}\leq \frac{1}{2}.
	\end{equation}
	For $\|u\|_{X}\leq M$, inequalities \eqref{laptop}, \eqref{lapse} and \eqref{lane} imply
	\begin{equation*}
		\|\Gamma u\|_{X}\leq M\quad \text{and} \quad d\left(\Gamma u,\Gamma \tilde{u}\right)\leq\frac{1}{2}\,d(u,\tilde{u}).
	\end{equation*}
	Therefore, $\Gamma$ is a contraction on $B[0,M]\subset X$, yielding a fixed point $u\in X$ of $\Gamma$. Finally, it follows from Proposition \ref{abss} that $u\in C\left(\left[0,T\right];L^{2}(\mathbb{R}^{n}_{+})\right)\cap C_{b}\left(\overline{\mathbb{R}_{+}\mkern-4mu}\,;\mathcal{H}^{0}(\mathbb{R}^{n-1}\times(0,T))\right)$.
\end{proof}
\subsubsection{Local $H^{s}$ solutions for $s\in(0,2)$}
We now consider $s\in(0,2)$ with $s\leq\frac{n}{2}$. We extend  Audiard's result \cite{facilitate} (which treated only the case $s=1$) to the $s\in(0,2)$ under a constraint  of the same form $p<\frac{4}{n-2s}$. To address the difficulties caused by  the time derivative of the nonlinear term, we establish two precise parameter claims  (\eqref{qwu}--\eqref{qwui} and \eqref{die}--\eqref{dier}) to bound  the norm of the nonlinearity.

\begin{proof}[Proof of Theorem \ref{chaste} for $s\in(0,2)$ with $s\leq \frac{n}{2}$]
	The solution space $X$ is defined by
	\begin{equation}
		X:=L^{\infty}(\mathbb{R}_{+};H^{s}(\mathbb{R}_{+}^n))\cap L^{q}(\mathbb{R}_{+};W^{s,r}(\mathbb{R}_{+}^n))\cap B^{\frac{s}{2}}_{q,2}(\mathbb{R}_{+};L^{r}(\mathbb{R}_{+}^n)).
	\end{equation}
	The closed ball in $X$ is complete under the metric $d(u,\tilde{u})=\|u-\tilde{u}\|_{L^{\infty}(\mathbb{R}_{+};L^{2})}+\|u-\tilde{u}\|_{L^{q}(\mathbb{R}_{+};L^{r})}$ (see Chapter 4 of \cite{facetious}). Let $\psi\in C_{c}^{\infty}(\mathbb{R})$, the cut-off function $\psi_{T}$ in \eqref{ahha} satisfies $\|\psi_{T}\|_{\dot{B}^{\frac{s}{2}}_{q,2}(\mathbb{R})}\approx T^{\frac{1}{q}-\frac{s}{2}}$.
	
	We first consider $s\in(0,1]$ with $s<\frac{n}{2}$. For this case, we choose the admissible pair
	$(q,r)=\left(\frac{4(p+2)}{p(n-2s)},\frac{n(p+2)}{n+sp}\right)$. From \eqref{confine}, we have
	\begin{equation}\label{orr}
		\|\Gamma u\|_{X}\lesssim \|u_{0}\|_{H^{s}}+\|h_{0}\|_{\mathcal{H}^{s}}+\|\psi_{T}F(u)\|_{L^{q'}(\mathbb{R}_{+};W^{s,r'})}+\|\psi_{T}F(u)\|_{B^{\frac{s}{2}}_{q',2}(\mathbb{R}_{+};L^{r'})}.
	\end{equation}
	From \eqref{riesz} and $W^{s,r}(\mathbb{R}_{+}^{n})\hookrightarrow L^{\frac{nr}{n-sr}}(\mathbb{R}_{+}^{n})$, we obtain
	\begin{equation*}
		\|F(u)(t)\|_{W^{s,r'}}\lesssim \|u(t)\|_{L^{\frac{nr}{n-sr}}}^{p} \|u(t)\|_{W^{s,r}}\lesssim  \|u(t)\|_{W^{s,r}}^{p+1}.
	\end{equation*}
	Then, applying H\"{o}lder's inequality yields
	\begin{equation}\label{bigf}
		\|\psi_{T} F(u)\|_{L^{q'}(\mathbb{R}_{+};W^{s,r'})}\lesssim\|\psi_{T}\|_{L^{\frac{q}{q-(p+2)}}(\mathbb{R}_{+})}  \|u\|_{L^{q}(\mathbb{R}_{+};W^{s,r})}^{p+1}\lesssim T^{\frac{q-(p+2)}{q}} \|u\|_{X}^{p+1}.
	\end{equation}
	By definition \eqref{denounce}, we have $\|\psi_{T}F(u)\|_{\dot{B}^{\frac{s}{2}}_{q',2}(\mathbb{R}_{+};L^{r'})}\lesssim A_{1}+A_{2}$, where
	\begin{align}
		&A_{1}=\left(\int_{\mathbb{R}} \left\|\psi_{T}(t+h)\left[ F(u)(t+h)-F(u)(t)\right]\right\|_{L^{q'}(I_{h};L^{r'})}^{2}  \frac { dh}{{|h|^{1+2s}}}\right)^{\frac{1}{2}},\label{a11}\\
		&A_{2}=\left(\int_{\mathbb{R}}  \left\| \left[\psi_{T}(t+h)-\psi_{T}(t) \right] F(u)(t)\right\|_{L^{q'}(I_{h};L^{r'})}^{2} \frac{dh}{|h|^{1+2s}}\right)^{\frac{1}{2}},\label{a2}
	\end{align}
	and $I_{h}=\{t\in\mathbb{R}_{+}: t+h\in\mathbb{R}_{+}\}$. For $A_{1}$, applying the \eqref{uproar} and  H\"{o}lder's
	inequality in space-time with the identities $\frac{1}{r'}=\frac{(n-sr)p}{nr}+\frac{1}{r}$ and $\frac{1}{q'}=\frac{q-(p+2)}{q}+\frac{p}{q}+\frac{1}{q}$, we derive
	\begin{equation*}
		\begin{aligned}
			&\left\|\psi_{T}(t+h)\left[ F(u)(t+h)-F(u)(t)\right]\right\|_{L^{q'}(I_{h};L^{r'})}\\
			\lesssim&\,
			\|\psi_{T}\|_{L^{\frac{q}{q-(p+2)}}(\mathbb{R}_+)}  \left(\|u(\cdot+h)\|_{L^{q}(\mathbb{R}_{+};W^{s,r})}^{p}+ \|u(\cdot)\| ^{p}_{L^{q}(\mathbb{R}_{+};W^{s,r})}\right) \|u(t+h)-u(t)\|_{L^{q}(I_{h};L^{r})},
		\end{aligned}
	\end{equation*}
	leading to
	\begin{equation}\label{a1}
		A_{1}\lesssim  T^{\frac{q-(p+2)}{q}} \|u\|_{L^{q}(\mathbb{R}_{+};W^{s,r})}^{p}\|u\|_{\dot{B}^{\frac{s}{2}}_{q,2}(\mathbb{R}_{+};L^{r})}.
	\end{equation}
	where $q >(p+2)$ is guaranteed by $(n-2s)p<4$.
	
	For $A_{2}$, we present a parameter existence claim: Let $\alpha$ and $\beta$ be defined as
	\begin{equation}\label{qwu}
		\alpha=\frac{(n-2s)p^{2}+(n-4s)p-4s}{(n-2s)p(p+1)}, \quad \beta=\frac{(2-s)q-2}{2(p+1)},
	\end{equation}
	where $(s, p, n)$ satisfy the hypotheses of Theorem \ref{chaste}.  There exist
	$\theta\in\left(\max\{0,\alpha\},\min\{1,\beta\}\right)$, $q_{1}\in \left(r,\frac{nr}{n-sr}\right)$ and $q_{2}\in \left(2,\frac{2n}{n-2s}\right)$ such that
	\begin{equation}\label{qwui}
		\frac{1}{r'(p+1)}=\frac{\theta}{q_{1}}+\frac{1-\theta}{q_{2}}.
	\end{equation}
	Let us prove this claim. We first show that $\left(\max\{0,\alpha\},\min\{1,\beta\}\right)$ is non-empty: $(i)$ Under the condition $s\leq 1$, the inequality $(2-s)q-2\geq q-2>0$ implies $\beta>0$.
	$(ii)$ Through direct computation, we have $\alpha<1$.
	$(iii)$ The relation $\alpha<\beta$ is guaranteed by the constraint $(n-2s)p<4$. So we can select $\theta$ in the above open interval. Next, the existence of  $q_{1}$ and $q_{2}$ follows from verifying that
	\begin{equation}\label{indad} \theta\left(\frac{1}{r}-\frac{s}{n}\right)+(1-\theta)\left(\frac{1}{2}-\frac{s}{n}\right)-\frac{1}{r'(p+1)}<0\quad \text{and} \quad \frac{\theta}{r}+\frac{1-\theta}{2}-\frac{1}{r'(p+1)}>0.
	\end{equation}
	Through detailed computation using the expression of $r$, the two inequalities in \eqref{indad} are equivalent to
	\begin{equation*}
		\theta> \frac{(n-2s)\,p^{2}+(n-4s)\, p-4s}{(n-2s)\, p\,(p+1)}=\alpha\quad \text{and} \quad 2\theta sp\,(p+1)+(1-\theta)\,np\,(p+1)+2sp>0,
	\end{equation*}
	respectively. Finally, it follows from the Intermediate Value Theorem that we can choose $q_{1}$ and $q_{2}$ satisfying \eqref{qwui}.
	
	Based on the above claim, we can estimate $A_{2}$ under the condition $(n-2s)p<4$. By using \eqref{qwui} with interpolation inequality, we have
	\begin{equation*}
		\|F(u)(t)\|_{L^{r'}}=\|u(t)\|_{L^{r'(p+1)}}^{p+1}\leq \left( \|u(t)\|_{L^{q_{1}}}^{\theta} \|u(t)\|_{L^{q_{2}}}^{1-\theta}\right)^{p+1}\lesssim \left( \|u(t)\|_{W^{s,r}}^{\theta} \|u(t)\|_{H^{s}}^{1-\theta}\right)^{p+1}.
	\end{equation*}
	Then, an application of H\"{o}lder's inequality gives
	{\small \begin{equation}\label{a22}
			\left\| \left[\psi_{T}(t+h)-\psi_{T}(t) \right] F\left(u(t)\right)\right\|_{L^{q'}(I_{h};L^{r'})}
			\lesssim \|\psi_{T}(t+h)-\psi_{T}(t)\|_{L^{q^{*}} (\mathbb{R}_{+})}  \|u\|_{L^{q}(\mathbb{R}_{+};W^{s,r})}^{\theta(p+1)}\|u\|_{L^{\infty}(\mathbb{R}_{+};H^{s})}^{(1-\theta)(p+1)},
	\end{equation}}
	where $q^{*}=\frac{q}{q-\theta(p+1)-1}$. Substituting \eqref{a22} into \eqref{a2} yields
	\begin{equation}\label{aa2}
		A_{2}\lesssim\|\psi_{T}\|_{\dot{B}^{s}_{q^{*},2}} \|u\|_{X}^{p+1}\lesssim T^{\frac{1}{q^{*}} - \frac{s}{2}}\|u\|_{X}^{p+1}.
	\end{equation}
	Combining \eqref{orr}, \eqref{bigf}, \eqref{a1} and \eqref{aa2},  we obtain
	\begin{equation}\label{s1}
		\|\Gamma u\|_{X}\lesssim \|u_{0}\|_{H^{s}}+\|h_{0}\|_{\mathcal{H}^{s}}+\left(T^{\frac{q-(p+2)}{q}}+T^{\frac{1}{q^{*}}-\frac{s}{2}}\right)\|u\|_{X}^{p+1},
	\end{equation}
	where $q>(p+2)$ is guaranteed by $p<\frac{4}{n-2s}$ and $\frac{1}{q^{*}}-\frac{s}{2}>0$ is equivalent to the constraint  $\theta<\beta$.
	From similar computations applied to \eqref{uproar} and \eqref{intimate}, we get
	\begin{equation}\label{s12}
		d\left(\Gamma u,\Gamma\tilde{u}\right)\lesssim T^{\frac{q-(p+2)}{q}} \left(\|u\|_{L^{q}(\mathbb{R}_{+};W^{s,r})}^{p}+\|\tilde{u}\|_{L^{q}(\mathbb{R}_{+};W^{s,r})}^{p}\right)d(u,\tilde{u}).
	\end{equation}
	
	For $s\in(1,2)$ with $s<\frac{n}{2}$, we select another admissible pair $(q,r)=\left(\frac{4(p+1)}{p(n-2s)}, \frac{2n(p+1)}{n+2sp}\right)$.  The techniques employed are analogous to those used in the case of $s\in(0,1]$ with $s<\frac{n}{2}$; therefore, we shall only present certain crucial steps and inequalities. It follows from \eqref{confine} that
	\begin{equation}\label{extrude}
		\|\Gamma u\|_{X}\lesssim \|u_{0}\|_{H^{s}}+\|h_{0}\|_{\mathcal{H}^{s}}+\|\psi_{T}F(u)\|_{L^{1}(\mathbb{R}_{+};H^{s})}+\|\psi_{T}F(u)\|_{B^{\frac{s}{2}}_{1,2}(\mathbb{R}_{+};L^{2})}.
	\end{equation}
	Starting from the identity $\frac{1}{2}=\frac{(n-sr)p}{nr}+\frac{1}{r}$, and combining it with inequalities \eqref{riesz} and \eqref{uproar}, we derive the results
	\begin{equation*}
		\|\psi_{T}F(u)\|_{L^{1}(\mathbb{R}_{+};H^{s})}\lesssim T^{\frac{q-(p+1)}{q}} \|u\|_{L^{q}(\mathbb{R}_{+};W^{s,r})}^{p+1},\quad A_{1}\lesssim T^{\frac{q-(p+1)}{q}}\|u\|_{L^{q}(\mathbb{R}_{+};W^{s,r})}^{p}\|u\|_{\dot{B}^{\frac{s}{2}}_{q,2}(\mathbb{R}_{+};L^{r})},
	\end{equation*}
	where $q>p+1$ is ensured by $(n-2s)p<4$. Here, the expressions for $A_{1}$ and $A_{2}$ are structurally congruent to \eqref{a11} and \eqref{a2}, requiring only the substitution of $q'$ and $r'$ with $1$ and $2$ respectively. We introduce a similar claim for $A_{2}$ : Let
	\begin{equation}\label{die}
		\alpha=\frac{(n-2s)(p+1)-n}{(n-2s)p}\quad \text{and}\quad \beta= \frac{(2-s)q}{2(p+1)},
	\end{equation}
	where $(s, p, n)$ satisfy the assumptions in Theorem \ref{chaste}, then there exist  parameters $\theta\in$ $\left(\max\{0,\alpha\},\min\{1,\beta\}\right)$, $q_{1}\in \left(r,\frac{nr}{n-sr}\right)$ and $q_{2}\in \left(2,\frac{2n}{n-2s}\right)$ satisfying
	\begin{equation}\label{dier}
		\frac{1}{2(p+1)} =\frac{\theta}{q_{1}}+\frac{1-\theta}{q_{2}}.
	\end{equation}
	From \eqref{dier} and interpolation inequality, we can obtain
	\begin{equation*}	A_{2}\lesssim\|\psi_{T}\|_{\dot{B}^{s}_{q^{*},2}(\mathbb{R}_{+})}\|u\|_{X}^{p+1}\lesssim T^{\frac{1}{q^{*}}-\frac{s}{2}}\|u\|_{X}^{p+1},
	\end{equation*}
	where $q^{*}=\frac{q}{q-\theta(p+1)}$. Therefore, we have
	\begin{equation}
		\|\Gamma u\|_{X}\lesssim \|u_{0}\|_{H^{s}}+\|h_{0}\|_{\mathcal{H}^{s}}+\left(T^{\frac{q-(p+1)}{q}}+T^{\frac{1}{q^{*}}-\frac{s}{2}}\right)\|u\|_{X}^{p+1},
	\end{equation}
	and
	\begin{equation}\label{lave}
		d\,(\Gamma u,\Gamma\tilde{u})\lesssim T^{\frac{q-(p+1)}{q}} \left(\|u\|_{L^{q}(\mathbb{R}_{+};W^{s,r})}^{p}+\|\tilde{u}\|_{L^{q}(\mathbb{R}_{+};W^{s,r})}^{p}\right)d(u,\tilde{u}),
	\end{equation}
	where $p<\frac{4}{n-2s}$ implies $q>(p+1)$ and $\frac{1}{q^{*}}-\frac{s}{2}>0$ is equivalent to $\theta<\beta$.
	
	Finally, consider the case \(s = \frac{n}{2}\), which can be handled by a slight modification to use the framework for \(s < \frac{n}{2}\) above. In fact, only two cases \(n=2, s=1\) and \(n=3, s=\frac{3}{2}\) are involved, and we choose \(\varepsilon > 0\) such that \(p < \frac{4}{n-2(s-\varepsilon)}\). For the former, take $(q, r) = \left( \frac{4(p+2)}{p(n-2(s-\varepsilon))}, \frac{n(p+2)}{n+(s-\varepsilon)p} \right)$ to carry out the analysis from \eqref{orr} to \eqref{s12}. For the latter, take $(q, r) = \left( \frac{4(p+1)}{p(n-2(s-\varepsilon))}, \frac{2n(p+1)}{n+2(s-\varepsilon)p} \right)$, thereby applying the treatment from \eqref{extrude} to \eqref{lave}.
\end{proof}
\subsubsection{Local $H^{s}$ solutions for $s\in[2,\frac{5}{2})$}
In the subsequent local solution theory, we fix $\psi=1$,  so that the operator $\Gamma$ defined in \eqref{ahha} reduces to
\begin{equation}\label{turbid}
	\begin{aligned}
		\Gamma u(t)=&\,e^{it\Delta}Eu_{0}+i\lambda\int_{0}^{t}e^{i(t-\tau)\Delta}EF(u(\tau))d\tau\\	&+S(t)\left(h_{0}-e^{i(\cdot)\Delta}Eu_{0}|_{x_{n}=0}-i\lambda\int_{0}^{\cdot}e^{i(\cdot-\tau)\Delta}EF(u(\tau))d\tau|_{x_{n}=0}\right).
	\end{aligned}
\end{equation}
\begin{lemma}\label{assimilate}
	Assume the parameters $(s,p,n)$ satisfy one of the cases:
	\begin{enumerate}[label=(\roman*)]
		\item  $s\in\left(\frac{3}{2},2\right]$ and $s<\frac{n}{2}$ for  $\frac{4-2s}{n-2s}<p<\frac{4}{n-2s}$ (or $s=\frac{n}{2}$ for $0<p<\infty$).
		\item $s\in\left(2,\frac{5}{2}\right)$ and $s<\frac{n}{2}$ for $1\leq p<\frac{4}{n-2s}$.
	\end{enumerate}
	Then there exist $\theta\in(0,1)$ such that for all $\varphi, \psi\in H^{s}(\mathbb{R}_{+}^{n})$, the following inequality holds
	\begin{equation}\label{cen}
		\|F(\varphi)-F(\psi)\|_{H^{s-2}}\lesssim \left(\|\varphi\|_{H^{s}}^{p+1-\theta}+\|\psi\|_{H^{s}}^{p+1-\theta}\right)
		\|\varphi-\psi\|_{H^{s-2}}^{\theta}.
	\end{equation}
\end{lemma}
\begin{proof}		For case $(i)$: $s \in \left(\frac{3}{2},2\right]$, the embedding $L^{\frac{2n}{n+2(2-s)}}(\mathbb{R}_{+}^{n})\hookrightarrow H^{s-2}(\mathbb{R}_{+}^{n})$  is valid because $s-2\in \left(-\frac{1}{2},0\right]$. Combining this with \eqref{uproar} yields
	\begin{equation}\label{shfi}
		\|F(\varphi)-F(\psi)\|_{H^{s-2}}\lesssim \|F(\varphi)-F(\psi)\|_{L^{\frac{2n}{n+2(2-s)}}}\lesssim
		\left\|\left(|\varphi|^{p}+|\psi|^{p}\right)\left|\varphi-\psi\right|\right\|_{L^{\frac{2n}{n+2(2-s)}}}.
	\end{equation}
	For parameters $(s,p,n)$, the inequality
	\begin{equation*}
		\frac{n-2s}{2n}<\frac{n+2(2-s)}{2n}-\frac{np-2s p}{2n}<\frac{1}{2}
	\end{equation*}
	ensures that there exists $\varepsilon\in \left(0,\frac{1}{100}\right)$ such that
	\begin{equation}\label{guar}		\frac{n-2(s-\varepsilon)}{2n}<\frac{n+2(2-s)}{2n}-\frac{np-2(s-\varepsilon)p}{2n}<\frac{1}{2}.
	\end{equation}
	This leads to the decomposition
	$\frac{n+2(2-s)}{2n}=\frac{1}{\tilde{p}}+\frac{1}{\tilde{q}},$ where
	\begin{equation*}
		\frac{1}{\tilde{p}}=\frac{np-2(s-\varepsilon)p}{2n},\quad \frac{1}{\tilde{q}}=\frac{n+2(2-s)}{2n}-\frac{np-2(s-\varepsilon)p}{2n}.
	\end{equation*}
	Note that the Sobolev embedding $H^{s}(\mathbb{R}_{+}^{n})\hookrightarrow H^{s-\varepsilon}(\mathbb{R}_{+}^{n})\hookrightarrow L^{\frac{2n}{n-2(s-\varepsilon)}}(\mathbb{R}_{+}^{n})$ holds. From  inequality \eqref{guar}, we obtain $H^{s-\varepsilon}(\mathbb{R}_{+}^{n})\hookrightarrow L^{\tilde{q}}(\mathbb{R}_{+}^{n})$. Applying H\"{o}lder's inequality to \eqref{shfi} and leveraging these two embeddings, we have
	\begin{equation}\label{udsf}
		\|F(\varphi)-F(\psi)\|_{H^{s-2}}\lesssim \left(\|\varphi\|_{H^{s}}^{p}+\|\psi\|_{H^{s}}^{p}\right)\|\varphi-\psi\|_{H^{s-\varepsilon}}.
	\end{equation}
	By the Gagliardo-Nirenberg's inequality, there exists $\theta\in (0,1)$ such that
	\begin{equation}\label{sf}
		\|\varphi-\psi\|_{H^{s-\varepsilon}}\lesssim \|\varphi-\psi\|_{H^{s}}^{1-\theta} \|\varphi-\psi\|_{H^{s-2}}^{\theta}\lesssim\left( \|\varphi\|_{H^{s}}^{1-\theta} +\|\psi\|_{H^{s}}^{1-\theta}\right)\|\varphi-\psi\|_{H^{s-2}}^{\theta}.
	\end{equation}
	Substituting \eqref{sf} into \eqref{udsf}, we obtain \eqref{cen}.

	For case (ii): $s\in \left(2,\frac{5}{2}\right)$, the parameter relation $\frac{n-4}{2n}<\frac{1}{2}-\frac{np-2s p}{2n}<\frac{1}{2}$
	ensures a perturbation parameter $\varepsilon\in\left(0,\frac{1}{100}\right)$ maintaining
	\begin{equation}\label{oppress} 0<\frac{n-2(2-\varepsilon)}{2n}<\frac{1}{2}-\frac{np-2(s-\varepsilon)p}{2n}<\frac{1}{2},
	\end{equation}
	where the left inequality follows from $n>4$.
	We therefore decompose $\frac{1}{2}=\frac{1}{p_{1}}+\frac{1}{q_{1}}$, where
	\begin{equation*}			\frac{1}{p_{1}}=\frac{1}{2}-\frac{np-2(s-\varepsilon)p}{2n},\quad\frac{1}{q_{1}}=\frac{np-2(s-\varepsilon)p}{2n}.
	\end{equation*}
	This decomposition can also be reformulated as $\frac{1}{2}=\frac{1}{p_{2}}+\frac{1}{q_{2}}$, where
	\begin{equation*} \frac{1}{p_{2}}=\frac{n-2(s-\varepsilon)}{2n},\quad\frac{1}{q_{2}}=\frac{1}{2}-\frac{n-2(s-\varepsilon)}{2n}.
	\end{equation*}
	By sequentially using \eqref{slate}, \eqref{trot}, and \eqref{riesz} with $\frac{1}{q_{2}}
	=\frac{1}{p_{1}}+\frac{p-1}{p_{2}}$, we obtain
	{\small	\begin{equation*}
			\begin{aligned}
				\|F(\varphi)-F(\psi)\|_{H^{s-2}}
				\lesssim&\|\varphi-\psi\|_{W^{s-2,p_{1}}}
				\sup_{\rho\in[0,1]} \left\|(1-\rho)\varphi+\rho\,\psi\right\|_{L^{\frac{2n}{n-2(s-\varepsilon)}}} ^{p}    \\
				&+ \|\varphi-\psi\|_{L^{\frac{2n}{n-2(s-\varepsilon)}}} \sup_{\rho\in[0,1]} \left\|(1-\rho)\varphi+\rho\,\psi\right\|_{L^{\frac{2n}{n-2(s-\varepsilon)}}}^{p-1} \left\|(1-\rho)\varphi+\rho\,\psi\right\|_{W^{s-2,p_{1}}} \\					\lesssim&\left(\|\varphi\|_{H^{s}}^{p}+\|\psi\|_{H^{s}}^{p}\right) \|\varphi-\psi\|_{H^{s-\varepsilon}}.
			\end{aligned}
	\end{equation*}}
	Here, the Sobolev embeddings $H^{s-\varepsilon}(\mathbb{R}_{+}^{n})\hookrightarrow L^{\frac{2n}{n-2(s-\varepsilon)}}(\mathbb{R}_{+}^{n})$ and $H^{s-\varepsilon}(\mathbb{R}_{+}^{n})\hookrightarrow W^{s-2,p_{1}}(\mathbb{R}_{+}^{n})$ are employed, where the latter is guaranteed by \eqref{oppress}.
	Similarly, applying the Gagliardo-Nirenberg's inequality yields \eqref{cen}.
\end{proof}

\begin{proof}[Proof of Theorem \ref{chaste} for $s\in[2,\frac{5}{2})$ with $s\leq \frac{n}{2}$]
	We define the solution space
	\begin{equation}\label{afh}
		X_{T}:=L^{q}(0,T;W^{s,r})\cap W^{1,q}(0,T;W^{s-2,r})\cap L^{\infty}(0,T;H^{s})
		\cap W^{1,\infty}(0,T;H^{s-2}).
	\end{equation}
	The set $E\subset X_{T}$, given by
	\begin{equation*}
		E=\{ u\in X_{T} |\, u(0)=u_{0}, \|u\|_{X_{T}}\leq M\}
	\end{equation*}
	is complete under the metric $d(u,\tilde{u})=\|u-\tilde{u}\|_{L^{\infty}(0,T;L^{2})}+\|u-\tilde{u}\|_{L^{q}(0,T;L^{r})}$.
	
	For $s<\frac{n}{2}$, we choose the admissible pair $(q,r)=\left(\frac{4(p+1)}{p(n-2s)}, \frac{2n(p+1)}{n+2sp}\right)$. Using the estimates \eqref{dilute} and \eqref{ratify}, we have
	\begin{equation}\label{ab1}
		\|\Gamma u\|_{X_{T}}\lesssim \|u_{0}\|_{H^{s}}+\|h_{0}\|_{\mathcal{H}^{s}}+\|F(u)(0)\|_{H^{s-2}}+\|F(u)\|_{L^{1}(0,T;H^{s})}+\|F(u)\|_{W^{1,1}(0,T;H^{s-2})}.
	\end{equation}
	Applying \eqref{riesz} with $\frac{1}{2}=\frac{1}{r}+\frac{(n-sr)p}{nr}$ and H\"{o}lder's inequality in the time variable, we obtain
	\begin{equation}\label{ab3}
		\|F(u)\|_{L^{1}(0,T;H^{s})}\lesssim T^{\frac{q-(p+1)}{q}}\|u\|_{L^{q}(0,T;W^{s,r})}^{p+1}.
	\end{equation}
	We can also rewrite $\frac{1}{2}=\frac{n-(s-2)r}{nr}+\frac{1}{r^{*}}$, where $\frac{1}{r^{*}}=\frac{n-2r}{nr}+\frac{(n-sr)(p-1)}{nr}$. By sequentially using inequalities \eqref{slate}, \eqref{trot} and \eqref{riesz}, we get
	{\small \begin{equation}\label{adorn}
			\begin{aligned}
				\|F(\varphi)-F(\psi)\|_{H^{s-2}}
				\lesssim &\, \|\varphi-\psi\|_{W^{s-2,r}}\sup_{\rho\in[0,1]}\|(1-\rho)\varphi+\rho\psi\|_{L^{\frac{nr}{n-sr}}}^{p}\\
				&+\|\varphi-\psi\|_{L^{\frac{nr}{n-(s-2)r}}}\sup_{\rho\in[0,1]}\|(1-\rho)\varphi+\rho\psi\|_{L^{\frac{nr}{n-sr}}}^{p-1}
				\|(1-\rho)\varphi+\rho\psi\|_{W^{s-2,\frac{nr}{n-2r}}} \\
				\lesssim  & \left(\|\varphi\|_{W^{s,r}}^{p}+\|\psi\|_{W^{s,r}}^{p}\right)\|\varphi-\psi\|_{W^{s-2,r}}.
			\end{aligned}
	\end{equation}}
	Here, the embedding $W^{s,r}(\mathbb{R}_{+}^{n})\hookrightarrow W^{s-2,\frac{nr}{n-2r}}(\mathbb{R}_{+}^{n})$ is used. Since $u\in W^{1,q}(0,T;W^{s-2,r})$, Lemma \ref{relay} implies that for any $V\subset\subset (0,T)$ and $|h|<d(V,\{0,T\})$, the inequality  $\|u(\cdot+h)-u(\cdot)\|_{L^{q}(V;W^{s-2,r})}\leq\|\partial_{t}u\|_{L^{q}(0,T;W^{s-2,r})}|h|$ holds. Then from estimate \eqref{adorn}, we have
	\begin{equation}\label{abs}
		\begin{aligned}
			\|F(u)(\cdot+h)-F(u)(\cdot)\|_{L^{\frac{q}{p+1}}(V;H^{s-2})}&\lesssim \|u\|_{L^{q}(0,T;W^{s,r})}^{p} \|u(\cdot+h)-u(\cdot)\|_{L^{q}(V;W^{s-2,r})}\\&\lesssim\|u\|_{L^{q}(0,T;W^{s,r})}^{p}\|\partial_{t}u\|_{L^{q}(0,T;W^{s-2,r})}|h|.
		\end{aligned}
	\end{equation}
	Since $W^{s-2,r}$ is reflexive, by \eqref{abs} and Lemma \ref{relay}, we obtain $F(u)\in W^{1,\frac{q}{p+1}}(0,T;H^{s-2})$ with
	\begin{equation}\label{ab4}
		\|F(u)\|_{\dot{W}^{1,1}(0,T;H^{s-2})}\lesssim T^{\frac{q-(p+1)}{q}} \|u\|_{L^{q}(0,T;W^{s,r})}^{p} \|u\|_{\dot{W}^{1,q}(0,T;W^{s-2,r})}.
	\end{equation}
	An application of Lemma \ref{assimilate} with $\varphi = u(t)$ and $\psi = u_0$ gives
	\begin{equation*}
		\|F(u)(t)-F(u_{0})\|_{H^{s-2}}\lesssim \left(\|u(t)\|_{H^{s}}^{p+1-\theta}+\|u_{0}\|_{H^{s}}^{p+1-\theta}\right) \|u(t)-u_{0}\|_{H^{s-2}}^{\theta},
	\end{equation*}
	where $\theta \in (0,1)$. Consequently, $u\in L^{\infty}(0,T;H^{s})\cap W^{1,\infty}(0,T;H^{s-2})$ implies $F(u)(0)=F(u_{0})$. Moreover, we have
	\begin{equation}\label{ab2}
		\|F(u)(0)\|_{H^{s-2}}\lesssim \|u_{0}\|_{H^{s}}^{p+1},
	\end{equation}
	which follows analogously by setting $\varphi = 0$ and $\psi=u_{0}$. Substituting \eqref{ab3}, \eqref{ab4} and \eqref{ab2} into \eqref{ab1}, we have
	\begin{equation}
		\|\Gamma u\|_{X_{T}}\lesssim \|u_{0}\|_{H^{s}}+\|h_{0}\|_{\mathcal{H}^{s}}+\|u_{0}\|_{H^{s}}^{p+1}+T^{\frac{q-(p+1)}{q}}\|u\|_{X_{T}}^{p+1},
	\end{equation}
	where $q > p+1$ is ensured by $p < \frac{4}{n-2s}$.	Similarly, from \eqref{uproar} and \eqref{intimate} we have
	\begin{equation}\label{afhh}
		d\left(\Gamma u,\Gamma \tilde{u}\right)\lesssim T^{\frac{q-(p+1)}{q}} \left(\|u\|_{L^{q}(0,T;W^{s,r})}^{p}+\|\tilde{u}\|_{L^{q}(0,T;W^{s,r})}^{p}\right) d(u,\tilde{u}).
	\end{equation}	
	
	For $s=\frac{n}{2}$, we need only consider $n=4$. Choosing $\varepsilon>0$ such that $p<\frac{4}{n-2(s-\varepsilon)}$, we can take $(q, r)=\left(\frac{4(p+1)}{p(n-2(s-\varepsilon))},\frac{2n(p+1)}{n+2(s-\varepsilon)p}\right)$ to apply the framework above and obtain the corresponding result.
\end{proof}
\subsection{High regularity local solutions}
This section addresses the high regularity case $s > \frac{n}{2}$ of Theorem \ref{chaste}. Since $n\geq2$, it suffices to consider $s \in \left(1, \frac{5}{2}\right)$. We continue to use the operator $\Gamma$ defined in \eqref{turbid}.
\begin{proof}[Proof of Theorem \ref{chaste} for \text{$s\in \left(1,\frac{3}{2}\right]$} with $s> \frac{n}{2}$] In this case, we only need to treat $n=2$. Let $$X_{T}:=L^{\infty}(0,T;H^{s}(\mathbb{R}_{+}^n)),$$ its subset $E=\{ u\in X_{T}: \|u\|_{X_{T}}\leq M \}$ with metric $d(u,\tilde{u})=\|u-\tilde{u}\|_{L^{\infty}(0,T;L^{2})}$. Since $s>1$ and $n=2$, by Lemma \ref{accessory} we have $\|F(u)(t)\|_{H^{s}}\lesssim \|u(t)\|_{L^{\infty}}^{p}\|u(t)\|_{H^{s}}$. Then from \eqref{buffoon} with $(\alpha,\beta)=(\frac{3}{3-s},2)$, we obtain
	\[
	\|\Gamma u\|_{X_{T}}\lesssim  \|u_{0}\|+\|h_{0}\|_{\mathcal{H}^{s}}+\langle T\rangle^{\frac{s}{3}}  T^{\frac{3-s}{3}}\|u\|_{X_{T}}^{p+1},
	\]
	and
	\[
	d\left(\Gamma u, \Gamma \tilde{u}\right)\lesssim \langle T\rangle^{\frac{s}{3}}  T^{\frac{3-s}{3}}\left(\|u\|_{X_{T}}^{p}+\|\tilde{u}\|_{X_{T}}^{p}\right)d(u,\tilde{u}).
	\]
\end{proof}
Before addressing the case $s\in \left(\frac{3}{2},\frac{5}{2}\right)$ with $s>\frac{n}{2}$, we present a necessary lemma, which also guarantees the regularity analysis in Section \ref{sec4}.
\begin{lemma}\label{imp}
	Let $s\in \left(\frac{3}{2},\frac{5}{2}\right)\setminus\{2\}$ with $1\leq p<\infty$ or $s=2$ with $0< p<\infty$,  then for all $\varphi(x,\theta)\in L^{\infty}_{\theta}(0,1; W_{x}^{|s-2|,\frac{n}{|s-2|}}\cap L_{x}^{\infty})$, $\psi\in H^{s-2}(\mathbb{R}_{+}^{n})$,
	the following estimate holds:
	{\small\begin{equation}\label{obstruct}
			\left\|\int_{0}^{1}\left|\varphi(\cdot,\theta)\right|^{p}d\theta\,\psi\right\|_{H^{s-2}}+\,\left\|\int_{0}^{1}\left|\varphi(\cdot,\theta)\right|^{p-2}\varphi^{2}(\cdot,\theta)d\theta\,\psi\right\|_{H^{s-2}}\lesssim \left\|\varphi\right\|_{L^{\infty}(0,1; W^{|s-2|,\frac{n}{|s-2|}}\cap L^{\infty})}^{p}
			\left\|\psi\right\|_{H^{s-2}}.
	\end{equation}}
\end{lemma}
\begin{proof}
	We note that it suffices to prove the case $s-2\geq0$, since $s-2<0$ follows from a duality argument. Using the estimates \eqref{trot} and \eqref{riesz}, we have
	\begin{equation*}
		\begin{aligned}
			\left\|\int_{0}^{1}|\varphi(\cdot,\theta)|^{p}d\theta\, \psi\right\|_{H^{s-2}}&\lesssim \left\||\varphi|^{p}\right\|_{L^{\infty}(0,1;W^{s-2,\frac{n}{s-2}})}\|\psi\|_{L^{\frac{2n}{n-2(s-2)}}}+\left\||\varphi|^{p}\right\|_{L^{\infty}(0,1;L^{\infty})} \|\psi\|_{H^{s-2}}\\
			&\lesssim \left(\|\varphi\|^{p-1}_{L^{\infty}(0,1;L^{\infty})}\|\varphi\|_{L^{\infty}(0,1;W^{s-2,\frac{n}{s-2}})}+\|\varphi\|^{p}_{L^{\infty}(0,1;L^{\infty})} \right)\|\psi\|_{H^{s-2}},
		\end{aligned}
	\end{equation*}
	where we have used the embedding $H^{s-2}(\mathbb{R}_{+}^{n})\hookrightarrow L^{\frac{2n}{n-2(s-2)}}(\mathbb{R}_{+}^{n})$.
	We can use a similar approach to estimate $\left\|\int_{0}^{1}|\varphi(\cdot,\theta)|^{p-2} \varphi^{2}(\cdot,\theta) d\theta\, \psi\right\|_{H^{s-2}}$, thus establishing the inequality~\eqref{obstruct}.
\end{proof}

\begin{proof}[Proof of Theorem \ref{chaste} for $s\in\left(\frac{3}{2},\frac{5}{2}\right)$ with $s>\frac{n}{2}$]
	We define the solution space
	\begin{equation}
		X_{T}:=L^{\infty}(0,T;H^{s}(\mathbb{R}_{+}^n))\cap W^{1,\infty}(0,T;H^{s-2}(\mathbb{R}_{+}^n)),
	\end{equation}
	and its subset
	\begin{equation*}
		E=\{ u\in X_{T}|\, u(0)=u_{0}, \|u\|_{X_{T}}\leq M \}
	\end{equation*}
	with metric $d(u,\tilde{u})=\|u-\tilde{u}\|_{L^{\infty}(0,T;L^{2})}$. It follows from \eqref{aggravate} that
	\begin{equation}\label{311}
		\begin{aligned}
			\|\Gamma			u\|_{X_{T}}\lesssim&\,\|u_{0}\|_{H^{s}}+\|h_{0}\|_{\mathcal{H}^{s}}+\|F(u)(0)\|_{H^{s-2}}\\
			&+\langle T\rangle ^{\frac{1}{2}}(1+T^{-\frac{2s-3}{4}})\left(\|F(u)\|_{L^{2}(0,T;H^{s})}+\|F(u)\|_{H^{1}(0,T;H^{s-2})}\right).
		\end{aligned}	
	\end{equation}
	It is easy to see from \eqref{riesz} that
	\begin{equation}\label{314}
		\|F(u)\|_{L^{2}(0,T;H^{s})}\lesssim T^{\frac{1}{2}} \|u\|_{L^{\infty}(0,T;H^{s})}^{p+1}.
	\end{equation}
	Applying equation \eqref{slate} and Lemma \ref{imp}, we obtain \begin{equation}\label{impo}
		\|F(u)(0)\|_{H^{s-2}}=\|F(u_{0})\|_{H^{s-2}}\lesssim \|u_{0}\|_{H^{s}}^{p+1},
	\end{equation}
	\begin{equation}
		\|F(u)(t+h)-F(u)(t)\|_{H^{s-2}}
		\lesssim\left(\|u(t+h)\|_{H^{s}}^{p}+\|u(t)\|_{H^{s}}^{p}\right) \|u(t+h)-u(t)\|_{H^{s-2}}.
	\end{equation}
	Here, the embedding $H^{s}(\mathbb{R}_{+}^{n})\hookrightarrow W^{|s-2|,\frac{n}{|s-2|}}(\mathbb{R}_{+}^{n})\cap L^{\infty}(\mathbb{R}_{+}^{n})$ is employed. Since $u\in W^{1,\infty}(0,T;H^{s-2})$, Lemma \ref{relay} implies that for any $V\subset\subset (0,T)$ and $|h|<d(V,\{0,T\})$, the inequality
	\begin{equation*}
		\|F(u)(\cdot+h)-F(u)(\cdot)\|_{L^{\infty}(V;H^{s-2})}\lesssim \|u\|_{L^{\infty}(0,T;H^{s})}^{p} \|\partial_{t}u\|_{L^{\infty}(0,T;H^{s-2})} |h|
	\end{equation*}
	holds. Moreover, we have $F(u)\in W^{1,\infty}(0,T;H^{s-2})$ and
	\begin{equation}\label{uahi}
		\|F(u)\|_{\dot{H}^{1}(0,T;H^{s-2})}\lesssim T^{\frac{1}{2}} \|u\|_{L^{\infty}(0,T;H^{s})}^{p}\|u\|_{\dot{W}^{1,\infty}(0,T;H^{s-2})}.
	\end{equation}
	Substituting \eqref{314}, \eqref{impo} and \eqref{uahi} into \eqref{311}, we have
	\begin{equation}\label{319}
		\|\Gamma u\|_{X_{T}}\lesssim \|u_{0}\|_{H^{s}}+\|u_{0}\|_{H^{s}}^{p+1}+\|h_{0}\|_{\mathcal{H}^{s}} +\langle T\rangle^{\frac{1}{2}} (T^{\frac{1}{2}}+T^{\frac{5-2s}{4}})\|u\|_{X_{T}}^{p+1}.
	\end{equation}
	Similarly, from \eqref{uproar} and  \eqref{intimate} we obtain
	\begin{equation}
		d(\Gamma u,\Gamma \tilde{u})\lesssim T \left(\|u\|_{L^{\infty}(0,T;H^{s})}^{p}+\|\tilde{u}\|_{L^{\infty}(0,T;H^{s})}^{p}\right) d(u,\tilde{u}).
	\end{equation}
\end{proof}	
Thus, we complete the proof of local existence for all cases in Theorem \ref{chaste}. Regarding uniqueness, we have the following proof.
\begin{proof}[Proof of Theorem \ref{chaste} for uniqueness]
	For any $T>0$, suppose the solutions $u,\tilde{u}$ belong to $C\left([0,T];H^{s}(\mathbb{R}^{n}_{+})\right)$.	For $s<1$ and $p\leq \frac{2+2s}{n-2s}$ (with $p<\frac{2+2s}{n-2s}$ when $n=2$), there exists $\rho\in \big[2,\frac{2n}{n-2}\big]$ (and $\rho\in\big[2,\frac{2n}{n-2}\big)$ when $n=2$) such that $
	\rho'(p+1) \in \left[2, \frac{2n}{n-2s}\right]$. The following argument is presented for $n>2$; the case $n=2$ requires only minor adjustments to the parameters. To choose such $\rho$, we consider two cases:
	\begin{enumerate}[label=(\roman*)]
		\item When $p < \frac{2}{n}$, we can choose $\rho \in \big[2, \frac{2n}{n-2}\big]$ such that $\rho'(p+1) = 2$.
		\item When $\frac{2}{n} \leq p \leq \frac{2+2s}{n-2s}$, taking $\rho = \frac{2n}{n-2}$, we can verify that $\rho'(p+1) \in \big[2, \frac{2n}{n-2s}\big]$.
	\end{enumerate}
	Consequently, we can choose $\gamma \in [2, \infty]$ such that $(\gamma, \rho)$ is an admissible pair. Then, from $u \in L^{\infty}(0, T; H^{s})$, it follows that
	\begin{equation*}
		F(u) \in L^{\gamma'}(0, T; L^{\rho'(p+1)}(\mathbb{R}^{n}_{+})).
	\end{equation*}
	The condition $s < 1$ also implies $\rho'(p+1) \in \big[2, \frac{2n}{n-2}\big)$. Next, choose $q \in (2, \infty]$ such that $\big(q, \rho'(p+1)\big)$ is an admissible pair. By  \eqref{intimate} and H\"{o}lder's inequality, we have for any $t\in[0,T]$,
	\begin{equation}
		\|u - \tilde{u}\|_{L^{q}(0, t; L^{\rho'(p+1)})} \lesssim \big( \|u\|_{L^{\infty}(0, T; H^{s})}^{p} + \|\tilde{u}\|_{L^{\infty}(0, T; H^{s})}^{p} \big) \|u - \tilde{u}\|_{L^{\gamma'}(0, t; L^{\rho'(p+1)})}.
	\end{equation}
	Note that $q > \gamma'$; then using \eqref{holmer}, we conclude that $u = \tilde{u}$ on $[0, T]$.
	
	For $s\geq1 $,  the embedding
	\[L^{\infty}\left(0,T;H^{s}(\mathbb{R}^{n}_{+})\right)\hookrightarrow L^{q}\left(0,T;L^{r}(\mathbb{R}^{n}_{+})\right)\] holds for any admissible pair $(q,r)$. When  $s\leq \frac{n}{2}$, we choose an $\varepsilon\geq 0$ such that $p<\frac{4}{n-2(s-\varepsilon)}$ and $s-\varepsilon\geq0$; then take the admissible pair $(q,r)=(\frac{8}{n-2(s-\varepsilon)},\frac{4n}{2n-(n-2(s-\varepsilon))p})$. Using \eqref{intimate}, \eqref{uproar} and $\frac{1}{r'}=\frac{np-2(s-\varepsilon)p}{2n}+\frac{1}{r}$, we obtain for any $t\in[0,T]$,
	\begin{equation}\label{dsf}
		\begin{aligned}
			\|u-\tilde{u}\|_{L^{q}(0,t;L^{r})}&\lesssim \|F(u)-F(\tilde{u})\|_{L^{q'}(0,t;L^{r'})}\\
			&\lesssim  \left(\|u\|_{L^{\infty}(0,T;H^{s-\varepsilon})}^{p} + \|\tilde{u}\|_{L^{\infty}(0,T;H^{s-\varepsilon})}^{p}\right)\|u-\tilde{u}\|_{L^{q'}(0,t;L^{r})}.
		\end{aligned}
	\end{equation}
	When $s>\frac{n}{2}$, selecting the admissible pair $(q,r)=(\infty, 2)$  and using $H^{s}(\mathbb{R}_{+}^{n})\hookrightarrow L^{\infty}(\mathbb{R}_{+}^{n})$ yields an estimate analogous to \eqref{dsf}. Note that in any case we have $q>2$, which implies $q>q'$. It then follows from \eqref{holmer} that $u = \tilde{u}$ on $[0, T]$.
	
	In other cases, if we further assume $u, \tilde{u} \in L^{q}(0,T; L^{r}(\mathbb{R}^{n}_{+}))$ for $(q, r)$ chosen as in the case $s \geq 1$, then an estimate similar to \eqref{dsf} holds, and thereby uniqueness follows.
\end{proof}
The maximal existence interval $[0,T_{\max})$ and blow-up statement for the solution $u$ follow from the classical extension procedure. Moreover, the local solution given by the fixed point argument after extension belongs to $ L_{\mathrm{loc}}^{q}\left([0,T_{\max});W^{s,r}(\mathbb{R}^{n}_{+})\right)$ for some admissible pair $(q,r)$. This property can be  extended to arbitrary admissible pairs, since the estimates in Section \ref{sec24} hold for all such pairs. Under the condition $p\geq \lceil s\rceil$ or $p\in2\mathbb{Z}_{+}$, continuous dependence follows by an argument parallel to the existence proof.
\subsection{Additional local results}\label{stipple}
The preceding two sections present the main body of the local theory, namely the proof of Theorem \ref{chaste}. This section provides some additional local results, which will be effectively applied in the later global analysis. Recall that Proposition \ref{weak} provides an additional $H^{2}$-estimate \eqref{begrudge} by using the $\mathcal{W}^{2}$ framework \eqref{huaw}--\eqref{invoice}. As stated in Remark \ref{remaa}, although Theorem \ref{chaste} does not yield solutions satisfying the boundary data in  $\mathcal{H}^2$ for $ p < 1$, we can use \eqref{begrudge} to construct  solutions satisfying the boundary data in  $\mathcal{W}^2$. Specifically, we have the following proposition.
\begin{proposition}\label{weak2}
	Let $0<p<1$ and $(n-4)p<4$. For $u_{0}\in H^{2}(\mathbb{R}^{n}_{+})$ and $h_{0}\in\mathcal{H}^{2}(\mathbb{R}_{+}^{n})$ satisfying the compatibility condition, there exists a unique maximal solution
	\begin{equation}\label{ew}
		u\in C\left(\left[0,T_{\max}\right);H^{2}(\mathbb{R}^{n}_{+})\right)\cap C\left(\overline{\mathbb{R}_{+}\mkern-4mu}\,;\mathcal{W}^{2}_{\mathrm{loc}}(\mathbb{R}^{n-1}\times(0,T_{\max}))\right)
	\end{equation}
	to the IBVP \eqref{ibvp}. Moreover, if $T_{\max}< \infty$, then $\lim\limits_{t\rightarrow T_{\max}} \left\|u(t)\right\|_{H^{2}(\mathbb{R}^{n}_{+})}=\infty$.
\end{proposition}
\begin{proof}
	The solution space is defined as
	\begin{equation}
		X_{T}:=L^{\infty}(0,T;H^{2}(\mathbb{R}^{n}_{+}))\cap W^{1,\infty}(0,T;L^{2}(\mathbb{R}^{n}_{+}))\cap W^{1,q}(0,T;L^{r}(\mathbb{R}^{n}_{+})),
	\end{equation}
	with its subset
	\begin{equation*}
		E=\{ u\in X_{T}|\,u(0)=u_{0},\|u\|_{X_{T}}\leq M\},
	\end{equation*}
	where the metric is given by $d(u,\tilde{u})=\|u-\tilde{u}\|_{L^{\infty}(0,T;L^{2})}+\|u-\tilde{u}\|_{L^{q}(0,T;L^{r})}$.
	
	For $n\geq4$, there exists an $\varepsilon>0$ such that $p<\frac{4}{n-2(2-\varepsilon)}$ ensures the  pair $(q,r)=\left(\frac{8}{n-2(2-\varepsilon)},\frac{4n}{2n-(n-2(2-\varepsilon))p}\right)$ is admissible.  It follows from \eqref{begrudge} that
	\begin{equation}\label{rhetoric}
		\|\Gamma u\|_{X_{T}}\lesssim \|u_{0}\|_{H^{2}}+\|h_{0}\|_{\mathcal{H}^{2}}+\|F(u)\|_{C([0,T];L^{2})}+\|F(u)\|_{W^{1,q'}(0,T;L^{r'})}.
	\end{equation}
	Since $u \in L^\infty(0,T;H^2) \cap W^{1,\infty}(0,T;L^2)$, an application of Lemma \ref{assimilate} with $\varphi = u(t)$ and $\psi = u(\tau)$ yields
	$F(u) \in B^\theta_{\infty,\infty}(0,T;L^2(\mathbb{R}_{+}^{n})) \hookrightarrow C([0,T];L^2(\mathbb{R}_{+}^{n}))$.  Additionally, we have $\|F(u)(0)\|_{L^2} \lesssim \|u_0\|_{H^2}^{p+1}$ analogously. Combining Lemmas \ref{assimilate} and \ref{relay}, we derive
	\begin{equation*}
		\begin{aligned}
			\|F(u)(t)-F(u)(0)\|_{L^{2}}&\lesssim   \left(\|u(t)\|_{H^{2}}^{p+1-\theta}+\|u_{0}\|_{H^{2}}^{p+1-\theta}\right) \|u(t)-u_{0}\|_{L^{2}}^{\theta}\\&\lesssim  t^{\theta}\left( \|u\|_{L^{\infty}(0,T;H^{2})}^{p+1-\theta}+\|u_{0}\|_{H^{2}}^{p+1-\theta}\right)\|\partial_{t}u\|_{L^{\infty}(0,T;L^{2})}^{\theta}.
		\end{aligned}
	\end{equation*}
	Therefore, we obtain
	\begin{equation}\label{chamber}
		\|F(u)\|_{C([0,T];L^{2})}\lesssim  T^{\theta} \left(\|u\|_{L^{\infty}(0,T;H^{2})}^{p+1-\theta} +\|u_{0}\|_{H^{2}}^{p+1-\theta}\right)\|\partial_{t}u\|_{L^{\infty}(0,T;L^{2})}^{\theta}+\|u_{0}\|_{H^{2}}^{p+1}.
	\end{equation}
	From \eqref{uproar} and
	$\frac{1}{r'}= \frac{np-2(2-\varepsilon)p}{2n}+\frac{1}{r}$, we have
	\begin{equation}\label{agitation}
		\|F(\varphi)-F(\psi)\|_{L^{r'}}\lesssim \left(\|\varphi\|_{H^{2}}^{p}+ \|\psi\|_{H^{2}}^{p}\right) \|\varphi-\psi\|_{L^{r}}.
	\end{equation}
	By substituting $\varphi=u(t+h)$ and $\psi=u(t)$ into \eqref{agitation} and applying Lemma \ref{relay}, we get
	\begin{equation}\label{suffrage}
		\|F(u)\|_{W^{1,q'}(0,T;L^{r'})}\lesssim T^{1-\frac{2}{q}} \|u\|_{L^{\infty}(0,T;H^{2})}^{p} \|u\|_{W^{1,q}(0,T;L^{r})}.
	\end{equation}
	Substituting \eqref{chamber} and \eqref{suffrage} into \eqref{rhetoric} yields
	\begin{equation}\label{longevity}
		\|\Gamma u\|_{X_{T}}\lesssim	 \|u_{0}\|_{H^{2}}+\|h_{0}\|_{\mathcal{H}^{2}}+\|u_{0}\|_{H^{2}}^{p+1}+\left(T^{1-\frac{2}{q}}+T^{\theta}\right) \left(\|u\|_{X_{T}}^{p+1-\theta}+\|u_{0}\|_{H^{2}}^{p+1-\theta}\right)\|u\|_{X_{T}}^{\theta},
	\end{equation}
	where $q>2$. Similarly, from \eqref{uproar} and \eqref{intimate} we obtain
	\begin{equation}
		d\,(\Gamma u,\Gamma \tilde{u})\lesssim T^{1-\frac{2}{q}} \left(\|u\|_{X_{T}}^{p}+\|\tilde{u}\|_{X_{T}}^{p}\right)d\,(u,\tilde{u}).
	\end{equation}
	
	For $n<4$, the embedding $H^{2}(\mathbb{R}_{+}^{n})\hookrightarrow L^{\infty}(\mathbb{R}_{+}^{n})$ holds, which allows us to choose the  admissible pair $(q, r)=(\infty, 2)$. The desired estimates then follow by similar arguments as in the case $n\geq4$. The uniqueness argument is analogous to that of Theorem~\ref{chaste}.
\end{proof}

Finally, based on Theorem \ref{chaste} and Proposition \ref{weak2}, we establish the following two corollaries that present some additional properties.
\begin{corollary}\label{core}
	The solution $u$ provided by Theorem \ref{chaste} satisfies the following properties:
	\begin{enumerate}[label=(\roman*)]
		\item For $s\in \left(\frac{3}{2},\frac{5}{2}\right)$, $u\in C^{1}\left([0,T_{\max});H^{s-2}(\mathbb{R}^{n}_{+})\right)$.
		\item The local $H^{1}$ solution $u$ belongs to $C^{1} \left(\overline{\mathbb{R}_{+}\mkern-4mu}\,; \mathcal{H}^{0}_{\mathrm{loc}}(\mathbb{R}^{n-1}\times (0,T_{\max}))\right)$.
	\end{enumerate}
\end{corollary}
\begin{proof}
	Since $u\in C\left([0,T_{\max});H^{s}\right)$, we apply Lemma \ref{assimilate} for $s\leq \frac{n}{2}$ and Lemma \ref{imp} for $s>\frac{n}{2}$ to deduce  $F(u)\in C\left([0,T_{\max});H^{s-2}\right)$.
	Property (i) then follows by combining this with the equation in \eqref{ibvp}, while property (ii) is a direct consequence of \eqref{contagious}.
\end{proof}
\begin{corollary}\label{core1}
	The solution given by Proposition \ref{weak2} satisfies $u\in C^{1}([0,T_{\max});L^{2}(\mathbb{R}^{n}_{+}))$.
\end{corollary}

\section{Global well-posedness for the nonlinear Schr\"{o}dinger equation}\label{sec4}
\subsection{$H^{1}$ a priori estimates}\label{sec41}
We first establish the $H^{1}$ a priori estimate for the solution $u$ of the IBVP \eqref{ibvp}.
\begin{proposition}\label{maritime}
	Let $0 < p \leq \frac{4}{n-2}$ if $\lambda < 0$, and $0 < p \leq \frac{4}{n+1}$ if $\lambda > 0$. For any $T>0$, there exists a continuous nondecreasing function $\alpha_T : \overline{\mathbb{R}_{+}\mkern-4mu} \to \overline{\mathbb{R}_{+}\mkern-4mu}$ with $\alpha_T(0) = 0$, such that every solution $u$ to the IBVP \eqref{ibvp} with the regularity:
	\begin{align}
		t&\mapsto u(\cdot,t) \in C\left([0,T);H^{2}(\mathbb{R}^{n}_{+})\right)\cap C^{1}\left([0,T);L^{2}(\mathbb{R}^{n}_{+})\right), \label{reg} \\
		x_{n}&\mapsto u(\cdot,x_{n},\cdot) \in C\left(\,\overline{\mathbb{R}_{+}\mkern-4mu}\,;{H}_{\mathrm{loc}}^{1}(\mathbb{R}^{n-1}\times(0,T))\right)\cap C^{1}\left(\,\overline{\mathbb{R}_{+}\mkern-4mu}\,;{L}_{\mathrm{loc}}^{2}(\mathbb{R}^{n-1}\times(0,T))\right),\label{reg1}
	\end{align}
	satisfies the estimate:
	\begin{equation}\label{drone}
		\sup_{t \in [0,T)} \|u(t)\|_{H^1(\mathbb{R}^n_+)} \leq\alpha_T \left( \|u_0\|_{H^1(\mathbb{R}^n_+)} + \|h_0\|_{H^1(\mathbb{R}^{n-1} \times (0,T))} \right).
	\end{equation}
\end{proposition}
\begin{proof}
	We employ the approach of Bona et al. \cite{figment} to establish our result. For simplicity, we write $L^{p}_{t}=L^{p}(\mathbb{R}^{n-1}\times(0,t))$ and $H^{1}_{t}=H^{1}(\mathbb{R}^{n-1}\times(0,t))$. We start with the following three identities for a solution $u$ to \eqref{ibvp}:
	\begin{align} \partial_{t} |u|^{2}&=-2\,\text{Im}\nabla\cdot (\overline{u}\, \nabla u ) ,\label{dove} \\ \partial_{n}\left(|\partial_{n}u|^{2}-|\nabla'u|^{2}+\frac{2\lambda}{p+2}|u|^{p+2}\right)&=-2\,\text{Re}\nabla'\cdot (\overline{\partial_{n}u}\, \nabla' u )-i\partial_{t}(u\overline{\partial_{n}u})+i\partial_{n}(u\overline{\partial_{t}u}),\label{concur} \\ \partial_{t}\left(|\nabla u|^{2}-\frac{2\lambda}{p+2}|u|^{p+2}\right)&=2\,\text{Re}\nabla\cdot (\overline{\partial_{t}u}\nabla u),\label{cult}
	\end{align}
	where $\nabla'$ is the gradient on $\mathbb{R}^{n-1}$. Identity \eqref{dove} can be derived by multiplying both sides of the equation in \eqref{ibvp} by $2\overline{u}$, then taking imaginary parts.  Identities \eqref{concur} and \eqref{cult} follow from multiplying by $2\overline{\partial_{n}u}$ and $2\overline{\partial_{t}u}$, respectively, then taking the real part. Integrating both sides of the identity \eqref{dove} on $\mathbb{R}^{n}_{+}\times[0,t]$ for $t\in[0,T)$, we obtain
	\begin{equation}\label{endow}
		\begin{aligned}	\int_{\mathbb{R}^{n}_{+}}|u(x,t)|^{2}dx&=\int_{\mathbb{R}^{n}_{+}}|u(x,0)|^{2}dx+\int_{\mathbb{R}^{n-1}}\int_{0}^{t}2\,\text{Im}\,[\overline{u}\partial_{n}u](x',0,\tau)d\tau dx'\\&\leq \|u_{0}\|_{L^2}^{2}+2\|h_{0}\|_{L^{2}_{t}}\|\partial_{n}u(\cdot,0,\cdot)\|_{L^{2}_{t}}.
		\end{aligned}
	\end{equation}
	To estimate $\|\partial_{n}u(\cdot,0,\cdot)\|_{L^{2}_{t}}$, we integrate \eqref{concur} on $\mathbb{R}^{n}_{+}\times[0,t]$:
	\begin{equation}\label{fsj}
		\begin{aligned} &\int_{\mathbb{R}^{n-1}}\int_{0}^{t}(|\partial_{n}u|^{2}-|\nabla'u|^{2})(x',0,\tau)d\tau dx'+\frac{2\lambda}{p+2}\int_{\mathbb{R}^{n-1}}\int_{0}^{t}\left|u(x',0,\tau)\right|^{p+2}d\tau dx'\\				=&-i\int_{\mathbb{R}^{n}_{+}}(u\overline{\partial_{n}u})(x,t)-(u\overline{\partial_{n}u})(x,0)dx-i\int_{\mathbb{R}^{n-1}}\int_{0}^{t}(u\overline{\partial_{t}u})(x',0,\tau)d\tau dx'.
		\end{aligned}
	\end{equation}
	Applying H\"{o}lder's inequality to \eqref{fsj} yields
	\begin{equation}\label{naive}
		\begin{aligned}
			\|\partial_{n}u(\cdot,0,\cdot)\|_{L^{2}_{t}}^{2} \leq & \,\|u(t)\|_{L^{2}}\|\partial_{n}u
			(t)\|_{L^{2}}+\|u_{0}\|_{L^{2}}\|\partial_{n}u_{0}\|_{L^{2}}\\
			&+\|h_{0}\|_{L^{2}_{t}}\|\partial_{t}h_{0}\|_{L^{2}_{t}}+\|h_{0}\|_{H^{1}_{t}}^{2}+\frac{2|\lambda|}{p+2}\|h_{0}\|_{L^{p+2}_{t}}^{p+2}.		\end{aligned}	
	\end{equation}
	Substituting \eqref{naive} into \eqref{endow} and using Young's inequality, we obtain
	\begin{equation*}
		\begin{aligned}
			\|u(t)\|_{L^{2}}^{2}&\leq 2\|h_{0}\|_{L^{2}_{t}}\|u(t)\|_{L^2}^{\frac{1}{2}}\|\partial_{n}u(t)\|_{L^2}^{\frac{1}{2}}+C_{1}(t)\\
			&\leq\frac{1}{4}\|u(t)\|_{L^2}^{2}+\frac{3}{2}\|\partial_{n}u(t)\|_{L^2}^{\frac{2}{3}}\|h_{0}\|_{L^2_t}^{\frac{4}{3}}+C_{1}(t),
		\end{aligned}	
	\end{equation*}
	where $C_{1}(t)=2\|u_{0}\|_{H^{1}}^{2}+3\|h_{0}\|^{2}_{H^{1}_{t}}+\frac{2|\lambda|}{p+2}\|h_{0}\|_{L_{t}^{p+2}}^{p+2}$. Therefore, we have
	\begin{equation}\label{smite} \|u(t)\|_{L^2}^{2}\leq2\|\partial_{n}u(t)\|_{L^2}^{\frac{2}{3}}\|h_{0}\|_{L^2_t}^{\frac{4}{3}}+\frac{4}{3}C_{1}(t).
	\end{equation}
	
	By integrating \eqref{cult} on $\mathbb{R}^{n}_{+}\times[0,t]$, we obtain
	\begin{equation}\label{gallop}
		\begin{aligned} &\int_{\mathbb{R}^{n}_{+}}|\nabla u|^{2}(x,t)-|\nabla u|^{2}(x,0)dx-\frac{2\lambda}{p+2}\int_{\mathbb{R}^{n}_{+}}|u|^{p+2}(x,t)-|u|^{p+2}(x,0)dx\\	=&\int_{\mathbb{R}^{n-1}}\int_{0}^{t}-2\,\text{Re}[\overline
			{\partial_{t}u}\partial_{n}u](x',0,\tau)d\tau dx'.
		\end{aligned}
	\end{equation}
	We note that the validity of both the differential equations \eqref{dove}--\eqref{cult} and the integral equations in \eqref{endow}--\eqref{fsj}, \eqref{gallop} is guaranteed by the regularity assumptions \eqref{reg}--\eqref{reg1}.
	
	For $\lambda<0$, from \eqref{gallop} we have
	\begin{equation}\label{411}
		\begin{aligned}	\|u(t)\|_{\dot{H}^1}^{2}\leq\|\partial_{n}u(\cdot,0,\cdot)\|_{L^{2}_{t}}^{2}+\|u_{0}\|_{H^{1}}^{2}+\frac{2|\lambda|}{p+2}\|u_{0}\|_{L^{p+2}}^{p+2}+\|\partial_{t}h_{0}\|_{L^2_t}^{2}.
		\end{aligned}
	\end{equation}
	By substituting \eqref{naive} into \eqref{411} and applying \eqref{smite} together with the Young's inequality, we derive
	\begin{equation}\label{napkin}
		\begin{aligned}
			\|u(t)\|_{\dot{H}^1}^{2}&\leq \|u(t)\|_{L^2}\|\partial_{n}u(t)\|_{L^2}+C_{1}(t)+C_{2}(t) \\ &\leq\left(\frac{1}{2}\|\partial_{n}u(t)\|_{L^2}^{\frac{4}{3}}\right)\left(2^{\frac{3}{2}}\|h_{0}\|_{L^2_t}^{\frac{2}{3}}\right)+\sqrt{\frac{4}{3}C_{1}(t)}\|\partial_{n}u(t)\|_{L^2}+C_{1}(t)+C_{2}(t)\\ &\leq\frac{1}{3}\|\partial_{n}u(t)\|_{L^2}^{2}+\frac{2^{\frac{9}{2}}}{3}\|h_{0}\|_{L^2_t}^{2} +\frac{1}{2}\|\partial_{n}u(t)\|_{L^2}^{2}+\frac{5}{3}C_{1}(t)+C_{2}(t)\\
			&\leq\frac{5}{6}\|\partial_{n}u(t)\|_{L^2}^{2}+C_{3}(t),
		\end{aligned}
	\end{equation}
	where  $C_{2}(t)=\|u_{0}\|_{H^{1}}^{2}+\frac{2|\lambda|}{p+2}\|u_{0}\|_{L^{p+2}}^{p+2}+\|h_{0}\|_{H^{1}_{t}}^{2}, $ $C_{3}(t)=10\left(C_{1}(t)+C_{2}(t)\right)$. The Sobolev embedding $H^{1}\hookrightarrow L^{p+2}$ (valid for the given range of $p$) guarantees the estimate for $\|u(t)\|_{\dot{H}^1}$. Combined with the estimate \eqref{smite}, this completes the proof for the case $\lambda < 0$.
	
	For $\lambda>0$, \eqref{gallop} implies that we need to estimate $\|u(t)\|_{L^{p+2}}$. By using Gagliardo-Nirenberg inequality and \eqref{smite}, we have
	\begin{equation}\label{grease}
		\begin{aligned}
			\|u(t)\|_{L^{p+2}}^{p+2}&\lesssim \|u(t)\|_{\dot{H}^{1}}^{\frac{np}{2}} \|u(t)\|_{L^{2}}^{2+p-\frac{np}{2}}\\
			&\lesssim \|u(t)\|_{\dot{H}^{1}}^{\frac{np}{2}} \left(\|\partial_{n}u(t)\|_{L^2}^{\frac{2+p}{3}-\frac{np}{6}}\|h_{0}\|_{L^2_t}^{\frac{4+2p-np}{3}}+C_{1}(t)^{\frac{4+2p-np}{4}}\right).
		\end{aligned}
	\end{equation}
	Recalling \eqref{napkin}, then from \eqref{grease} we have
	\begin{equation}\label{abide}
		\|u(t)\|_{\dot{H}^{1}}^{2}\leq \frac{5}{6}\|\partial_{n}u(t)\|_{L^2}^{2}+ C\left(\|u(t)\|_{\dot{H}^{1}}^{\frac{2+p+np}{3}}\|h_{0}\|_{L^{2}_{t}}^{\frac{4+2p-np}{3}}+\|u(t)\|_{\dot{H}^{1}}^{\frac{np}{2}}C_{1}(t)^{\frac{4+2p-np}{4}}\right)+C_{3}(t),
	\end{equation}
	where $C>0$ is a constant.		When $0<p<\frac{4}{n+1}$, we have $\frac{2+p+np}{3}$, $\frac{np}{2}<2$, then an application of Young's inequality gives
	\begin{equation}
		\|u(t)\|_{\dot{H}^{1}}^{2}\leq \left(\frac{5}{6}+\varepsilon\right)\|u(t)\|_{\dot{H}^{1}}^{2}+C({\varepsilon}) \left(C_{1}(t)^{\beta_{1}}+C_{1}(t)^{\beta_{2}}\right) +C_{3}(t),
	\end{equation}
	where $\beta_{1},\beta_{2}>0$.	
	
	When $p=\frac{4}{n+1}$, we have $\frac{2+p+np}{3}=2$ and $\frac{np}{2}<2$. Applying the Young's inequality to \eqref{abide}, we obtain
	\begin{equation*}
		\|u(t)\|_{\dot{H}^{1}}^{2}\leq \left(\frac{5}{6}+C \|h_{0}\|_{L^{2}_{t}}^{\nu_{1}}\right)   \|u(t)\|_{\dot{H}^{1}}^{2}+\varepsilon\|u(t)\|_{\dot{H}^{1}}^{2} + C(\varepsilon) C_{1}(t)^{\nu_{2}} +C_{3}(t),
	\end{equation*}
	where $\nu_{1},\nu_{2}>0$. Choosing $\varepsilon = \frac{1}{12}$ yields
	\begin{equation*}	\left(\frac{1}{12}-C\|h_{0}\|_{L^{2}_{t}}^{\nu_{1}}\right)\|u(t)\|_{\dot{H}^{1}}^{2}\lesssim C_{1}(t)^{\nu_{2}}+C_{3}(t).
	\end{equation*}
	On the other hand, we can choose $\delta>0$ such that for all $\rho, \tau \in (0,T)$ with $|\rho-\tau|<\delta$,  $\|h_{0}\|^{\nu_{1}}_{L^{2}(\mathbb{R}^{n-1}\times(\rho,\tau))}\leq \tfrac{1}{24C}$. Set $t_{j}=\frac{jT}{\lfloor\frac{T}{\delta}\rfloor+1}$ ($j=0, 1, \dots , \lfloor\frac{T}{\delta}\rfloor+1$), then $\|h_{0}\|^{\nu_{1}}_{L^{2}(\mathbb{R}^{n-1}\times(t_{j},t_{j+1}))}\leq \frac{1}{24C}$.  First estimate $\|u(t)\|_{H^{1}}$ on $[0,t_{1}]$, then take $u(t_{1})$ as new initial data to extend the estimate to $[t_{1},t_{2}]$. Iterating this process yields the bound for $\|u(t)\|_{H^{1}}$ over all $t \in [0,T)$.
\end{proof}

\subsection{Regularity analysis for $s\in[1,\frac{5}{2})$}\label{sec42}
Building upon the solution properties established in Theorem \ref{chaste} and Corollary~\ref{core}, we introduce the function space $\mathscr{X}^{s,p}_{T}$ to facilitate the subsequent analysis.
\begin{definition}\label{xspt}
	Let $s\in \left[0,\frac{5}{2}\right)$, $p\geq \lceil s\rceil-1$ with $(n-2s)p<4$, and $T\in (0,\infty]$.
	The space $\mathscr{X}^{s,p}_{T}$ consists of all functions
	\[u\in C\left([0,T);H^{s}(\mathbb{R}^{n}_{+})\right)\cap C\left(\overline{\mathbb{R}_{+}\mkern-4mu}\,;\mathcal{H}^{s}_{\mathrm{loc}}(\mathbb{R}^{n-1}\times(0,T))\right)\] that solve \eqref{ibvp} and additionally satisfy:
	\begin{enumerate}[label=(\arabic*)]
		\item For all $s \in \left[1,\frac{5}{2}\right)$,
		$u \in \bigcap_{(q,r)\, \text{admissible}} L^{q}_{\mathrm{loc}}\left(\left[0,T\right);W^{s,r}(\mathbb{R}^{n}_{+})\right)$;
		\item For $s\in \left(\frac{3}{2},\frac{5}{2}\right)$,  $u\in C^{1}\left([0,T);H^{s-2}(\mathbb{R}^{n}_{+})\right)$;
		\item For $s=1$, $u\in C^{1} \left(\overline{\mathbb{R}_{+}\mkern-4mu}\,; \mathcal{H}^{0}_{\mathrm{loc}}(\mathbb{R}^{n-1}\times(0,T))\right)$.
	\end{enumerate}
\end{definition}
We note that the above definition of $\mathscr{X}^{2,p}_{T}$ does not cover the case $0 < p < 1$, since in this range a solution satisfying the boundary data in $\mathcal{H}^{2}$ is absent. The requirement can, however, be weakened by replacing  $\mathcal{H}^{2}$ with $\mathcal{W}^{2}$.   In view of Proposition~\ref{weak2} and Corollary~\ref{core1}, we now define the corresponding space:
\begin{definition}\label{xspt1}
	For $0 < p < 1$ with $(n-4)p<4$ and $T\in (0,\infty]$,	the space $\mathscr{X}^{2,p}_{T}$ consists of all functions
	\begin{equation*}
		u\in C\left([0,T);H^{2}(\mathbb{R}^{n}_{+})\right)\cap C\left(\overline{\mathbb{R}_{+}\mkern-4mu}\,;\mathcal{W}^{2}_{\mathrm{loc}}(\mathbb{R}^{n-1}\times(0,T))\right)\cap C^{1}\left([0,T);L^{2}(\mathbb{R}^{n}_{+})\right)
	\end{equation*}
	that solve \eqref{ibvp}.
\end{definition}

Building upon Proposition \ref{mago}, the following Proposition \ref{zealous1}  establishes regularity properties in the space \(\mathscr{X}^{s,p}_{T}\) for \(s \in (1, 2)\). For the one-dimensional case (see \cite{figment}), the regularity properties for $s\in(1,2)$ is derived by nonlinear interpolation, which relies on the key embedding $H^{1}(\mathbb{R}_{+})\hookrightarrow L^{\infty}(\mathbb{R}_{+})$. However, for $n\geq2$ the embedding $H^{1}(\mathbb{R}_{+}^{n})\hookrightarrow L^{\infty}(\mathbb{R}_{+}^{n})$ fails,  thus this technique cannot be  generalized to higher dimensions.
For $s \in \left(1,\frac{3}{2}\right]$, we resolve the difficulties through the development a decomposition-recursive technique \eqref{421+}--\eqref{423+}. For $s \in \left(\frac{3}{2},2\right)$, the focus is on obtaining pointwise estimates for the time derivative of the nonlinear term. Lemma \ref{imp} provides the difference estimate \eqref{buddy}, and the $L^{p}$-mean continuity transforms them into pointwise one \eqref{ideology}.
\begin{proposition}\label{zealous1}
	Let $s \in (1,2)$ with $1 \leq p < \frac{8}{3(n-2)}$,
	$u_{0}\in H^{s}(\mathbb{R}^{n}_{+})$ and $h_{0}\in \mathcal{H}^{s}(\mathbb{R}_{+}^{n})$ satisfy the compatibility condition. For any $T \in (0,\infty]$, if $u\in\mathscr{X}^{1,p}_T$, then $u \in \mathscr{X}^{s,p}_T$.
\end{proposition}
\begin{proof}
	By Theorem \ref{chaste} and Corollary~\ref{core}, there exist $T_{\max}\in(0,\infty]$ and a solution belonging to $\mathscr{X}^{s,p}_{T_{\max}}$. By uniqueness, to establish $u \in \mathscr{X}^{s,p}_{T}$, it suffices to prove $T_{\max} \geq T$. We proceed by contradiction, assuming $T_{\max} < T$.

	We first analyze the case $s \in \left(1,\frac{3}{2}\right]$. When $n=2$, estimate \eqref{buffoon} with $\left(\alpha,\beta\right)=\left(\frac{3}{3-s},2\right)$ yields that, for $t\in [0,T_{\max})$, we have
	\begin{equation}\label{wallet}
		\|u(t)\|_{H^{s}}\leq \|u\|_{C([0,t];H^{s})} \lesssim \|u_{0}\|_{H^{s}}+\|h_{0}\|_{\mathcal{H}^{s}} + \langle t\rangle^{\frac{s}{3}} \|F(u)\|_{L^{\frac{3}{3-s}}(0,t;H^{s})}.
	\end{equation}
	Since $u\in \mathscr{X}^{1,p}_{T}$ and $T_{\max} < T$, Definition \ref{xspt} implies $u\in L^{q}\left(0,T_{\max};W^{1,r}\right)$ for all admissible pairs $(q,r)$. We choose $(q,r)=\left(3p,\frac{6}{3p-2}\right)$, which satisfies $r>2$. By Lemma~\ref{accessory} and the embedding  $W^{1,r}(\mathbb{R}_{+}^{n})\hookrightarrow L^{\infty}(\mathbb{R}_{+}^{n})$, we obtain
	\begin{equation}\label{purse}
		\left\|F(u)(t)\right\|_{H^{s}}\lesssim \|u(t)\|_{L^{\infty}}^{p}\|u(t)\|_{H^{s}}\lesssim \|u(t)\|_{W^{1,r}}^{p}\|u(t)\|_{H^{s}}.
	\end{equation}
	Substituting \eqref{purse} into \eqref{wallet}, we have
	\begin{equation}
		\|u(t)\|_{H^{s}}^{\frac{3}{3-s}}\lesssim \|u_{0}\|_{H^{s}}^{\frac{3}{3-s}}+\|h_{0}\|_{\mathcal{H}^{s}}^{\frac{3}{3-s}}+ \langle t\rangle^{\frac{s}{3-s}} \int_{0}^{t} \|u(\tau)\|_{W^{1,r}}^{\frac{3p}{3-s}}\|u(\tau)\|_{H^{s}}^{\frac{3}{3-s}}d\tau.
	\end{equation}
	An application of Gronwall's inequality gives
	\begin{equation}
		\|u(t)\|_{H^{s}}^{\frac{3}{3-s}}\lesssim \left(\|u_{0}\|_{H^{s}}^{\frac{3}{3-s}}+\|h_{0}\|_{\mathcal{H}^{s}}^{\frac{3}{3-s}}\right) \exp\left(C \langle t\rangle^{\frac{s}{3-s}}\int_{0}^{t}\|u(\tau)\|_{W^{1,r}}^{\frac{3p}{3-s}} d\tau\right),
	\end{equation}
	where $C>0$ is a constant. Since $\frac{3p}{3-s}< 3p=q$,  $\|u(t)\|_{H^{s}}$ is bounded on $[0,T_{\max})$.  This contradicts the blow-up criterion, proving $T_{\max} \geq T$.

	When $3\leq n\leq4$, applying \eqref{buffoon} with $(\alpha,\beta)=(2,\frac{6n}{3n+6-4s})$ and Lemma \ref{accessory} with $\frac{3n+6-4s}{6n}=\frac{1}{2}+\frac{3-2s}{3n}$, we obtain
	\begin{equation}\label{apt}
		\|u(t)\|_{H^{s}}^{2}\lesssim \|u_{0}\|_{H^{s}}^{2}+\|h_{0}\|_{\mathcal{H}^{s}}^{2}+ \langle t\rangle^{\frac{2s}{3}} \int_{0}^{t} \|u(\tau)\|_{L^{\frac{3np}{3-2s}}}^{2p} \|u(\tau)\|_{H^{s}}^{2}d\tau.
	\end{equation}
	We choose the admissible pair $(q,r) = \left(2p, \frac{2np}{np-2}\right)$ to estimate  $\|u(\tau)\|_{L^{\frac{3np}{3-2s}}}^{2p}$. The embedding $W^{1,r}(\mathbb{R}_{+}^{n})\hookrightarrow L^{\frac{3np}{3-2s}}(\mathbb{R}_{+}^{n})$ holds provided that
	\begin{equation*}
		\frac{1}{2}-\frac{1}{np}-\frac{1}{n} <\frac{3-2s}{3np}\leq \frac{1}{2}-\frac{1}{np},
	\end{equation*}
	which is equivalent to
	\begin{equation}\label{dislodge}
		\frac{12-4s}{3n}\leq	p<\frac{4(3-s)}{3(n-2)}.
	\end{equation}
	For any $1\leq p<\frac{8}{3(n-2)}$, there exists $s'\in\left(1,\frac{3}{2}\right)$ such that for all $s\in(1,s']$, the inequality \eqref{dislodge} is satisfied. Applying this  embedding and Gronwall's inequality to \eqref{apt}, we conclude $u\in\mathscr{X}^{s,p}_{T}$ for $s\in (1,s']$.
	While our current result holds only for certain small ranges of $s$ (particularly as $p$ approaches $\frac{8}{3(n-2)}$), the technique developed below will ultimately extend this to the full range $s \in \left(1,\frac{3}{2}\right]$.

	In order to establish $u\in \mathscr{X}^{s,p}_{T}$ for $s\in \left(1,\frac{3}{2}\right]$, we decompose the interval via $\left(1,\frac{3}{2}\right]=\bigcup_{k=1}^{\infty}\left(s_{k+1},s_{k}\right]$, where $s_{k}= 1+\frac{1}{2^{k}}$, $k\in \mathbb{Z}_{+}$.
	The proof reduces to showing that for all  $k\geq 1$, if $u\in\mathscr{X}^{s_{k+1},p}_{T}$, then  $u\in\mathscr{X}^{s,p}_{T}$ for $s\in (s_{k+1},s_{k}]$. The key embedding $W^{s_{k+1},r}(\mathbb{R}_{+}^{n})\hookrightarrow L^{\frac{3np}{3-2s}}(\mathbb{R}_{+}^{n})$ holds for $s\in(s_{k+1},s_{k}]$ provided that
	\begin{equation}\label{421+}		\frac{1}{2}-\frac{1}{np}-\frac{s_{k+1}}{n} <\frac{3-2s}{3np}\leq \frac{1}{2}-\frac{1}{np},
	\end{equation}
	which is equivalent to
	\begin{equation}\label{shoe}
		\frac{12-4s}{3n}\leq p<\frac{2^{k+2}(3-s)}{3(2^{k}(n-2)-1)}.
	\end{equation}
	For $s\in (s_{k+1},s_{k}]$, we verify:
	\begin{equation*}
		\begin{aligned}
			\frac{12-4s}{3n}\leq1\leq p&<\frac{8}{3(n-2)}=\frac{4\cdot 2^{k+1}}{3\cdot2^{k}(n-2)}\\
			&<\frac{4(2^{k+1}-1)}{3(2^{k}(n-2)-1)}
			=\frac{2^{k+2}(3-s_{k})}{3(2^{k}(n-2)-1)}\\
			&<\frac{2^{k+2}(3-s)}{3(2^{k}(n-2)-1)},
		\end{aligned}
	\end{equation*}
	confirming that \eqref{shoe} always holds for $k\geq1$. Combining \eqref{apt} with this embedding yields, for $s\in (s_{k+1},s_{k}]$,
	\begin{equation}\label{423+}
		\|u(t)\|_{H^{s}}^{2}\lesssim \|u_{0}\|_{H^{s}}^{2}+\|h_{0}\|_{\mathcal{H}^{s}}^{2}+\langle t\rangle^{\frac{2s}{3}}\int_{0}^{t}\|u(\tau)\|_{W^{s_{k+1},r}}^{2p} \|u(\tau)\|_{H^{s}}^{2}d\tau.
	\end{equation}
	This completes the proof for $s \in \left(1,\frac{3}{2}\right]$.

	For $s\in \left(\frac{3}{2},2\right)$, using \eqref{aggravate} and H\"{o}lder's inequality in time variable yields:
	\begin{equation}\label{detain}
		\begin{aligned}	
			\|u(t)\|_{H^{s}}+\|\partial_{t}u(t)\|_{H^{s-2}}&
			\lesssim\,\|u_{0}\|_{H^{s}}+\|h_{0}\|_{\mathcal{H}^{s}}+ \|F(u)(0)\|_{H^{s-2}}\\
			&+\left(\langle t\rangle^{\frac{3}{4}}+t^{1-\frac{s}{2}}\right)\|F(u)\|_{L^{4}(0,t;H^{s})}+t^{\frac{1}{4}}\|F(u)\|_{\dot{W}^{1,4}(0,t;H^{s-2})}.
		\end{aligned}
	\end{equation}
	From Lemmas \ref{assimilate} and \ref{imp}, we have
	\begin{equation}\label{lawsuit}
		\|F(u)(0)\|_{H^{s-2}}\lesssim \|u_{0}\|_{H^{s}}^{p+1}.
	\end{equation}
	When $2\leq n\leq3$, we choose $(q, r)=\left(5p, \frac{10np}{5np-4}\right)$ and note that $u\in L^{q}(0,T_{\max};W^{\frac{3}{2},r})$. Then the embedding \[W^{\frac{3}{2},r}(\mathbb{R}_{+}^{n})\hookrightarrow L^{\infty}(\mathbb{R}_{+}^{n})\cap W^{2-s,\frac{n}{2-s}}(\mathbb{R}_{+}^{n})\] holds since $2<r<\frac{n}{2-s}$.
	From Lemma \ref{accessory} and $W^{\frac{3}{2},r}(\mathbb{R}_{+}^{n})\hookrightarrow L^{\infty}(\mathbb{R}_{+}^{n})$, we have
	\begin{equation}\label{radiant}
		\left\|F(u)(t)\right\|_{H^{s}}\lesssim \|u(t)\|_{W^{\frac{3}{2},r}}^{p}\|u(t)\|_{H^{s}}.
	\end{equation}
	Using \eqref{slate}, Lemma \ref{imp} and $W^{\frac{3}{2},r}(\mathbb{R}_{+}^{n})\hookrightarrow W^{2-s,\frac{n}{2-s}}(\mathbb{R}_{+}^{n})$, we obtain
	\begin{equation}\label{buddy}
		\|F(u)(t+h)-F(u)(t)\|_{H^{s-2}}\lesssim \left(\|u(t+h)\|_{W^{\frac{3}{2},r}}^{p}+\|u(t)\|_{W^{\frac{3}{2},r}}^{p}\right) \|u(t+h)-u(t)\|_{H^{s-2}}.
	\end{equation}
	From Definition \ref{xspt}, we have $u\in C^{1}\left([0,T_{\max});H^{s-2}\right)$.
	Then we can conclude that $$F(u)\in W_{\mathrm{loc}}^{1,\frac{q}{p}}(0,T_{\max};H^{s-2})$$ by using \eqref{buddy} and Lemma \ref{relay}. Therefore, $F(u)$ is differentiable almost everywhere with respect to $t$. On the other hand, according to $L^{p}$-mean continuity.  and diagonal rule,  $u\in L^{q}_{\mathrm{loc}}(0,T;W^{\frac{3}{2},r})$ implies that there exists a sequence $\{h_{n}\}$ with $h_{n}\rightarrow0$ as $ n\rightarrow\infty$, such that for almost every $t\in (0,T)$,
	\begin{equation*}
		\lim\limits_{n\rightarrow\infty} \|u(t+h_{n})\|_{W^{\frac{3}{2},r}}=\|u(t)\|_{W^{\frac{3}{2},r}}.
	\end{equation*}
	Replacing $h$ in \eqref{buddy} with $h_{n}$, dividing both sides of the inequality by $h_{n}$, and letting $n\to\infty$ yields for a.e. $t\in(0,T_{\max})$:
	\begin{equation}\label{ideology} \left\|\partial_{t}F(u)(t)\right\|_{H^{s-2}}\lesssim\|u(t)\|_{W^{\frac{3}{2},r}}^{p}\left\|\partial_{t}u(t)\right\|_{H^{s-2}}.
	\end{equation}
	Substituting \eqref{lawsuit},  \eqref{radiant} and  \eqref{ideology} into \eqref{detain}, we have
	\begin{equation}
		\begin{aligned} \|u(t)\|_{H^{s}}^{4}+\|\partial_{t}u(t)\|_{H^{s-2}}^{4}\lesssim&\,\|u_{0}\|_{H^{s}}^{4}+\|h_{0}\|_{\mathcal{H}^{s}}^{4}+\|u_{0}\|_{H^{s}}^{4(p+1)}\\
			&+ \langle t\rangle^{3} \int_{0}^{t}\|u(\tau)\|_{W^{\frac{3}{2},r}}^{4p} \left(\|u(\tau)\|_{H^{s}}^{4}+\|\partial_{t}u(\tau)\|_{H^{s-2}}^{4}\right)d\tau.
		\end{aligned}
	\end{equation}
	An application of Gronwall's inequality implies
	\begin{equation}\label{defy} \|u(t)\|_{H^{s}}^{4}+\|\partial_{t}u(t)\|_{H^{s-2}}^{4}\lesssim  \left(\|u_{0}\|_{H^{s}}^{4}+\|h_{0}\|_{\mathcal{H}^{s}}^{4}+\|u_{0}\|_{H^{s}}^{4(p+1)}\right) \exp\left(C\langle t\rangle ^{3}\int_{0}^{t} \|u(\tau)\|_{W^{\frac{3}{2},r}}^{4p} d\tau\right),
	\end{equation}
	where $C>0$ is a constant. The condition $4p<5p=q$ ensures that $\|u(t)\|_{H^{s}}$ is bounded on $[0,T_{\max})$.
	
	When $n=4$, the required embedding fails, so the analysis splits into two cases based on the range of $s$. For $s\in \left(\frac{3}{2},\frac{7}{4}\right]$, from \eqref{aggravate} and H\"{o}lder's inequality, we have
	\begin{equation}\label{magnify}
		\begin{aligned}
			\|u(t)\|_{H^{s}}+\|\partial_{t}u(t)\|_{H^{s-2}}\lesssim &\, \|u_{0}\|_{H^{s}}+\|h_{0}\|_{\mathcal{H}^{s}}+\|F(u)(0)\|_{H^{s-2}} \\ &+\left(\langle t\rangle^{\frac{11}{8}}+t^{\frac{7}{8}-\frac{s}{2}}\right) \|F(u)\|_{L^{\frac{8}{3}}(0,t;H^{s})}
			+t^{\frac{1}{8}}\|F(u)\|_{\dot{W}^{1,\frac{8}{3}}(0,t;H^{s-2})}.
		\end{aligned}	
	\end{equation}
	The special admissible pair $(q,r)=\left(\frac{11}{3},\frac{11}{4}\right)$ ensures the embedding $$W^{\frac{3}{2},r}(\mathbb{R}_{+}^{n})\hookrightarrow L^{\infty}(\mathbb{R}_{+}^{n})\cap \,W^{2-s,\frac{n}{2-s}}(\mathbb{R}_{+}^{n}),$$ thus \eqref{lawsuit}, \eqref{radiant} and \eqref{ideology} remain valid. Substituting these estimates into \eqref{magnify} and applying Gronwall's inequality yields
	\begin{equation*}
		\|u(t)\|_{H^{s}}^{\frac{8}{3}}\lesssim  \left(\|u_{0}\|_{H^{s}}^{\frac{8}{3}}+\|h_{0}\|_{\mathcal{H}^{s}}^{\frac{8}{3}}+\|u_{0}\|_{H^{s}}^{\frac{8}{3}(p+1)}\right) \exp\left(C\langle t\rangle ^{\frac{11}{3}}\int_{0}^{t} \|u(\tau)\|_{W^{\frac{3}{2},r}}^{\frac{8}{3}p} d\tau\right).
	\end{equation*}
	Since $\frac{8}{3}p<\frac{11}{3}=q$, we conclude $u\in\mathscr{X}^{s,p}_{T}$ for $s\in \left(\frac{3}{2},\frac{7}{4}\right]$. For $s\in \left(\frac{7}{4},2\right)$, taking $(q,r)=\left(\frac{16}{3},\frac{32}{13}\right)$ provides the embedding
	$$W^{\frac{7}{4},r}(\mathbb{R}_{+}^{n})\hookrightarrow L^{\infty}(\mathbb{R}_{+}^{n})\cap W^{2-s,\frac{n}{2-s}}(\mathbb{R}_{+}^{n})$$
	and analogous arguments yield $u\in\mathscr{X}^{s,p}_{T}$.
\end{proof}

Finally, we give the regularity result for \(s \in \left[2, \frac{5}{2}\right)\).
\begin{proposition}\label{zealous2}
	Let $u_{0}\in H^{s}(\mathbb{R}^{n}_{+})$ and $h_{0}\in \mathcal{H}^{s}(\mathbb{R}_{+}^{n})$ satisfy the compatibility condition.  For any $T \in (0,\infty]$,  the following regularity properties hold:
	\begin{enumerate}[label=(\roman*)]
		\item For $s=2$ with $0 < p < \frac{4}{n-2}$, $u \in \mathscr{X}^{1,p}_T\implies u \in \mathscr{X}^{2,p}_T$.
		\item For $s \in \left(2,\frac{5}{2}\right)$ with $2 \leq p < \frac{4}{n-2}$, $u \in \mathscr{X}^{2,p}_T \implies u \in \mathscr{X}^{s,p}_T$.
	\end{enumerate}
\end{proposition}
\begin{proof}
	For $s=2$, by Theorem \ref{chaste} with Corollary \ref{core}--\ref{core1} and Proposition \ref{weak2}, there exist $T_{\max}\in(0,\infty]$ and a solution belonging to $\mathscr{X}^{2,p}_{T_{\max}}$. As in Proposition \ref{zealous1}, arguing by contradiction, we assume  $0<T_{\max}<T$ and  bound $\|u(t)\|_{H^{2}(\mathbb{R}^{n}_{+})}$ on $[0,T_{\max})$. The  estimate  \eqref{begrudge} gives for all $t \in [0,T_{\max})$:
	\begin{equation}\label{panorama} 		
		\|u\|_{C([0,t];H^{2})}+\|u\|_{W^{1,q}(0,t;L^{r})}
		\lesssim\|u_{0}\|_{H^{2}}+\|h_{0}\|_{\mathcal{H}^{2}}+ \|F(u)\|_{C([0,t];L^{2})}+\|F(u)\|_{W^{1,\gamma'}(0,t;L^{\rho'})}.
	\end{equation}
	For $0< p < \frac{4}{n-2}$, we now determine the exponents appearing in \eqref{panorama}. We choose  $2<\sigma,r,\rho<\frac{2n}{n-2}$, such that the relation
	$\frac{1}{\rho'}= \frac{p}{\sigma}+\frac{1}{r}$ holds. Indeed, consider the function $f_{p}:\mathbb{R}^{3}\rightarrow \mathbb{R}$, $f_{p}(x,y,z)=x+py+z-1$, which satisfies  $f_{p}(\frac{1}{2},\frac{1}{2},\frac{1}{2})>0$ and $f_{p}(\frac{n-2}{2n},\frac{n-2}{2n},\frac{n-2}{2n})<0$. The exponents $q$ and gamma are chosen so that $(q,r)$ and $(\gamma,\rho)$ are admissible pairs; clearly we have  $q>\gamma'$.
	
	From Sobolev embeddings and Gagliardo-Nirenberg inequality, we derive
	\begin{equation}\label{meticulous}
		\|F(u)(t)\|_{L^{2}}\lesssim \|F(u)(t)\|_{W^{1,\rho'}}\lesssim \|u(t)\|_{L^{\sigma}}^{p} \|u(t)\|_{W^{1,r}}\lesssim \|u(t)\|_{H^{1}}^{p+1-\theta} \|u(t)\|_{H^{2}}^{\theta},
	\end{equation}
	where $\theta\in(0,1)$. Substituting \eqref{meticulous} into \eqref{panorama}, and applying  Young's inequality yields
	\begin{equation}\label{transient} \|u\|_{C([0,t];H^{2})}+\|u\|_{W^{1,q}(0,t;L^{r})}\lesssim \|u_{0}\|_{H^{2}}+\|h_{0}\|_{\mathcal{H}^{2}}+ \|u\|_{C([0,t];H^{1})}^{\frac{p+1-\theta}{1-\theta}}+\|F(u)\|_{W^{1,\gamma'}(0,t;L^{\rho'})}.
	\end{equation}
	
	We also have the following inequalities
	\begin{equation}
		\|F(u)(\tau)\|_{L^{\rho'}}\lesssim \|u(\tau)\|_{H^{1}}^{p} \|u(\tau)\|_{L^{r}},
	\end{equation}
	\begin{equation}
		\|F(u)(\tau+h)-F(u)(\tau)\|_{L^{\rho'}}\lesssim \left(\|u(\tau+h)\|_{H^{1}}^{p}+\|u(\tau)\|_{H^{1}}^{p}\right) \|u(\tau)\|_{L^{r}},
	\end{equation}		
	where $\tau$, $\tau+h\in(0,t)$. Then from Lemma \ref{relay}, we deduce that
	\begin{equation}\label{iconography}
		\|F(u)\|_{W^{1,\gamma}(0,t;L^{\rho'})}\lesssim \|u\|_{C([0,t];H^{1})}^{p} \|u\|_{W^{1,\gamma'}(0,t;L^{r})}.
	\end{equation}
	Substituting \eqref{iconography} into \eqref{transient}, we make use of $q>\gamma'$ to apply Lemma \ref{holmer}. Combining this with assumptions $u\in\mathscr{X}^{1,p}_{T}$ and $T_{\max}<T$,  we derive $\|u\|_{W^{1,q}(0,t;L^{r})}\leq M(t)$, where $M\in C_{b}[0,T_{\max})$. Using $q>\gamma'$ once more, we see that $\|u\|_{W^{1,\gamma'}(0,t;L^{r})}$ is controlled. Together with \eqref{transient}, we finally obtain the  bound for $\|u\|_{C([0,t];H^{2})}$.
	
	For $s\in \left(2,\frac{5}{2}\right)$ by using $u\in L^{q}\left(0,T_{\max};W^{2,r}\right)$ where $(q,r)=\left(5p,\frac{10np}{5np-4}\right)$ to ensure the embedding
	\[W^{2,r}(\mathbb{R}_{+}^{n})\hookrightarrow L^{\infty}(\mathbb{R}_{+}^{n})\cap W^{s-2,\frac{n}{s-2}}(\mathbb{R}_{+}^{n}).\]
	From \eqref{ratify} and an analysis analogous to that in \eqref{lawsuit}--\eqref{defy}, we derive the desired result.
\end{proof}
\subsection{Proof of Theorem \ref{global}}
In this section, we mainly use the $H^{1}$ a priori estimate in Section \ref{sec41} and the regularity analysis in Section \ref{sec42} to prove Theorem \ref{global}.

Before proving Theorem \ref{global}, we first establish the existence of the solution $u \in \mathscr{X}^{2,p}_\infty$ for the case $s=2$ under the broader parameter range \eqref{bigp} in Proposition \ref{ag}. This preliminary result serves two key purposes:
(i) To handle the case $s=1$ through an approximation argument;
(ii) To directly construct global solutions for $s=2$ as stated in Theorem \ref{global}, within the framework of Definition~\ref{xspt}.
\begin{proposition}\label{ag}
	Let $u_{0}\in H^{2}(\mathbb{R}_{+}^{n})$ and  $h_{0}\in\mathcal{H}^{2}(\mathbb{R}_{+}^{n})$
	satisfy the compatibility condition. Under the following parameter constraints:
	\begin{equation}\label{bigp}
		\text{when $\lambda<0$, $p<\tfrac{4}{n-2}$}\;\text{or when $\lambda>0$,  $p\leq \tfrac{4}{n+1}$},
	\end{equation}
	there exists a unique solution $u\in\mathscr{X}^{2,p}_{\infty}$.
\end{proposition}
\begin{proof}
	Since $u_{0} \in H^{2}(\mathbb{R}_{+}^{n}) \subset H^{1}(\mathbb{R}_{+}^{n})$ and $h_{0} \in \mathcal{H}^{2}(\mathbb{R}_{+}^{n}) \subset \mathcal{H}^{1}(\mathbb{R}_{+}^{n})$, Theorem \ref{chaste} guarantees the existence of a maximal existence time $T_{\max} \in (0,\infty]$ and a solution $u\in~\mathscr{X}^{1,p}_{T_{\max}}$. By the blow-up criterion,  if $T_{\max} < \infty$, then
	\begin{equation}\label{contradict}
		\lim_{t \to T_{\max}} \|u(t)\|_{H^{1}(\mathbb{R}^n_+)} = \infty.
	\end{equation}
	From Proposition \ref{zealous2}, we conclude that $u \in \mathscr{X}^{2,p}_{T_{\max}}$, which implies $u$ satisfies the regularity condition \eqref{reg}. Building upon the definitions of $\mathscr{X}^{2,p}_{T_{\max}}$ (see Definitions \ref{xspt}--\ref{xspt1}) and $\mathcal{W}^{2}$~(see \eqref{huaw}), together with \eqref{facile}, we obtain \[\partial_{t}u\in C\left(\,\overline{\mathbb{R}_{+}\mkern-4mu}\,;L_{\mathrm{loc}}^{2}(\mathbb{R}^{n-1}\times (0, T_{\max}) )\right).\]
	The same result also holds for $\partial_{x_{i}}u$ ($1\leq i\leq n-1$) by using the definition of $\mathscr{X}^{1,p}_{T_{\max}}$. Finally, Definition~\ref{xspt}(3) with \eqref{facile} implies
	\[u\in C^{1}\left(\,\overline{\mathbb{R}_{+}\mkern-4mu}\,;L^{2}_{\mathrm{loc}}(\mathbb{R}^{n-1}\times (0,T_{\max}))\right).\]
	Consequently, $u$ satisfies the regularity assumptions required by Proposition \ref{maritime}.
	
	Using the a priori estimate \eqref{drone} and the embedding \eqref{pleat}, we derive that $\|u(t)\|_{H^{1}}$ is bounded on $[0,T_{\max})$.  This contradicts \eqref{contradict}; therefore, we conclude that $T_{\max}=\infty$ and $u\in \mathscr{X}^{1,p}_{\infty}$. Applying Proposition \ref{zealous2} directly gives $u\in \mathscr{X}^{2,p}_{\infty}$.
\end{proof}
Combining results from the above we can now prove Theorem \ref{global}.
\begin{proof}[Proof of Theorem \ref{global}]
	For $s=1$, given $u_0 \in H^1(\mathbb{R}_+^n)$ and $h_0 \in \mathcal{H}^1(\mathbb{R}^{n}_{+}) \cap H^1(\mathbb{R}^{n}_{+})$, Theorem \ref{chaste} and Corollary~\ref{core} establish the existence of a maximal solution $u \in \mathscr{X}^{1,p}_{T_{\max}}$.
	Assume $T_{\max}<~\infty$.
	We construct approximating sequences:
	\begin{enumerate}[label=(\roman*)]
		\item $\{u_{0m}\} \subset H^2(\mathbb{R}_+^n)$ with $\lim\limits_{m\to\infty} u_{0m} = u_0$ in $H^1(\mathbb{R}_+^n)$;
		\item $\{h_{0m}\} \subset \mathcal{H}^2(\mathbb{R}^{n}_{+})$ with $\lim\limits_{m\to\infty} h_{0m} = h_0$ in $H^1(\mathbb{R}^{n}_{+})$,
	\end{enumerate}
	where $u_{0m}$ and $h_{0m}$ satisfy the compatibility condition. These sequences are justified as follows: We first choose $u_{0m}$ satisfying (i). Since $u_{0}-h_{0}\in H_{0}^{1}(\mathbb{R}_{+}^{n})$, we can select $g_{0m}\in\mathcal{H}_{0}^{2}(\mathbb{R}_{+}^{n})$ such that $g_{0m}\rightarrow u_{0}-h_{0}$ in $H^{1}(\mathbb{R}_{+}^{n})$, and finally define $h_{0m}=u_{0m}-g_{0m}$. For these initial and boundary data, it follows from Proposition \ref{ag} that there exists
	$u_{m}\in \mathscr{X}^{1,p}_{\infty}\cap\mathscr{X}^{2,p}_{\infty}$. Applying Proposition~\ref{maritime}, we obtain
	\begin{equation}\label{cram}
		\|u_{m}\|_{L^{\infty}(0,T_{\max};H^{1})}\lesssim  \alpha_{T_{\max}} \left(\|u_{0m}\|_{H^{1}}+\|h_{0m}\|_{H^{1}(\mathbb{R}^{n-1}\times(0,T_{\max}))} \right).
	\end{equation}
	Thus, there exists a subsequence $\{u_{m_{j}}\}$ such that $u_{m_{j}}\rightarrow \tilde{u}$ weakly star in $L^{\infty}(0,T_{\max};H^{1})$.
	
	Since $0<p<\frac{4}{n-2}$, as in the proof of Proposition \ref{zealous2}, we choose $2<\sigma,r,\rho<\frac{2n}{n-2}$, such that $\frac{1}{\rho'}= \frac{p}{\sigma}+\frac{1}{r}$, and take $q$ and $\gamma$ such that $(q,r)$ and $(\gamma,\rho)$ are admissible pairs.
	
	From \eqref{intimate},
	for any $T\in(0,T_{\max})$ and $t\in[0,T]$, we have
	\begin{equation*}
		\|u_{m}-u\|_{L^{q}(0,t;L^{r})}\lesssim \|u_{0m}-u_{0}\|_{L^{2}}+\|h_{0m}-h_{0}\|_{\mathcal{H}^{0}}+\|F(u_{m})-F(u)\|_{L^{\gamma'}(0,t;L^{\rho'})}.
	\end{equation*}
	Using \eqref{uproar}, H\"{o}lder's inequality,  and the embedding $H^{1}(\mathbb{R}^{n}_{+})\hookrightarrow L^{\sigma}(\mathbb{R}^{n}_{+})$, we obtain
	\begin{equation*}
		\|F(u_{m})-F(u)\|_{L^{\gamma'}(0,t;L^{\rho'})}\lesssim	\left(\|u_{m}\|_{L^{\infty}(0,t;H^{1})}^{p}+\|u\|_{L^{\infty}(0,t;H^{1})}^{p}\right) \|u_{m}-u\|_{L^{\gamma'}(0,t;L^{r})}.
	\end{equation*}
	Combining these estimates yields
	\begin{equation}\label{gj}
		\|u_{m}-u\|_{L^{q}(0,t;L^{r})}\leq c_{T}\left( \|u_{0m}-u_{0}\|_{L^{2}}+\|h_{0m}-h_{0}\|_{\mathcal{H}^{0}}\right)+ c_{T}\|u_{m}-u\|_{L^{\gamma'}(0,t;L^{r})}.
	\end{equation}
	By applying Lemma \ref{holmer} to \eqref{gj}, we derive
	\begin{equation}\label{gui}
		\|u_{m}-u\|_{L^{q}(0,T;L^{r})}\lesssim C_{T}  \left(\|u_{0m}-u_{0}\|_{L^{2}}+\|h_{0m}-h_{0}\|_{\mathcal{H}^{0}}\right).
	\end{equation}
	The embedding \eqref{emb} implies $H^{1}(\mathbb{R}^{n}_{+}) \hookrightarrow \mathcal{H}^{0}(\mathbb{R}^{n}_{+})$. From \eqref{gui}, it follows that $u_{m}\rightarrow u$ in $L^{q}(0,T;L^{r})$, hence $u = \tilde{u}$ almost everywhere.
	
	By the properties of weak-star convergence and the continuity of $\alpha_{T_{\max}}$, the estimate \eqref{cram} yields
	\begin{equation*}
		\|u\|_{L^{\infty}(0,T_{\max};H^{1})}
		\leq \varliminf_{m_{j}\rightarrow\infty}\|u_{m_{j}}\|_{L^{\infty}(0,T_{\max};H^{1})}
		\lesssim \alpha_{T_{\max}}\left(\|u_{0}\|_{H^{1}}+\|h_{0}\|_{H^{1}(\mathbb{R}^{n-1}\times(0,T_{\max}))}\right),
	\end{equation*}
	which contradicts the blow-up criterion, proving $T_{\max}=\infty$.
	Consequently, we conclude that the solution $u \in \mathscr{X}^{1,p}_\infty$, which completes the proof of Theorem~\ref{global} for the case $s=1$. An application of Proposition \ref{zealous1} yields the result for $s\in(1,2)$.
	
	For the case $s = 2$, the global solution stated in Theorem~\ref{global} can be obtained by Proposition~\ref{ag}. For higher regularity $s\in\left(2,\frac{5}{2}\right)$, the result follows immediately from Proposition~\ref{zealous2}.
\end{proof}

\subsection{$L^{2}$ a priori estimates}
This section focuses on the establishment of the  $L^{2}$ estimates.
\begin{proposition}\label{vindictive}
	Suppose that $p\leq \frac{2}{n}$ and $(p,n)\neq (1,2)$. Let $u_{0}\in H^{2}(\mathbb{R}^{n}_{+})$ and $h_{0}\in\mathcal{H}^{2}(\mathbb{R}^{n}_{+})$ satisfy the compatibility condition. For any $T>0$, every solution $u$ to IBVP \eqref{ibvp} possessing the regularity
	\begin{align}
		t&\mapsto u(\cdot,t) \in C\left([0,T);H^{2}(\mathbb{R}^{n}_{+})\right)\cap C^{1}\left([0,T);L^{2}(\mathbb{R}^{n}_{+})\right),\label{assuage} \\
		x_{n}&\mapsto u(\cdot,x_{n},\cdot) \in  C^{1}\left(\,\overline{\mathbb{R}_{+}\mkern-4mu}\,;L^{2}_{\mathrm{loc}}(\mathbb{R}^{n-1}\times(0,T))\right)\label{enigmatic},
	\end{align}
	satisfies the  estimate
	\begin{equation}
		\|u(t)\|_{L^{2}}\lesssim  (1+t) A\left(\|u_{0}\|_{L^{2}}+\|h_{0}\|_{\mathcal{H}^{0}}\right) e^{tA\left(\|u_{0}\|_{L^{2}}+\|h_{0}\|_{\mathcal{H}^{0}}\right)},
	\end{equation}
	for all $t\in[0,T)$. Here, $A:\overline{\mathbb{R}_{+}}\rightarrow \overline{\mathbb{R}_{+}}$ is a continuous,  non-decreasing function with $A(0)=0$.	
\end{proposition}
\begin{proof}
	We define $v$ as the solution to the following linear IBVP for the unforced Schr\"{o}dinger equation
	\begin{align}\label{coagulate}
		\left\{\begin{aligned}
			&	i\partial_{t}v+\Delta v=0, \quad (x, t) \in \mathbb{R}_{+}^{n} \times \mathbb{R}_{+}, \\
			&v(x,0)=u_0(x), \quad x\in\mathbb{R}^{n}_{+},\\& v(x',0,t)=h_{0}(x',t), \quad  (x',t)\in \mathbb{R}^{n-1}\times \mathbb{R}_{+}.
		\end{aligned}
		\right.				
	\end{align}
	Then the difference $w:=u-v$ satisfies
	\begin{equation}\label{bolster}
		\left\{\begin{aligned}
			&i\partial_{t}w+\Delta w + \lambda |v+w|^{p} (v+w)=0, \quad (x,t) \in \mathbb{R}^{n}_{+} \times \mathbb{R}_{+},\\
			&w(x,0)=0, \quad x\in\mathbb{R}^{n}_{+},\\& w(x',0,t)=0, \quad  (x',t)\in \mathbb{R}^{n-1}\times \mathbb{R}_{+}.
		\end{aligned}
		\right.
	\end{equation}
	Multilplying $\bar{w}$ to both sides of the equation in \eqref{bolster}, and then taking the imaginary part, we obtain
	\begin{equation}\label{ineffable}
		\partial_{t} |w|^{2} + 2\text{Im} [\Delta w \bar{w}]+ 2\lambda\text{Im}[|v+w|^{p} v \bar{w}]=0, \quad (x,t) \in \mathbb{R}^{n}_{+} \times \mathbb{R}_{+}.
	\end{equation}
	Integrating both sides of \eqref{ineffable} over $\mathbb{R}^{n}_{+}\times (0,t)$, and utilizing integration by parts we derive
	\begin{equation}\label{pulchritudinous}
		\begin{aligned}
			\|w(t)\|_{L^{2}}^{2}=&\,\|w(0)\|_{L^{2}}^{2}+ 2\text{Im}\int_{0}^{t} \int_{\mathbb{R}^{n-1}} (w\partial_{n} w)(x',0,\tau) dx'd\tau \\&+2\text{Im}\int_{0}^{t}\int_{\mathbb{R}^{n}_{+}}  |\nabla  w(x,\tau)|^{2}dxd\tau+ 2\lambda  \text{Im}    \int_{0}^{t} \int_{\mathbb{R}^{n}_{+}} (|v+w|^{p}v\bar{w}) (x,\tau) dxd\tau\\=&\,2\lambda  \text{Im}    \int_{0}^{t} \int_{\mathbb{R}^{n}_{+}} (|v+w|^{p}v\bar{w})(x,\tau) dxd\tau\\\lesssim&\int_{0}^{t}\int_{\mathbb{R}^{n}_{+}} |w|^{p+1} |v| (x,\tau)+ |v|^{p+1} |w| (x,\tau)dx d\tau,
		\end{aligned}
	\end{equation}
	an application of  H\"{o}lder's inequality leads to
	\begin{equation}\label{inextricable}
		\begin{aligned}
			\|w(t)\|_{L^{2}}^{2}\lesssim &\, \|v\|_{L^{\frac{4}{np}}(0,t;L^{\frac{2}{1-p}})} \left(\int_{0}^{t}\|w(\tau)\|_{L^{2}}^{\frac{4(p+1)}{4-np}}d\tau\right)^{\frac{4-np}{4}} \\&+ \|v\|_{L^{\frac{4(p+1)}{np}}(0,t;    L^{2(p+1)})} \left(\int_{0}^{t}\|w(\tau)\|_{L^{2}}^{\frac{4}{4-np}}d\tau\right)^{\frac{4-np}{4}}.
		\end{aligned}
	\end{equation}
	Given the range of $p$, both $\frac{2}{1-p}$ and $2(p+1)$ lie in  $[2,\infty)$ when $n=2$, and in  $\left[2,\frac{2n}{n-2}\right]$, when $n>2$. This guarantees that both
	$\big(\frac{4}{np},\frac{2}{1-p}\big)$ and $\big(\frac{4(p+1)}{np},2(p+1)\big)$ are admissible pairs.	Hence, applying the Strichartz estimate \eqref{intimate} with $f=0$ yields
	\begin{equation}
		\max \Bigl\{ 	\|v\|_{L^{\frac{4(p+1)}{np}}(0,t  ;  L^{2(p+1)})},	\|v\|_{L^{\frac{4}{np}}(0,t;L^{\frac{2}{1-p}})}\Bigr\}\lesssim \|u_{0}\|_{L^{2}}+\|h_{0}\|_{\mathcal{H}^{0}}.
	\end{equation}
	Using the fact that $\frac{4(p+1)}{4-np}\leq \frac{8}{4-np}$, we combine estimate \eqref{inextricable} with Young's inequality, to obtain
	\begin{equation}\label{venerate}
		\|w(t)\|_{L^{2}}^{\frac{8}{4-np}}\leq C(\|u_{0}\|_{L^{2}}+\|h_{0}\|_{\mathcal{H}^{0}}) \left(t+\int_{0}^{t}\|w(\tau)\|_{L^{2}}^{\frac{8}{4-np}}d\tau\right),
	\end{equation}
	where $C(x)=ax^{b}$ with $a,b>0$. Then Gronwall's inequality applied to \eqref{venerate} gives
	\begin{equation}
		\|w(t)\|_{L^{2}}^\frac{8}{4-np}\leq tC(\|u_{0}\|_{L^{2}}+\|h_{0}\|_{\mathcal{H}^{0}})  e^{tC(\|u_{0}\|_{L^{2}}+\|h_{0}\|_{\mathcal{H}^{0}})}.
	\end{equation}
	Recall that $u=v+w$.  Together with the linear estimate $\|v\|_{L^{\infty}(0,t;L^{2})}\lesssim \|u_{0}\|_{L^{2}}+\|h_{0}\|_{\mathcal{H}^{0}}$ (which also provided by \eqref{intimate}), this yields the final estimate for $u$.
\end{proof}
The method developed above is completely adapt to the case $n=1$, providing a new estimate for the one-dimensional IBVP \eqref{ibvpde}.
\begin{corollary}\label{magnanimous}
	 Let $p\leq1$, $u_{0}\in H^{2}(\mathbb{R}_{+})$ and $h_{0}\in H^{\frac{5}{4}}(\mathbb{R}_{+})$ satisfy the compatibility condition. For any $T>0$, every solution $u$ to IBVP \eqref{ibvpde} possessing the regularity
	\begin{align}
		t&\mapsto u(\cdot,t) \in C\left([0,T);H^{2}(\mathbb{R}_{+})\right)\cap C^{1}\left([0,T);L^{2}(\mathbb{R}_{+})\right), \\
		x&\mapsto u(x,\cdot) \in  C^{1}\big(\,\overline{\mathbb{R}_{+}\mkern-4mu}\,;H_{\mathrm{loc}}^{\frac{1}{4}}(0,T)\big),		\end{align}
   satisfies the  estimate
	\begin{equation}
		\|u(t)\|_{L^{2}}\lesssim  (1+t) A (\|u_{0}\|_{L^{2}}+\|h_{0}\|_{H^{\frac{1}{4}}})\; e^{tA\big(\left\|u_{0}\right\|_{L^{2}}+\left\|h_{0}\right\|_{H^{\frac{1}{4}}}\big)},
	\end{equation}
	for all $t\in[0,T)$. Here, $A:\overline{\mathbb{R}_{+}}\rightarrow \overline{\mathbb{R}_{+}}$ is a continuous,  non-decreasing function with $A(0)=0$.	
\end{corollary}
\begin{proof}
	Let $v$ be the solution to the linear problem:
	\begin{align}
		\left\{\begin{aligned}
			&	i\partial_t v+\partial_{x}^{2} v=0, \quad (x,t) \in \mathbb{R}_{+} \times \mathbb{R}_{+},\\
			&v(x,0)=u_0(x),\quad  v(0,t)=h_0(t).
		\end{aligned}
		\right.		
	\end{align}
	For $v$, the Strichartz estimate (see \cite{chatter,figment}) states that
	\begin{equation}
		\|v\|_{L^{q}(0,t;L^{r}(\mathbb{R}_{+}))}\lesssim \|u_{0}\|_{L^{2}(\mathbb{R}_{+})} + \|h_{0}\|_{H^{\frac{1}{4}}(\mathbb{R}_{+})}.
	\end{equation}
	In the range $p\leq1$,	both the pairs  $\big(\frac{4}{p},\frac{2}{1-p}\big)$ and $\big(\frac{4(p+1)}{p},2(p+1)\big)$ are admissible. The estimate for $w=u-v$ follows by setting  $n=1$ in the proof of Proposition \ref{vindictive}, this completes the proof.
\end{proof}
Proposition \ref{vindictive} will be applied to the case $s=0$, $n\geq2$ in Theorem \ref{new L2}. We now establish the following proposition to prepare for proving the case $s\in(0,1)$, $n\geq2$ in Theorem \ref{new L2}.
\begin{proposition}\label{pragmatic}
	Suppose $s\in(0,1)$ and $p\leq \min\{1,  \frac{2(s+1)}{n}\}$. Let $u_{0}\in H^{2}(\mathbb{R}^{n}_{+})$ and $h_{0}\in\mathcal{H}^{2}(\mathbb{R}^{n}_{+})$ satisfy the compatibility condition. Then for any solution $u$ to IBVP \eqref{ibvp} that satisfies \eqref{assuage}--\eqref{enigmatic},  the  estimate
	\begin{equation}
		\|u(t)\|_{L^{2}}\lesssim  (1+t) A\left(\|u_{0}\|_{H^{s}}+\|h_{0}\|_{\mathcal{H}^{s}}\right) e^{tA\left(\|u_{0}\|_{H^{s}}+\|h_{0}\|_{\mathcal{H}^{s}}\right)}
	\end{equation}
	holds for all $t\in[0,T)$. Here, $A:\overline{\mathbb{R}_{+}}\rightarrow \overline{\mathbb{R}_{+}}$ is a continuous, non-decreasing function with $A(0)=0$.	
\end{proposition}
\begin{proof}
	We define $v$ as the solution of \eqref{coagulate} and $w:=u-v$.  Since the estimate for $v$ is known, it suffices to bound $w$. Applying H\"{o}lder's inequality with respect to $x$ in \eqref{pulchritudinous}, we obtain
	\begin{equation}\label{esoteric}
		\begin{aligned}
			\|w(t)\|_{L^{2}}^{2}& \lesssim \int_{0}^{t} \|w(\tau)\|_{L^{2}}^{p+1}\|v(\tau)\|_{L^{\frac{2}{1-p}}}+ \|v(\tau)\|_{L^{2(p+1)}}^{p+1}\|w(\tau)\|_{L^{2}} d\tau.
		\end{aligned}
	\end{equation}
	We first treat the case $n=2$, for which only  $p=1$ remains, since $p\in(0,1)$ is already proved in Proposition \ref{vindictive}. For any $s>0$, choose $r\in(2,\infty)$ such that $W^{s,r}(\mathbb{R}^{2}_{+})\hookrightarrow L^{\infty}(\mathbb{R}^{2}_{+})$, and let $q\in(2,\infty)$ be such that $(q,r)$ is an admissible pair. Combining this with \eqref{esoteric} leads to
	\begin{equation}
		\begin{aligned}
			\|w\|_{L^{\infty}(0,t;L^{2})}^{2}\lesssim\|v\|_{L^{q}(0,t;W^{s,r})}\|w\|_{L^{2q'}(0,t;L^{2})}^{2}+ t^{\frac{1}{2}}\|v\|_{L^{4}(0,t;L^{4})}^{2}\|w\|_{L^{\infty}(0,t;L^{2})}.
		\end{aligned}
	\end{equation}
	Note that $(4,4)$ is an admissible pair for $n=2$. Applying Young's inequality, we have
	\begin{equation}
		\begin{aligned}	\|w\|_{L^{\infty}(0,t;L^{2})}^{2}&\lesssim\|v\|_{L^{q}(0,t;W^{s,r})}\|w\|_{L^{2q'}(0,t;L^{2})}^{2}+ t\|v\|_{L^{4}(0,t;L^{4})}^{4}\\&\lesssim C(\|u_{0}\|_{L^{2}}+\|h_{0}\|_{\mathcal{H}^{s}}) \left(1+t+\|w\|_{L^{2q'}(0,t;L^{2})}^{2}\right),
		\end{aligned}
	\end{equation}
	where $C(x)=a_{1}x^{b_{1}}+a_{2}x^{b_{2}}$ with $a_{i},b_{i}>0$ for $i=1,2$. From Gronwall's inequality, we bound $w$ as follows:
	\begin{equation}
		\|w\|_{L^{\infty}(0,t;L^{2})}\lesssim (1+t)^{\frac{1}{2}} C(\|u_{0}\|_{H^{s}}+\|h_{0}\|_{\mathcal{H}^{s}})  e^{tC(\|u_{0}\|_{H^{s}}+\|h_{0}\|_{\mathcal{H}^{s}})}.
	\end{equation}
	
	For $n>2$, the estimate in the range $p\leq \frac{2}{n}$ has been proved in Proposition \ref{vindictive}; we therefore consider the remaining range $\frac{2}{n}<p\leq\frac{2(s+1)}{n}$. A direct  computation gives
	\begin{equation}
		\frac{n-2}{2n}-\frac{s}{n}\leq \frac{1-p}{2}< \frac{n-2}{2n},
	\end{equation}
	which implies $W^{s,\frac{2n}{n-2}}(\mathbb{R}^{n}_{+})\hookrightarrow L^{\frac{2}{1-p}}(\mathbb{R}^{n}_{+})$. On the other hand,  from computation we can show for any $p$, there exists $r\in [2,\frac{2}{n-2}]$, such that $W^{s,r}(\mathbb{R}^{n}_{+})\hookrightarrow L^{2(p+1)}(\mathbb{R}^{n}_{+})$. Indeed:
	\begin{itemize}
		\item  when $\frac{2}{n}<p\leq \frac{2}{n-2}$, we have $2<2(p+1)\leq \frac{2n}{n-2}$,
		\item  when $\frac{2}{n-2}\leq p\leq \frac{2(s+1)}{n}$, we have $W^{s,\frac{2n}{n-2}}(\mathbb{R}^{n}_{+})\hookrightarrow L^{2(p+1)}(\mathbb{R}^{n}_{+})$.
	\end{itemize}
	Hence, we can choose $q\in[2,\infty]$ such that $(q,r)$ is an admissible pair. Applying these two embeddings  to \eqref{esoteric}, and then using the estimate \eqref{confine} for $v$, we obtain
	\begin{equation}
		\begin{aligned}
			\|w\|_{L^{\infty}(0,t;L^{2})}^{2}&\lesssim \|v\|_{L^{2}(0,t;W^{s,\frac{2n}{n-2}})} \|w\|_{L^{2(p+1)}(0,t;L^{2})}^{p+1}+t^{1-\frac{p+1}{q}}\|v\|_{L^{q}(0,t;W^{s,r})}^{p+1} \|w\|_{L^{\infty}(0,t;L^{2})}\\&\lesssim C(\|u_{0}\|_{H^{s}}+\|h_{0}\|_{\mathcal{H}^{s}})\left(\|w\|_{L^{2(p+1)}(0,t;L^{2})}^{p+1}+(1+t) \|w\|_{L^{\infty}(0,t;L^{2})}\right).
		\end{aligned}
	\end{equation}
	A use of Young's inequality yields
	\begin{equation}
		\|w\|_{L^{\infty}(0,t;L^{2})}\leq C(\|u_{0}\|_{H^{s}}+\|h_{0}\|_{\mathcal{H}^{s}}) (1+t+\|w\|_{L^{2(p+1)}(0,t;L^{2})}).
	\end{equation}
	Applying  Gronwall's inequality to this estimate leads to
	\begin{equation}
		\|w\|_{L^{\infty}(0,t;L^{2})}\lesssim (1+t) C(\|u_{0}\|_{H^{s}}+\|h_{0}\|_{\mathcal{H}^{s}})  e^{tC(\|u_{0}\|_{H^{s}}+\|h_{0}\|_{\mathcal{H}^{s}})}.
	\end{equation}
\end{proof}
\subsection{Regularity analysis for $s\in[0,1)$}\label{painstakingly}
This section provides the regularity analysis for  $s\in[0,1)$. We still rely on the space $\mathscr{X}^{s,p}_T$ given by Definitions \ref{xspt} and \ref{xspt1}, and benefit from techniques similar to those in Proposition~\ref{zealous1}.
\begin{proposition}\label{laudatory}
	Suppose $s\in(0,1)$ and $p\leq 1$. Further assume that $p\leq \frac{6-4s}{3(n-2)}$ when $n\geq6$, and  $p<\frac{2}{n-2}$ when $n<6$.  Let $u_{0}\in H^{s}(\mathbb{R}^{n}_{+})$ and $h_{0}\in \mathcal{H}^{s}(\mathbb{R}^{n}_{+})$ satisfy the compatibility condition. For any $0<T \leq
	\infty$, if $u \in \mathscr{X}^{0,p}_T$, then  $u \in \mathscr{X}^{s,p}_T$.
\end{proposition}
\begin{proof}
	Following the contradiction scheme of Proposition \ref{zealous1}, we assume $0<T_{\max}<T$ and proceed to bound $\|u(t)\|_{H^{s}(\mathbb{R}^{n}_{+})}$ on $[0,T_{\max})$. For $n=2$, select $\beta\in \left(\frac{3n}{3n+6-4s},2\right)$ satisfying $\frac{1}{\beta}-\frac{1}{2}\leq \frac{p}{2}$, and take $\alpha\in\left(\frac{3}{3-s},2\right)$ such that $(\alpha,\beta)$ satisfy the condition in \eqref{buffoon}. Define $\sigma$ by $\frac{1}{\sigma}=\frac{1}{\beta}-\frac{1}{2}$. Applying \eqref{buffoon} together with \eqref{riesz} yields
	\begin{equation}
		\|u(t)\|_{H^{s}}^{\alpha}\lesssim \|u_{0}\|_{H^{s}}^{\alpha}+\|h_{0}\|_{\mathcal{H}^{s}}^{\alpha}+ \langle t\rangle^{\frac{\alpha}{3}}\int_{0}^{t} \|u(\tau)\|_{L^{\sigma p}}^{\alpha p} \|u(\tau)\|_{H^{s}}^{\alpha}  d\tau.
	\end{equation}
	Note that $\sigma p\in[2,\infty)$ and $\alpha p<2$, and $u\in\mathscr{X}^{0,p}_{T}$ implies that $$u\in L^{\alpha p}(0,T_{\max};L^{\sigma p}(\mathbb{R}^{2}_{+})).$$ The bound for  $\|u(t)\|_{H^{s}(\mathbb{R}^{2}_{+})}$ then follows from Gronwall's inequality.
	
	We now consider $n>2$ with  $p\leq\frac{6-4s}{3(n-2)}$. Define $\sigma=\frac{2n}{p(n-2)}$, then take  $\beta$ such that $\frac{1}{\beta}=\frac{1}{2}+\frac{1}{\sigma}$, and the range of $p$ implies  $\beta\in \left[\frac{3n}{3n+6-4s},2\right)$.  Then we can choose  $\alpha\in\left(\frac{3}{3-s},2\right]$ such that $(\alpha,\beta)$ satisfy the condition of the estimate \eqref{buffoon}. Combining \eqref{buffoon} with \eqref{riesz}, we have
	\begin{equation}
		\begin{aligned}
			\|u(t)\|_{H^{s}}^{\alpha}&\lesssim \|u_{0}\|_{H^{s}}+\|h_{0}\|_{\mathcal{H}^{s}}+ \int_{0}^{t} \|F(u)(\tau)\|_{W^{s,\beta}}  d\tau\\& \lesssim  \|u_{0}\|_{H^{s}}+\|h_{0}\|_{\mathcal{H}^{s}}+ \langle t\rangle^{\frac{\alpha}{3}}\int_{0}^{t} \|u(\tau)\|_{L^{p\sigma}}^{\alpha p} \|u(\tau)\|_{H^{s}}  d\tau.
		\end{aligned}
	\end{equation}
	Note that $p\sigma=\frac{2n}{n-2}$ and $u\in\mathscr{X}^{0,p}_{T}$ implies that $u\in L^{\alpha p}\big(0,T_{\max};L^{\frac{2n}{n-2}}\big)$. Thus we obtain the bound for $\|u(t)\|_{H^{s}(\mathbb{R}^{n}_{+})}$.    This means that we have completed the proof of the case $n\geq6$ with $p\leq \frac{6-4s}{3(n-2)}$.

Next, we prove the case $n<6$ with $p<\frac{2}{n-2}$. Observe that  $\frac{6-4s}{3(n-2)}<\frac{2}{n-2}$ for $s\in(0,1)$.  The proof above shows that the proposition holds for $n<6$ with $p<\frac{6-4s}{3(n-2)}$. Therefore, we only need to prove  the remaining range $n<6$ with $\frac{6-4s}{3(n-2)}<p <\frac{2}{n-2}$. The above argument shows that for any  $p<\frac{2}{n-2}$, there exists some $\varepsilon>0$ such that  $u\in\mathscr{X}^{0,p}_{T}$ implies $u\in \mathscr{X}^{s,p}_{T}$ for $s\in (0,\varepsilon]$. We now extend this result to the full range $s\in(0,1)$. Taking $(\alpha,\beta)=\left(2,\frac{6n}{3n+6-4s}\right)$ in estimate \eqref{buffoon} yields
	\begin{equation}
		\|u(t)\|_{H^{s}}^{2}\lesssim \|u_{0}\|_{H^{s}}^{2}+\|h_{0}\|_{\mathcal{H}^{s}}^{2}+ \langle t\rangle^{\frac{2s}{3}} \int_{0}^{t} \|u(\tau)\|_{L^{\frac{3np}{3-2s}}}^{2p} \|u(\tau)\|_{H^{s}}^{2}d\tau.
	\end{equation}
	Following the procedure used in  Proposition \ref{zealous1}, we define $s_{k}=r^{k}$ with $r\in(0,1)$. The proof reduces to verifying that for any $k\geq0$, if $u\in\mathscr{X}^{s_{k+1},p}_{T}$, then $u\in\mathscr{X}^{s,p}_{T}$ for $s\in(s_{k+1},s_{k}]$. Since $u\in\mathscr{X}^{s_{k+1},p}_{T}$ implies $u\in L^{2}(0,T_{\max};W^{s_{k+1},\frac{2n}{n-2}})$, it
	this is then converted to  deriving the embedding $W^{s_{k+1},\frac{2n}{n-2}}\hookrightarrow L^{\frac{3np}{2-2s}}$ for $s\in(s_{k+1},s_{k}]$, provided that
	\begin{equation}\label{laudable}
		\frac{n-2}{2n}-\frac{s_{k+1}}{n}< \frac{3-2s}{3np}\leq \frac{n-2}{2n}.
	\end{equation}
	The inequality  $\frac{3-2s}{3np}\leq \frac{n-2}{2n}$ follows from the assumption $p> \frac{6-4s}{3(n-2)}$. Thus, only the left-hand inequality needs to be satisfied, which requires
	\begin{equation}
		p< f_{r}(k)=\frac{2(3-2r^{k})}{3(n-2-2r^{k+1})}\quad\text{for}\quad k\geq0 .
	\end{equation}
	By computation, for $n=3,4$, we have $p<\frac{2}{n-2}\leq f_{\frac{3}{4}}(k)$.  For $n=5$, there exists $r\in(0,1)$ such that $p<f_{r}(0)\leq\frac{2}{n-2}$; and note that $f_{r}(k)\rightarrow \frac{2}{n-2}$ as $k\rightarrow\infty$, then $p<f_{r}(k)$ holds. Therefore, inequality \eqref{laudable} holds for all $k$, which completes the proof.
\end{proof}
\begin{corollary}\label{parsimony}
	Suppose $s\in(0,1)$, $p\leq 1$, $u_{0}\in H^{s}(\mathbb{R}_{+})$ and $h_{0}\in H^{\frac{2s+1}{4}}(\mathbb{R}_{+})$ satisfy the compatibility condition. For any $0<T \leq
	\infty$, if
	\begin{equation*}
		u\in  C([0,T);L^{2}(\mathbb{R}_{+}))\cap C\big(\overline{\mathbb{R}_{+}};H^{\frac{1}{4}}_{\mathrm{loc}}(0,T)\big)
	\end{equation*}
	satisfies \eqref{ibvpde} and $u\in\bigcap_{(q,r)\, \text{admissible}} L^{q}_{\mathrm{loc}}\left(\left[0,T\right);L^{r}(\mathbb{R}_{+})\right)$, then $u \in C([0,T);H^{s}(\mathbb{R}_{+}))$.
\end{corollary}
\begin{proof}
	Within the same framework as Proposition \ref{laudatory}, we proceed to obtain an $H^{s}$ bound. According to the linear estimate in \cite{figment}, we can obtain a unified linear estimate
	\begin{equation*}
		\|u\|_{C([0,t];H^{s})} \lesssim  \|u_{0}\|_{H^{s}}+ \|h_{0}\|_{H^{\frac{2s+1}{4}}}+ \langle t\rangle^{\frac{1}{2}} \|F(u)\|_{L^{2}(0,t;H^{s})}.
	\end{equation*}
	Applying it to \eqref{ibvpde} gives
	\begin{equation}\label{475}
		\begin{aligned}
			\|u(t)\|_{H^{s}}^{2}\lesssim  \|u_{0}\|_{H^{s}}^{2}+\|h_{0}\|_{H^{\frac{2s+1}{4} }}^{2}+ \langle t\rangle \|F(u)\|^2_{L^{2}(0,t;H^{s})}.
		\end{aligned}
	\end{equation}
	Using  $\|F(u)(t)\|_{H^{s}}\lesssim \|u(t)\|_{L^{\infty}}^{p} \|u(t)\|_{H^{s}}$, and the fact that $(q,r)=(4,\infty)$ is an admissible pair in the one-dimensional setting, together with \eqref{475}, we obtain
	\begin{equation}
		\|u(t)\|_{H^{s}}^{2}\lesssim \|u_{0}\|_{H^{s}}^{2}+\|h_{0}\|_{H^{\frac{2s+1}{4} }}^{2}+\langle t\rangle \int_{0}^{t}\|u(\tau)\|_{L^{\infty}}^{2p} \|u(\tau)\|_{H^{s}}^{2}d\tau,
	\end{equation}
	thereby establishing the desired $H^{s}$ bound.
\end{proof}
\begin{remark}
	Propositions \ref{pragmatic}--\ref{laudatory} and Corollary~\ref{parsimony} will be used to prove Theorem~\ref{new L2} in the next section. Indeed, they also hold for $s=1$, thus the results for $s\in (0,1)$ in Theorem \ref{new L2} hold for $s=1$. This yields two improvements for the global result in  $H^{1}(\mathbb{R}^{n}_{+})$. First, for $n=2,3$ and $p\leq1$, the additional $H^{1}$ condition on the boundary data can be removed in Theorem \ref{global}. For $n=1$ and $p\leq1$ the one-dimensional result in \cite{figment} is also improved, where the requirement on the boundary data can be weakened from $H^{1}(\mathbb{R}_{+})$ to $ H^{\frac{3}{4}}(\mathbb{R}_{+})$. Second, for the focusing Schr\"{o}dinger equation, the  global  $H^{1}$ solutions for the IVP are known to exist   for $p< \frac{4}{n}$  \cite{Ginibre}. In contrast, for the nonhomogeneous IBVP, Theorem~\ref{global} establishes global solutions for  $p\leq\frac{4}{n+1}$, due to the use of  Gagliardo-Nirenberg inequality to the nonhomogeneous boundary setting in the proof of Proposition~\ref{maritime}. It is unknown if the upper bound $\frac{4}{n+1}$ can be improved to  $\frac{4}{n}$; here, in the four-dimensional setting, we achieve $p\in(0,1)$, offering a possibility to narrow this gap.
\end{remark}
\subsection{Proof of Theorem \ref{new L2}}

\begin{proof}[Proof of Theorem \ref{new L2}]
We present a detailed proof only for the case  $n\geq2$. The proof for $n=1$ follows analogously by applying Corollaries~\ref{magnanimous} and \ref{parsimony}, and Theorem~5.3 in \cite{figment}.

For  initial data $u_0 \in L^2(\mathbb{R}_+^n)$ and   boundary data $h_0 \in \mathcal{H}^0(\mathbb{R}^{n}_{+})$, Theorem~\ref{chaste}   establishes the existence of a maximal solution $u\in\mathscr{X}^{0,p}_{T_{\max}}$. Assume $T_{\max}<\infty$. By a density argument, we choose the following approximating sequences:
	\begin{enumerate}[label=(\roman*)]
		\item $\{u_{0m}\} \subset H_{0}^2(\mathbb{R}_+^n)$ with $\lim\limits_{m\to\infty} u_{0m} = u_0$ in $L^{2}(\mathbb{R}_+^n)$;
		\item $\{h_{0m}\} \subset \mathcal{H}_{0}^2(\mathbb{R}^{n}_{+})$ with $\lim\limits_{m\to\infty} h_{0m} = h_0$ in $\mathcal{H}^0(\mathbb{R}^{n}_{+})$,
	\end{enumerate}
	then $u_{0m}$ and $h_{0m}$ satisfy the compatibility condition naturally. For $u_{0m}$ and $h_{0m}$, it follows from Proposition \ref{ag} that there exists $u_{m}\in \mathscr{X}^{2,p}_{\infty}$. Applying Proposition~\ref{vindictive} to $u_{m}$, we obtain the estimate
	\begin{equation}
		\|u_{m}(t)\|_{L^{2}}\leq (1+t) A\left(\|u_{0m}\|_{L^{2}}+\|h_{0m}\|_{\mathcal{H}^{0}}\right) e^{CtA\left(\|u_{0m}\|_{L^{2}}+\|h_{0m}\|_{\mathcal{H}^{0}}\right)}.
	\end{equation}
	From the continuous dependence, for any $t\in[0,T_{\max})$,  $$\|u_{m}(t)-u(t)\|_{L^{2}(\mathbb{R}^{n}_{+})}\rightarrow0.$$ It then follows that for $t\in[0,T_{\max})$
	\begin{equation}
		\|u(t)\|_{L^{2}}\leq (1+t) A\left(\|u_{0}\|_{L^{2}}+\|h_{0}\|_{\mathcal{H}^{0}}\right) e^{CtA\left(\|u_{0}\|_{L^{2}}+\|h_{0}\|_{\mathcal{H}^{0}}\right)}.
	\end{equation}
	This bound contradicts the assumption $T_{\max}<\infty$. Hence we complete the proof for the case $n\geq2$ with $s=0$.
	
	For $s\in(0,1)$, $u_0 \in H^{s}(\mathbb{R}_+^n)$ and $h_0 \in \mathcal{H}^s(\mathbb{R}^{n}_{+})$, it follows from Theorem~\ref{chaste}   that  there exists a maximal solution $u\in\mathscr{X}^{0,p}_{T_{\max}}$. We take
	\begin{enumerate}[label=(\roman*)]
		\item $\{u_{0m}\} \subset H^2(\mathbb{R}_+^n)$ with $\lim\limits_{m\to\infty} u_{0m} = u_0$ in $H^{s}(\mathbb{R}_+^n)$;
		\item $\{h_{0m}\} \subset \mathcal{H}^2(\mathbb{R}^{n}_{+})$ with $\lim\limits_{m\to\infty} h_{0m} = h_0$ in $\mathcal{H}^s(\mathbb{R}^{n}_{+})$,
	\end{enumerate}
	where $u_{0m}$ and $h_{0m}$ satisfy the compatibility condition. These sequences are justified as follows: If $s\in(0,\frac{1}{2}]$, this follows from a density argument by taking $u_{0m}\in H_{0}^{2}(\mathbb{R}^{n}_{+})$ and $h_{0m}\in\mathcal{H}_{0}^{2}(\mathbb{R}^{n}_{+})$. If $s\in(\frac{1}{2},1)$, for $h_{0}\in\mathcal{H}^{s}(\mathbb{R}^{n}_{+})$, we first take a sequence $\{h_{0m}\}\subset \mathcal{H}^{2}(\mathbb{R}^{n}_{+})$ such that $h_{0m}\rightarrow h_{0}$ in $\mathcal{H}^{s}(\mathbb{R}^{n}_{+})$. By the trace identity of $\mathcal{H}^{s}(\mathbb{R}^{n}_{+})$, this implies $h_{0m}(\cdot,0)\rightarrow h_{0}(\cdot,0)$ in $H^{s-\frac{1}{2}}(\mathbb{R}^{n-1})$. Then, utilizing the inverse trace theorem, we choose $\{u_{0m}\}\subset  H^{2}(\mathbb{R}^{n}_{+})$ with $u_{0m}(\cdot,0)=h_{0m}(\cdot,0)$ and
	\begin{equation} \|u_{0m}-u_{0}\|_{H^{s}(\mathbb{R}^{n}_{+})}\lesssim \|  h_{0m}(\cdot,0)\rightarrow h_{0}(\cdot,0)\|_{H^{s-\frac{1}{2}}(\mathbb{R}^{n-1})}\lesssim \|h_{0m}-h_{0}\|_{\mathcal{H}^{s}(\mathbb{R}^{n}_{+})}\rightarrow0.
	\end{equation}
	For $u_{0m}$ and $h_{0m}$, Proposition \ref{ag} shows the existence of $u_{m}\in\mathscr{X}^{2,p}_{\infty}$. Combining  Proposition~\ref{pragmatic}  with continuous dependence leads to
	\begin{equation}
		\|u(t)\|_{L^{2}}\leq (1+t) A\left(\|u_{0}\|_{H^{s}}+\|h_{0}\|_{\mathcal{H}^{s}}\right) e^{CtA\left(\|u_{0}\|_{H^s}+\|h_{0}\|_{\mathcal{H}^{s}}\right)},
	\end{equation}
	which yields $u\in\mathscr{X}^{0,p}_{\infty}$. The case $s\in(0,1)$ then follows from  Proposition~\ref{laudatory}.
\end{proof}

\section*{Acknowledgements}
\addcontentsline{toc}{section}{Acknowledgements}
This work is supported by the National Natural Science Foundation of China (Grant No. 12271104).

\end{document}